\documentclass[11pt]{amsart}
\usepackage[T1]{fontenc}
\usepackage{lmodern}
\usepackage{microtype}
\usepackage{amsmath,amssymb,mathtools,mathrsfs}
\usepackage{enumitem}
\usepackage{booktabs}
\usepackage{array}
\usepackage{needspace}
\usepackage{longtable}
\usepackage{xcolor}
\usepackage[colorlinks=true,linkcolor=blue,citecolor=blue,urlcolor=blue]{hyperref}
\usepackage[nameinlink,noabbrev]{cleveref}
\usepackage{geometry}
\allowdisplaybreaks
\newtheorem{theorem}{Theorem}[section]
\newtheorem{proposition}[theorem]{Proposition}
\newtheorem{corollary}[theorem]{Corollary}
\newtheorem{lemma}[theorem]{Lemma}
\newtheorem{conjecture}[theorem]{Conjecture}
\theoremstyle{definition}

\newtheorem{remark}[theorem]{Remark}

\newcommand{\RR}{\mathbb R}
\newcommand{\CC}{\mathbb C}
\newcommand{\ZZ}{\mathbb Z}
\newcommand{\QQ}{\mathbb Q}
\newcommand{\dd}{\,\mathrm d}
\newcommand{\Ree}{\operatorname{Re}}
\newcommand{\Imm}{\operatorname{Im}}

\title[A genus-two spectral curve of the Riemann $\Xi$-function]
{A Genus-Two Spectral Curve of the Riemann\\
$\Xi$-Function: Real Structure and Certified\\
Arithmetic Separation of Periods}
\author{Michel Planat}
\thanks{Institut FEMTO-ST, CNRS UMR 6174, Universit\'e Marie et Louis Pasteur, Besan\c{c}on, France. Email: \texttt{michel.planat@femto-st.fr}.}
\date{23 October 2026}

\begin{document}
\begin{abstract}
We associate a compact genus-two real spectral core with a transverse logarithmic-derivative discriminant of the Riemann $\Xi$-function. Directed interval arithmetic certifies its branch geometry, a real $P_{\mathrm{III}}(D_6)$ scalar-normal-form realization, and associated Kummer and Shioda--Inose models. Zero-action resonance reduces to rationality of two real period ratios. Their certified heat motion yields generic local nonresonance and Diophantine separation, while explicit Farey windows impose denominator barriers of $1{,}036{,}003$ and $140{,}769$ throughout $|\kappa|\le3\times10^{-10}$ in the heat time $\kappa$. Fixed-point irrationality and exact Picard rank remain open. The heat tangent differs from the matched fixed-exponent isomonodromic direction, but admits a local moving-exponent curve lift. Separately, a reciprocal-head canopy certificate combined with published finite-RH and rectangular time-barrier verifications bounds the de Bruijn--Newman constant by $\Lambda\le433/2500=0.1732$ without a QRH assumption. These local arithmetic and global analytic results share a phase--transverse bridge; no equivalence between core nonresonance and RH is asserted.
\end{abstract}

\maketitle

\section{Introduction}

Let
\[
 \Xi(z)=\xi\!\left(\frac12+iz\right)
\]
be the Riemann $\Xi$-function, where
$\xi(s)=\tfrac12s(s-1)\pi^{-s/2}\Gamma(s/2)\zeta(s)$ is the completed zeta function,
$\zeta$ is the Riemann zeta function and $\Gamma$ is Euler's gamma function.  Classical approaches to the Riemann hypothesis are organized around the zeros, growth, positivity, and Fourier representation of $\Xi$; see, for example, \cite{Edwards1974,Lagarias1999,Csordas2015}.  Here we isolate a different object: a spectral double cover naturally attached to the transverse product
\begin{equation}\label{eq:tau-def-intro}
 \tau_x(T)=\frac{\Xi(x+i\sqrt T)\Xi(x-i\sqrt T)}{\Xi(x)^2}.
\end{equation}
The square root disappears from the product, so $\tau_x$ is locally analytic in $T$ wherever $\Xi(x)\ne0$.  The construction arose from an RH-motivated $P_{\mathrm{III}}(D_6)$ and Stieltjes program \cite{PlanatD6GFJ2026,PlanatD6Jensen2026,Planat2026}, but the present paper has a deliberately narrower objective: to determine, with certification, the geometry and arithmetic of the first compact genus-two core at the arithmetic fixed point $x=0$.

The fixed-point discriminant admits a regular theta-moment numerator.  Put
\[
 D(z)=\frac14\Xi(z)=\int_0^\infty \Phi(u)\cos(zu)\,\dd u,
\qquad
 Y(T)=D(i\sqrt T)=\sum_{n\ge0}\frac{M_{2n}}{(2n)!}T^n,
\]
where $\Phi$ is the standard even theta kernel.  The normalized spectral discriminant takes the form
\[
 \Delta_0(T)=\frac{K(T)}{Y(T)^6},
\]
with $K$ analytic through $T=0$.  This regularization converts branch-point location into the certified study of the zeros of an explicitly controlled entire function.

The first result is geometric.  Directed theta-moment and Rouch\'e/Krawczyk bounds isolate exactly six simple branch values inside a conjugation-invariant domain: two real values and two nonreal conjugate pairs.  The normalized double cover over this domain is canonically identified with the restriction of a compact hyperelliptic sextic of genus two and real topological type
\[
 \boxed{(g,k,a)=(2,1,1)}.
\]
Here $g$ is the genus, $k$ the number of connected components of the real locus, and $a=0$ or $1$ according as that locus divides the complex curve into two components or does not.  The same branch sextic has two additional geometric interpretations.  First, it admits an orientation-differential realization with the $(2,4,4)$ pole signature associated with $P_{\mathrm{III}}(D_6)$.  Second, and more substantively, it is certified to coincide, after a real affine change of spectral coordinate, with the sextic numerator arising from the scalar normal form of a real $D_6$ Lax phase point.  We call the latter hyperelliptic double the \emph{scalar-normal-form orientation curve}; it should not be confused with the invariant Hitchin characteristic spectral curve.  The Jacobian is the abelian variety of degree-zero divisor classes of the curve.  Its quotient by the sign involution, followed by resolution of the sixteen nodes, is the associated Kummer K3 surface.  A K3 surface is a smooth simply connected compact complex surface with trivial canonical bundle.  Explicit Kummer and Shioda--Inose models then make the Humbert and N\'eron--Severi questions computationally accessible.

The second result is arithmetic.  The global Humbert analysis is most naturally written in a canonical Vinnikov basis with $H=I_2$, whereas the certified action integrals use the natural three-cut basis described below.  In that cut basis $(\alpha_1,\alpha_2,\beta_1,\beta_2)$ the real involution has matrix
\[
 \tau_*\alpha=\alpha,
 \qquad
 \tau_*\beta=H_{\rm cut}\alpha-\beta,
 \qquad
 H_{\rm cut}=\begin{pmatrix}1&1\\1&0\end{pmatrix}.
\]
For the action differential $\varpi_0=\sqrt{\Delta_0(T)}\,\dd T$, write
\[
 u_j=\oint_{\alpha_j}\varpi_0,
 \qquad
 v_j=\oint_{\beta_j}\varpi_0,
 \qquad
 v=\frac12H_{\rm cut}u+i r.
\]
The four calibrated cut cycles lie inside the certified normalized core and form a
symplectic rank-four lattice $\Gamma_{\rm core}$ (its intersection pairing is integral, alternating and nondegenerate), abstractly identified with
$H_1(X_2,\ZZ)$.  The action homomorphism is evaluated on these actual core
representatives, and we define
\[
 \Lambda_2=\ker\!\left(
  \mathcal A_{\rm core}:\Gamma_{\rm core}\to\CC,
  \ \gamma\mapsto\oint_\gamma\varpi_0\right).
\]
This calibrated zero-action lattice is trivial exactly when both real ratios
$u_1/u_2$ and $r_1/r_2$ are irrational.  We do not prove those irrationalities.  Instead, fully directed contour integration and exact Farey-neighbor arithmetic give
\[
 \frac{13749}{419336}<\frac{u_1}{u_2}<\frac{20219}{616667},
\qquad
 \frac{16579}{8678}<-\frac{r_1}{r_2}<\frac{252355}{132091},
\]
and therefore the denominator barriers
\[
 \boxed{q_u\ge1{,}036{,}003,\qquad q_r\ge140{,}769}
\]
for any rational value in lowest terms.  This is a finite-height nonresonance theorem, not an irrationality proof.

This motivates the structural conjecture
\begin{conjecture}[Spectral nonresonance at the arithmetic fixed point]\label{conj:nonresonance}
For the calibrated $x=0$ genus-two core,
\[
 \frac{u_1}{u_2}\notin\QQ,
 \qquad
 \frac{r_1}{r_2}\notin\QQ.
\]
Equivalently, $\Lambda_2=0$.
\end{conjecture}
The conjecture is deliberately separated from the certified Farey theorem: finite-width interval enclosures can exclude rational values only up to finite denominator height.

The Kummer/Humbert analysis provides a complementary measure of arithmetic genericity.  A Humbert surface $H_D$ is the locus of principally polarized abelian surfaces admitting an additional non-scalar endomorphism symmetric for the polarization, of Humbert discriminant $D$.  The N\'eron--Severi group consists of divisor classes modulo algebraic equivalence; its rank is the Picard number.  We exclude the nine loci
\[
 H_4,H_5,H_8,H_9,H_{12},H_{13},H_{16},H_{17},H_{20},
\]
obtain a Mordell--Weil height barrier $\widehat h(P)\ge21/2$ for any hypothetical nonzero section on the Shioda--Inose model, and reduce the full real-slice Humbert problem to two exact arithmetic statements about the imaginary part of the Vinnikov period matrix: irrationality of its determinant and $\QQ$-linear independence of its three entries.  Here the Mordell--Weil group is the group of sections of the elliptic fibration, and $\widehat h$ is its canonical quadratic height in Shioda's normalization.  The union of Humbert divisors is dense on this real slice, so finite neighborhood exclusion cannot settle the global Picard-rank question.

The heat time is denoted by $\kappa$; the Polymath time is $t=4\kappa$, as fixed in \eqref{eq:heat-normalization}.  The de Bruijn--Newman constant $\Lambda$ is the threshold for reality of all zeros of $H_t$: they are all real exactly when $t\ge\Lambda$.  This real number is distinct from the zero-action lattice $\Lambda_2$.  The heat deformation connects the geometry to the arithmetic.  It transports the selected branch points and calibrated contours, so the same theta-derived data determine both the motion of the curve and the motion of its action periods.  The resulting positive ratio velocities make rational crossings transverse: each fixed nonzero cycle can resonate at most once locally, and almost every nearby heat time has quantitative action separation.  Uniform bounds also preserve the displayed Farey windows on an explicit heat interval.  These generic statements do not settle the irrationality conjecture at the distinguished time $\kappa=0$.

The phase--transverse identity supplies a second connection, from the same discriminant to global zero control.  A uniform off-axis phase sign gives a zero strip, which de Bruijn contraction converts into an upper bound for the Newman constant.  Effective approximation, exterior estimates and a finite rectangular nonvanishing cover implement this route.  The final bound $\Lambda\le0.1732$ combines that cover with the published finite-RH and rectangular time-barrier verifications.  It is logically independent of the Humbert exclusions and the period-irrationality conjecture.  The companion studies \cite{PlanatNewmanCompanion2026,PlanatNewmanClassical2026} develop separate collective-contraction and classical-strip comparisons, with their own stated premises.

The distinction between local arithmetic and global zero reality is essential.  Almost every sufficiently small negative heat time has a nonresonant core and nonreal heat zeros simultaneously.  Thus generic core nonresonance alone cannot characterize reality of the full heat family.  At $\kappa=0$, no necessity or sufficiency of $\Lambda_2=0$ for RH is known.

The paper is organized as follows.  Section~\ref{sec:theta} derives the regular theta-moment discriminant and its phase interpretation.  Section~\ref{sec:geometry} constructs the real genus-two core, its $D_6$ realization and K3 models, and isolates the Humbert/Picard question.  Section~\ref{sec:heat-deformation} studies the moving branch geometry and its relation to isomonodromic directions.  Sections~\ref{sec:period-lattice} and~\ref{sec:calibration} identify and differentiate the calibrated action lattice.  Section~\ref{sec:farey} gives finite-height arithmetic certificates and the separate global Newman-bound argument.  Section~\ref{sec:certification} explains the numerical architecture; Section~\ref{sec:discussion} returns to the unresolved arithmetic questions.  Table~\ref{tab:certificate} locates the principal results in this order.

\Needspace{38\baselineskip}
\begingroup\small
\setlength{\tabcolsep}{4pt}
\renewcommand{\arraystretch}{1.12}
\begin{longtable}{@{}>{\raggedright\arraybackslash}p{0.18\textwidth}>{\raggedright\arraybackslash}p{0.49\textwidth}>{\raggedright\arraybackslash}p{0.28\textwidth}@{}}
\caption{Guide to the principal results and open questions.}\label{tab:certificate}\\
\toprule
Stage & Meaning & Main statements\\
\midrule
\endfirsthead
\multicolumn{3}{l}{\textit{Table~\thetable\ continued}}\\
\toprule Stage & Meaning & Main statements\\\midrule
\endhead
\bottomrule\endfoot
Transverse phase & An exact phase identity and regular numerator make the branch problem accessible to theta-moment bounds. & \Cref{prop:phase-transverse,prop:regular-K}\\
Real core & Six selected simple branch values determine a genus-two core of real type $(2,1,1)$. & \Cref{thm:real-pair,thm:c1-pair,thm:c2-pair,cor:genus2,prop:real-type}\\
$D_6$ realization & The branch sextic matches a real scalar-normal-form orientation curve; a regular fixed-exponent deformation is obstructed. & \Cref{thm:D6-certified-realization,thm:D6-first-order-no-go}\\
K3 arithmetic & Nine Humbert exclusions and a section-height barrier constrain enhancements; a real-slice criterion identifies the remaining Picard problem. & \Cref{thm:low-Humbert,prop:v7-MW-height-barrier,thm:v8-global-Humbert}; \cref{conj:v9-minimal-Picard}\\
Heat geometry & All fifteen branch distances contract locally.  Enlarged tangent rank permits a moving-exponent curve lift, while the matched isomonodromic direction alone fails. & \Cref{thm:heat-branch-contraction,thm:heat-D6-obstruction,prop:enlarged-Lax-rank,cor:local-Lax-lift}\\
Action kernel & Zero-action resonance is equivalent to rationality of at least one of two real ratios; irrationality at zero is conjectural. & \Cref{thm:kernel-factorization,cor:irrationality-criterion}; \cref{conj:nonresonance}\\
Period motion & Positive velocities give simple resonances, generic Diophantine separation and nonresonant cores even with nonreal heat zeros. & \Cref{thm:heat-action-motion,prop:heat-simple-resonances,thm:heat-generic-Diophantine,cor:heat-nonresonance-nonreal-zeros}\\
Finite-height arithmetic & Farey windows force $q_u\ge1{,}036{,}003$ and $q_r\ge140{,}769$, uniformly for $|\kappa|\le3\times10^{-10}$; residence time measures an arithmetic gap. & \Cref{thm:farey,cor:denominator-barriers,thm:explicit-Farey-heat,prop:Farey-residence-time}\\
Global zero control & A phase criterion and exterior estimates reduce the Newman bound to a finite cover plus published barrier inputs. & \Cref{prop:Newman-strip-input,prop:global-phase-strip,thm:centered-phase-tail,cor:finite-phase-cap}\\
Certified ceiling & Reciprocal-head localization culminates in $\Lambda\le433/2500=0.1732$, without QRH. & \Cref{thm:inverse-canopy,thm:localized-canopy,cor:fine-canopy,cor:extended-head-canopy}\\
\end{longtable}
\endgroup

For clarity, the status labels used in a few intermediate statements have the following meaning: status \textup{[A]} denotes an analytic statement or a computer-assisted statement backed by a directed/exact certificate in the supplement, whereas status \textup{[N]} denotes a numerical observation not used as theorem input.

\section{Theta moments and the fixed-point discriminant}\label{sec:theta}

We use the classical theta kernel
\begin{equation}\label{eq:Phi}
 \Phi(u)=\sum_{n=1}^\infty
 \left(2\pi^2n^4e^{9u/2}-3\pi n^2e^{5u/2}\right)
 e^{-\pi n^2e^{2u}}.
\end{equation}
The Jacobi transformation makes $\Phi$ even; it is real analytic, strictly positive, and doubly exponentially decaying together with all derivatives.  With
\begin{equation}\label{eq:D}
 D(z)=\int_0^\infty\Phi(u)\cos(zu)\,\dd u,
 \qquad \Xi(z)=4D(z),
\end{equation}
put
\begin{equation}\label{eq:moments-Y}
 M_m=\int_0^\infty\Phi(u)u^m\,\dd u,
 \qquad
 Y(T)=D(i\sqrt T)
 =\sum_{n\ge0}\frac{M_{2n}}{(2n)!}T^n.
\end{equation}
Thus $Y$ is entire and
\[
 4Y(T)=\xi\!\left(\frac12+\sqrt T\right).
\]

For real $x$ with $\Xi(x)\ne0$, define
\begin{equation}\label{eq:tau-sigma}
 \tau_x(T)=\frac{\Xi(x+i\sqrt T)\Xi(x-i\sqrt T)}{\Xi(x)^2},
 \qquad
 \mathscr A_x(T)=\frac{\dd}{\dd T}\log\tau_x(T),
 \qquad
 \sigma_x(T)=T\mathscr A_x(T).
\end{equation}
At $x=0$, evenness gives
\begin{equation}\label{eq:A0-Y}
 \mathscr A_0(T)=2\frac{Y'(T)}{Y(T)}.
\end{equation}

\subsection{A phase--transverse-amplitude interpretation}\label{subsec:phase-transverse}
The logarithmic derivative in \eqref{eq:tau-sigma} has a direct phase
interpretation that is useful independently of the later branch-point
construction.  For $T>0$ at which
$\Xi(x\pm i\sqrt T)\ne0$, define the unnormalized transverse amplitude
\begin{equation}\label{eq:B-static}
 \mathscr B_x(T):=
 \partial_T\log\!\left\{\Xi(x+i\sqrt T)\Xi(x-i\sqrt T)\right\}.
\end{equation}
For real $x$ the product in braces is $|\Xi(x+i\sqrt T)|^2>0$.
Whenever $\Xi(x)\ne0$, the normalization in \eqref{eq:tau-sigma} is
independent of $T$, so $\mathscr B_x=\mathscr A_x$.

\begin{proposition}[Exact phase--transverse-amplitude bridge]
\label{prop:phase-transverse}
Fix $\eta>0$ and a real $x$ for which $\Xi(x\pm i\eta)\ne0$.
Choose a continuous local branch
\[
 P_\eta(s):=\arg\Xi(s-i\eta)
\]
near $s=x$.  Then
\begin{equation}\label{eq:phase-transverse-bridge}
 \boxed{
 \left.\frac{\dd}{\dd s}P_\eta(s)\right|_{s=x}
 =-\Imm\frac{\Xi'}{\Xi}(x+i\eta)
 =\eta\,\mathscr B_x(\eta^2).}
\end{equation}
If also $\Xi(x)\ne0$, the last term is
$\eta\mathscr A_x(\eta^2)$.  In particular, at the arithmetic fixed point,
\begin{equation}\label{eq:phase-fixed-point}
 \boxed{
 \left.\frac{\dd}{\dd s}\arg\Xi(s-i\eta)\right|_{s=0}
 =2\eta\,\frac{Y'(\eta^2)}{Y(\eta^2)}.}
\end{equation}
\end{proposition}

\begin{proof}
Reality of $\Xi$ on the real axis gives
$\Xi'(x-i\eta)/\Xi(x-i\eta)=\overline{q}$, where
$q=\Xi'(x+i\eta)/\Xi(x+i\eta)$.  Therefore
\[
 \left.\frac{\dd}{\dd s}\arg\Xi(s-i\eta)\right|_{s=x}
 =\Imm\overline q=-\Imm q.
\]
On the other hand, differentiating \eqref{eq:B-static} at $T=\eta^2$
gives
\[
 \mathscr B_x(\eta^2)
 =\frac{i q-i\overline q}{2\eta}
 =-\frac{\Imm q}{\eta}.
\]
This proves \eqref{eq:phase-transverse-bridge};
\eqref{eq:phase-fixed-point} follows from \eqref{eq:A0-Y}.
\end{proof}

\begin{remark}[Relation to the heat-flow companion]
\label{rem:phase-companion}
The identity above is static and requires neither RH nor any Newman
barrier estimate.  Under the de Bruijn--Newman deformation the same
calculation identifies the off-axis logarithmic-derivative probe used in
the collective-contraction argument with the heat-evolved transverse
amplitude.  Section~\ref{sec:heat-deformation} makes this local evolution
explicit.  Global localisation, source transfer, and contact arguments
remain in the companion heat-flow studies
\cite{PlanatNewmanCompanion2026,PlanatNewmanClassical2026}; they are not
used in the fixed-point genus-two, Humbert, or period arguments here.
\end{remark}
Let $A_*=\mathscr A_0(0)=M_2/M_0$ and set
\[
 P(T)=\sigma_0'(T)=\mathscr A_0(T)+T\mathscr A_0'(T).
\]
The normalized discriminant used here is
\begin{equation}\label{eq:Delta-general}
 \Delta_0(T)=
 \{\sigma_0''(T)\}^2
 -4\left\{\frac{\sigma_0'(T)-A_*}{T}\right\}^2
 \{P(T)-T^2\mathscr A_0'(T)\}.
\end{equation}
The quotient in the middle is analytic at $T=0$.

To remove all apparent singular behavior, write $Y_j=Y^{(j)}$ and
\begin{align}
 N_2&=Y_2Y-Y_1^2,\notag
 N_3&=Y_3Y^2-3Y_2Y_1Y+2Y_1^3,
\end{align}
and define
\begin{align}
 \mathcal S&=4N_2Y+2TN_3,\notag\\
 \mathcal E&=\frac{2Y_1Y+2TN_2-A_*Y^2}{T},\notag\\
 \mathcal R&=2Y_1Y+2(T-T^2)N_2.
\end{align}
The numerator defining $\mathcal E$ vanishes at $T=0$.

\begin{proposition}[Regular entire numerator]\label{prop:regular-K}
On every domain on which $Y\ne0$,
\begin{equation}\label{eq:Delta-K}
 \boxed{
 \Delta_0(T)=\frac{K(T)}{Y(T)^6},
 \qquad
 K(T)=\mathcal S(T)^2-4\mathcal E(T)^2\mathcal R(T).}
\end{equation}
In particular, $K$ is analytic through $T=0$.
\end{proposition}

\begin{proof}
From \eqref{eq:A0-Y},
\[
 \mathscr A_0'=2\frac{N_2}{Y^2},
 \qquad
 \mathscr A_0''=2\frac{N_3}{Y^3}.
\]
A direct substitution gives
\[
 P'=\frac{\mathcal S}{Y^3},
 \qquad
 \frac{P-A_*}{T}=\frac{\mathcal E}{Y^2},
 \qquad
 P-T^2\mathscr A_0'=\frac{\mathcal R}{Y^2},
\]
and \eqref{eq:Delta-K} follows from \eqref{eq:Delta-general}.
\end{proof}

The spectral curve is the normalization of
\begin{equation}\label{eq:spectral-curve}
 \Sigma_0:\qquad y^2=\Delta_0(T),
\end{equation}
with action differential
\begin{equation}\label{eq:action-form}
 \varpi_0=y\,\dd T=\sqrt{\Delta_0(T)}\,\dd T.
\end{equation}
Every simple zero of $K$ away from $Y=0$ is a simple ramification point of \eqref{eq:spectral-curve}.

\section{Certified six-point geometry}\label{sec:geometry}

All branch-point statements in this section are directed interval results.  Directed interval arithmetic encloses exact quantities with outward-rounded bounds.  Rouch\'e tests count zeros by comparing an analytic function with a dominant term on a contour; Krawczyk tests use an interval Newton-type inclusion to certify existence and uniqueness of a solution.  The computations use the theta representation \eqref{eq:Phi}, moment intervals from \eqref{eq:moments-Y}, analytic tails, and Rouch\'e tests for the complex zeros.  Details of the certification architecture are summarized in \cref{sec:certification}.

\begin{theorem}[Real branch pair]\label{thm:real-pair}
There are unique simple real zeros $a,b$ of $K$ in
\begin{equation}\label{eq:real-boxes}
 \boxed{
 -65.85<a<-65.70,
 \qquad
 46.3829<b<46.3830.}
\end{equation}
Moreover $Y$ is nonzero on the connecting real corridor and $\Delta_0(T)>0$ for $a<T<b$.
\end{theorem}

\begin{proof}
The directed moment certificate gives
\[
 K(-65.85)<-2.72\times10^{-20},\qquad
 K(-65.70)>4.31\times10^{-20},
\]
with $K'>7.03\times10^{-19}$ throughout the left interval.  On the right interval it gives
\[
 K(46.3829)>5.5399\times10^{-17},\qquad
 K(46.3830)<-1.3335\times10^{-18},
\]
with $K'<-5.6734\times10^{-13}$.  Hence each interval contains exactly one simple zero.  Directed polynomial evaluation gives $Y>0$ on a covering of the whole corridor and $K>0$ on a finite rational-cell covering of the interior, with monotonicity handling the two endpoint boxes.  Equation~\eqref{eq:Delta-K} then yields $\Delta_0>0$ on $(a,b)$.
\end{proof}

\begin{theorem}[First nonreal branch pair]\label{thm:c1-pair}
Let
\[
 \widehat c_1=-231.0846402275628298879620029048
 +59.4949465030781144742436618558\,i.
\]
Then $K$ has exactly one simple zero $c_1$ in
\begin{equation}\label{eq:c1-disk}
 \boxed{|T-\widehat c_1|<10^{-9}},
\end{equation}
and exactly one simple zero $\bar c_1$ in the conjugate disk.  The function
$Y$ is nonzero on both disks.
\end{theorem}

\begin{proof}
For numerical stability the directed certificate applies Rouch\'e's theorem to
$Q=T^2K$.  On $|T-\widehat c_1|=10^{-9}$, the finite Taylor remainder and the
analytic entire-series tail give
\[
 |Q(\widehat c_1)|+R_{\mathrm{nl}}+R_{\mathrm{tail}}
 <5.352016461\times10^{-32}
 <9.054459459\times10^{-32}
 \le |Q'(\widehat c_1)|10^{-9}.
\]
It also gives $|Y|>4.356747335\times10^{-4}$ throughout the disk.  Hence the
disk contains exactly one zero of $Q$, counted with multiplicity; because it
is away from $T=0$, this is exactly one simple zero of $K$.  Conjugation gives
the lower-half-plane zero.
\end{proof}

\begin{theorem}[Second nonreal branch pair]\label{thm:c2-pair}
Let
\[
 \widehat c_2=-269.8852991848204115111948309919
 +150.2255429937459843575350632811\,i.
\]
Then $K$ has exactly one simple zero $c_2$ in
\begin{equation}\label{eq:c2-disk}
 \boxed{|T-\widehat c_2|<10^{-9}},
\end{equation}
and exactly one simple zero $\bar c_2$ in the conjugate disk.  The function
$Y$ is nonzero on both disks.
\end{theorem}

\begin{proof}
Again use the regularized numerator $Q=T^2K$.  On
$|T-\widehat c_2|=10^{-9}$ the outward bounds give
\[
 |Q(\widehat c_2)|+R_{\mathrm{nl}}+R_{\mathrm{tail}}
 <1.603287091\times10^{-32}
 <1.746001014\times10^{-32}
 \le |Q'(\widehat c_2)|10^{-9}.
\]
The same disk computation gives $|Y|>6.323600546\times10^{-4}$.  Rouch\'e's
theorem therefore gives exactly one simple zero of $K$ in the disk and one in
its conjugate.
\end{proof}

\begin{remark}
The theorem does not assert that these six points are the only branch points of
the full noncompact spectral curve.  It asserts that they form a certified
finite core.  The zero sets of $K$ and $Y$ are discrete, and the certificates
above show that $Y$ does not vanish at any of the six selected roots.  Hence one
may choose a bounded conjugation-invariant Jordan domain $\Omega_2$ containing
precisely these six selected zeros of $K$, no other zero of $K$, and no zero of
$Y$.  We take $\Omega_2$ simply connected with boundary avoiding all these
loci.
\end{remark}

\begin{corollary}[Genus-two core]\label{cor:genus2}
Let $\Omega_2$ be such a Jordan domain.  The normalization of the two-sheeted cover \eqref{eq:spectral-curve} over $\Omega_2$ has genus two and two boundary components.
\end{corollary}

\begin{proof}
Riemann--Hurwitz gives
\[
 \chi(\widetilde\Omega_2)=2\chi(\Omega_2)-6=-4.
\]
Because the branch number is even, the boundary of $\Omega_2$ has two lifts.  Thus $-4=2-2g-2$, hence $g=2$.
\end{proof}

The six selected roots define the monic real sextic
\begin{equation}\label{eq:PX-definition}
 P_X(T)=(T-a)(T-b)(T-c_1)(T-\bar c_1)(T-c_2)(T-\bar c_2).
\end{equation}
On $\Omega_2$ write $K=H P_X$.  The quotient $H$ extends holomorphically and
never vanishes there, because the six selected zeros are simple and there are
no other zeros of $K$ in $\Omega_2$.  Since $\Omega_2$ is simply connected,
choose $S$ with $S^2=H$.  At the level of function fields, the substitution
$U=yY^3$ turns $y^2=K/Y^6$ into $U^2=K$; after normalization it therefore
extends across any removable points of the original affine presentation.  The
further change $W_X=U/S$ identifies the normalized core with the restriction of
\begin{equation}\label{eq:Xi-sextic-cover}
 \boxed{X_2:\quad W_X^2=P_X(T)}
\end{equation}
to $\Omega_2$.  We henceforth define $X_2$ to be the smooth projective
completion of \eqref{eq:Xi-sextic-cover}.  This gives an intrinsic compact model
for the certified core and does not assert that the full noncompact spectral
curve has no further branch points.

\begin{remark}[Scope of the action differential]\label{rem:action-scope}
The algebraic compactification $X_2$ and the original action differential play
different roles.  The biholomorphism above identifies $X_2$ with the normalized
spectral cover only over $\Omega_2$, and $\varpi_0=\sqrt{\Delta_0(T)}\,\dd T$
is canonically defined on that normalized core.  We do \emph{not} assume that
$\varpi_0$ extends meromorphically across the artificial projective completion
of $X_2$.  All action periods below are therefore evaluated on the explicit
cut and connector cycles lying inside the certified core.  Their four-cycle
lattice is symplectically isomorphic to $H_1(X_2,\ZZ)$, but the action
functional is attached to these calibrated core representatives rather than
to an unproved global extension of $\varpi_0$.
\end{remark}

Complex conjugation acts on this model by
\begin{equation}\label{eq:real-involution}
 \tau(T,W_X)=(\bar T,\bar W_X).
\end{equation}

\begin{proposition}[Real topological type]\label{prop:real-type}
The compact real surface $(X_2,\tau)$ has topological type
\[
 \boxed{(g,k,a)=(2,1,1)}.
\]
\end{proposition}

\begin{proof}
The smooth projective completion of a sextic double cover with six simple
branch values has genus two.  For real $T$, the two conjugate quadratic
factors in \eqref{eq:PX-definition} are strictly positive, so the sign of
$P_X(T)$ is the sign of $(T-a)(T-b)$.  Thus the compact real locus consists of
one oval, hence $k=1$.  To determine the dividing invariant, use the
hyperelliptic analysis of Gross--Harris \cite[\S6, proof of Prop.~6.3]{GrossHarris1981}:
if a real hyperelliptic curve is dividing, then in a real even-degree model its
branch roots are either all real or all nonreal.  Here the branch configuration
is mixed (exactly two real roots and two conjugate pairs), so $X_2$ is
nondividing and $a=1$.
\end{proof}

\begin{corollary}[Refined real boxes used in the Lax realization]\label{cor:refined-branch-boxes}
The real branch points satisfy
\begin{equation}\label{eq:refined-real-branches}
 \boxed{
 -65.78507454<a<-65.78507452,
 \qquad
 46.38299764<b<46.38299766.}
\end{equation}
Together with the $10^{-9}$ complex disks in
\cref{thm:c1-pair,thm:c2-pair}, these are the branch boxes used in the certified
$D_6$ coefficient match.
\end{corollary}

\begin{proof}
The coarse real-root theorem gives exactly one simple zero in each of
$[-65.85,-65.70]$ and $[46.3829,46.3830]$.  The refined $M_{80}$ moment
certificate gives opposite directed signs of $\Delta_0$ at the two endpoints
of each interval in \eqref{eq:refined-real-branches}; hence the unique coarse
roots lie in the displayed subintervals.
\end{proof}

\subsection{A \texorpdfstring{$P_{\mathrm{III}}(D_6)$}{PIII(D6)} orientation-double bridge}\label{subsec:D6-bridge}

The genus-two value also has an independent interpretation from the decorated
$P_{\mathrm{III}}(D_6)$ geometry.  Chekhov--Mazzocco--Rubtsov associate to
$P_{\mathrm{III}}(D_6)$ the cusp/Katz signature
\[
 (0,2,2),\qquad (0,1,1),
\]
and, on the Hitchin side, the pole-order pattern
\begin{equation}\label{eq:D6-pole-orders}
 (2,4,4)
\end{equation}
for the corresponding meromorphic quadratic differential $\phi$; see
Appendix~A, Table~2 of \cite{ChekhovMazzoccoRubtsov2017}.  The following elementary observation gives
a more precise bridge to the present six-branch-point core than a Schottky-double
analogy.

\begin{proposition}[$D_6$ orientation double and the six-point core]\label{prop:D6-orientation-double}
Let $\phi$ be a meromorphic quadratic differential on $\mathbb P^1$ with pole
orders $(2,4,4)$, and suppose its zeros are simple.  Then $\phi$ has exactly six
zeros, and its canonical orientation double
\begin{equation}\label{eq:D6-orientation-cover}
 \pi_{D_6}:C_{D_6}\longrightarrow\mathbb P^1,
 \qquad \lambda^2=\phi,
\end{equation}
is a genus-two hyperelliptic curve.  After a M\"obius normalization placing the
two fourth-order poles at $0,1$ and the second-order pole at $\infty$, one may
write
\begin{equation}\label{eq:D6-sextic-model}
 \phi(z)=C\,\frac{P_{D_6}(z)}{z^4(z-1)^4}\,\dd z^2,
 \qquad
 P_{D_6}(z)=\prod_{j=1}^6(z-z_j),
\end{equation}
and the orientation double has the affine hyperelliptic model
\begin{equation}\label{eq:D6-sextic-cover}
 \boxed{W^2=C P_{D_6}(z)}.
\end{equation}

Let
\[
 B_X=\{a,b,c_1,\bar c_1,c_2,\bar c_2\}
\]
be the branch locus of the compact model $X_2\to\mathbb P^1$ in
\eqref{eq:Xi-sextic-cover}.
For any labelling $Z(\phi)=\{z_1,\ldots,z_6\}$ and any orientation-preserving
homeomorphism
\[
 h:\mathbb P^1\longrightarrow\mathbb P^1,
 \qquad h(Z(\phi))=B_X,
\]
the map $h$ lifts to a homeomorphism
\begin{equation}\label{eq:D6-topological-lift}
 \boxed{\widetilde h:C_{D_6}\longrightarrow X_2}
\end{equation}
unique up to the hyperelliptic deck involution.  If the six points are paired
into three cuts and $h$ is chosen to carry those cuts to
$\Gamma_0,\Gamma_1,\Gamma_2$, the lift identifies the associated cut-cycle and
connector symplectic bases.
\end{proposition}

\begin{proof}
For a quadratic differential on $\mathbb P^1$,
$\deg\operatorname{div}(\phi)=\deg K_{\mathbb P^1}^{\otimes2}=-4$.
The three poles in \eqref{eq:D6-pole-orders} have total order $10$, hence the
total zero order is $10-4=6$.  Under the stated genericity assumption these are
six simple zeros.  The orientation double branches at the odd-order points of
$\operatorname{div}(\phi)$.  All three pole orders are even, so the branch locus
consists exactly of the six simple zeros.  Riemann--Hurwitz gives
\[
 2g(C_{D_6})-2=2(-2)+6=2,
\]
therefore $g(C_{D_6})=2$.

With the pole normalization in \eqref{eq:D6-sextic-model}, writing
$\lambda=W\,\dd z/[z^2(z-1)^2]$ gives \eqref{eq:D6-sextic-cover} directly.

By \eqref{eq:Xi-sextic-cover}, the compact Xi core is likewise a two-sheeted
cover of the sphere branched exactly at $B_X$.  On the complement of the branch
locus, both covers are classified by the
homomorphism to $\mathbb Z/2\mathbb Z$ that sends each small branch meridian to
the nontrivial element.  Consequently an orientation-preserving homeomorphism
carrying one labelled branch set to the other preserves the covering monodromy
and therefore lifts, uniquely up to deck transformation.  Choosing compatible
cut systems gives the final homological statement.
\end{proof}

\begin{corollary}[A $D_6$-signature orientation differential on the Xi core]\label{cor:explicit-D6-Xi}
Let
\[
 P_X(T)=(T-a)(T-b)(T-c_1)(T-\bar c_1)(T-c_2)(T-\bar c_2)
\]
be the certified Xi branch sextic, and choose distinct points
$p,q\in\mathbb P^1\setminus\bigl(B_X\cup\{\infty\}\bigr)$.
Then, for every $\kappa\in\mathbb C^\times$,
\begin{equation}\label{eq:Xi-D6-qdiff}
 \boxed{
 \phi_X^{D_6}(T)
 =\kappa\,\frac{P_X(T)}{(T-p)^4(T-q)^4}\,\dd T^2}
\end{equation}
has six simple zeros at $B_X$ and pole orders $(4,4,2)$ at
$p,q,\infty$.  Hence its cusp/Katz data are $(2,2,0)$ and $(1,1,0)$,
respectively, i.e. the Chekhov--Mazzocco--Rubtsov $D_6$ data
$(0,2,2)$ and $(0,1,1)$ up to ordering.  Its canonical orientation
double is biholomorphic to $X_2$.  In the sextic coordinate of
\eqref{eq:Xi-sextic-cover}, the identification is explicit:
\begin{equation}\label{eq:Xi-D6-lambda}
 \boxed{
 \lambda_X
 =\sqrt\kappa\,\frac{W_X}{(T-p)^2(T-q)^2}\,\dd T,
 \qquad
 \lambda_X^2=\phi_X^{D_6}.}
\end{equation}
If $p$ and $q$ are chosen outside $\overline{\Omega_2}$, then on the
spectral core the $D_6$ orientation differential and the Xi action
differential are related by the nowhere-vanishing holomorphic factor
\begin{equation}\label{eq:D6-Xi-differential-weight}
 \boxed{
 \lambda_X
 =G_{p,q}(T)\,\varpi_0,
 \qquad
 G_{p,q}(T)
 =\sqrt\kappa\,
 \frac{Y(T)^3}{S(T)(T-p)^2(T-q)^2},}
\end{equation}
where $S^2=H$ in the factorization $K=HP_X$ used in the construction of
\eqref{eq:Xi-sextic-cover}.  This pole-signature construction alone is not a
Lax realization; the certified van der Put--Saito phase-point realization is
proved in \cref{thm:D6-certified-realization}.
\end{corollary}

\begin{proof}
The finite poles in \eqref{eq:Xi-D6-qdiff} have order four because
$p,q\notin B_X$.  Since $P_X$ has degree six, its coefficient behaves as
$T^{-2}$ at infinity; after the coordinate change $T=1/w$ and
$\dd T^2=w^{-4}\dd w^2$, this is a pole of order two.  All poles have
even order and all six zeros have odd order one, so the orientation
double branches exactly at $B_X$.  Substituting
\eqref{eq:Xi-D6-lambda} gives $\lambda_X^2=\phi_X^{D_6}$ and leaves the
underlying sextic equation $W_X^2=P_X(T)$ unchanged.  This proves the
biholomorphic identification with $X_2$.  Finally,
$W_X=yY^3/S$ and $\varpi_0=y\,\dd T$, so
\eqref{eq:D6-Xi-differential-weight} follows directly.  If
$p,q\notin\overline{\Omega_2}$, every factor in $G_{p,q}$ is holomorphic
and nonzero on $\Omega_2$.
\end{proof}

A Lax connection is a matrix differential system whose compatible deformation encodes a Painlev\'e equation.  An isomonodromic deformation preserves its monodromy and Stokes data, the latter describing changes of asymptotic solution bases near irregular singularities.

\begin{proposition}[Scalar normal form of the $P_{\mathrm{III}}(D_6)$ Lax connection]\label{prop:D6-scalar-normal}
Use the $P_{\mathrm{III}}(D_6)$ rank-two connection of van der Put and Saito
\cite[\S4.5]{vanDerPutSaito2009}, written as
\[
 \nabla_{\dd/\dd z}=\frac{\dd}{\dd z}+\frac{1}{z^2}A(z),
 \qquad
 A(z)=\begin{pmatrix}A_{11}(z)&L(z)\\ z-q&-A_{11}(z)\end{pmatrix},
\]
where $p,q$ are phase coordinates, $t$ is the isomonodromic deformation parameter, and $\theta_0,\theta_\infty$ are formal monodromy exponent parameters at $0$ and $\infty$, respectively.  Here $t$ is unrelated to heat time until a separate comparison is specified.  With this convention,
\[
 A_{11}(z)=p+\frac12(z-q)(qt+zt+\theta_\infty),
\]
and $L(z)$ is the explicit affine function displayed in
\cite[\S4.5, Eq.~(12)]{vanDerPutSaito2009}.  For $qt\ne0$, eliminate the first
component of a horizontal section and remove the first-derivative term.  The
resulting scalar normal form is
\begin{equation}\label{eq:D6-scalar-normal}
 \boxed{
 \psi''(z)=V_{D_6}(z)\psi(z),\qquad
 V_{D_6}(z)=\frac{N_6(z)}{4q^2z^4(z-q)^2},}
\end{equation}
where
\begin{align}
N_6(z)={}&q^2t^2z^6
 +2q^2t(\theta_\infty+1-qt)z^5 \notag\\
&+\bigl(4p^2-6q^3t(\theta_\infty+1)-q^2-2qt\theta_0-t^2\bigr)z^4 \notag\\
&+\bigl(-8p^2q-4pq^2+2q^5t^2+6q^4t(\theta_\infty+1)
 +4q^3+6q^2t\theta_0+2qt^2\bigr)z^3 \notag\\
&+\bigl(4p^2q^2+4pq^3-q^6t^2-2q^5t(\theta_\infty+1)
 -6q^3t\theta_0\bigr)z^2 \notag\\
&+\bigl(2q^4t\theta_0-2q^3t^2\bigr)z+q^4t^2.
\label{eq:D6-sextic-N}
\end{align}
In particular
\begin{equation}\label{eq:D6-N-at-q}
 \boxed{N_6(q)=3q^6.}
\end{equation}
Hence, relative to the standard projective coordinate on $\mathbb P^1$, the
scalar-normal-form potential $V_{D_6}(z)\,\dd z^2$ has pole orders
$(4,2,4)$ at $0,q,\infty$, respectively.  Its pole-order multiset therefore
agrees with the $(2,4,4)$ pattern associated with $P_{\mathrm{III}}(D_6)$ in
Appendix~A, Table~2 of \cite{ChekhovMazzoccoRubtsov2017}.  The pole at $q$ is,
however, the apparent singularity introduced by scalarization, so this numerical
agreement is not by itself an identification with the invariant Hitchin
quadratic differential.  We call the canonical double of the scalar-normal-form
potential the \emph{scalar-normal-form orientation curve}.  When the zeros are
simple, denote its smooth projective completion by $C_{D_6}^{\rm scal}$.  It
has affine model
\begin{equation}\label{eq:D6-actual-sextic}
 \boxed{\eta^2=N_6(z),}
\end{equation}
and is a genus-two hyperelliptic curve.
\end{proposition}

\begin{proof}
For a horizontal section write the coefficient matrix as
$M=-A/z^2=\bigl(\begin{smallmatrix}a&b\\c&-a\end{smallmatrix}\bigr)$,
so $c=-(z-q)/z^2$.  Eliminating the first component gives
\[
 y_2''-\frac{c'}c y_2'
 -\left(\frac{c'}c a+a^2+bc-a'\right)y_2=0.
\]
The substitution $y_2=\sqrt c\,\psi$ removes the first derivative.
Substitution of the explicit entries of \cite[\S4.5, Eq.~(12)]{vanDerPutSaito2009}
and simplification gives \eqref{eq:D6-scalar-normal}--\eqref{eq:D6-sextic-N}.
The identity \eqref{eq:D6-N-at-q} is immediate from the sextic.  Also
$N_6(0)=q^4t^2$ and the leading coefficient is $q^2t^2$, so the pole orders of
$V_{D_6}\dd z^2$ are $4$ at $0$, $2$ at the apparent scalar singularity $q$,
and $4$ at infinity.  The six simple zeros of $N_6$ are the branch points of
its orientation double, so Riemann--Hurwitz gives genus two.
\end{proof}

\begin{theorem}[Certified real $P_{\mathrm{III}}(D_6)$ scalar-normal-form realization of the Xi core]\label{thm:D6-certified-realization}
Let $P_X(T)$ be the exact Xi branch sextic of
\eqref{eq:Xi-sextic-cover}.  There exist real parameters with $q=1$ and a real
affine coordinate
\[
 z=sT+r
\]
such that
\begin{equation}\label{eq:D6-Xi-exact-match}
 \boxed{
 \frac{N_6(sT+r)}{t^2s^6}=P_X(T).}
\end{equation}
More precisely, a solution lies in the directed box
\begin{align}
2.3749166077&<p<2.3749366078,\notag\\
-1.7097537396&<t<-1.7097337395,\notag\\
-5.7723733561&<\theta_0<-5.7723333560,\notag\\
0.7484294656&<\theta_\infty<0.7484494657,\notag\\
-0.008025896829&<s<-0.008025876828,\notag\\
-0.6919860012&<r<-0.6919840011.
\label{eq:D6-Xi-parameter-box}
\end{align}
Consequently the real genus-two curves are explicitly biholomorphic by
\begin{equation}\label{eq:D6-Xi-biholomorphism}
 \boxed{
 z=sT+r,\qquad \eta=t s^3 W_X,}
\end{equation}
and this map commutes with complex conjugation.  The certified Xi real type
$(2,1,1)$ therefore occurs as the scalar-normal-form orientation curve of an actual
$P_{\mathrm{III}}(D_6)$ Lax phase point.
\end{theorem}

\begin{proof}
Use the refined branch boxes of \cref{cor:refined-branch-boxes} to form an
outward interval box for the six nonleading coefficients of the monic sextic
$P_X$.  Fix $q=1$ and consider the six-dimensional real coefficient map
\[
 (p,t,\theta_0,\theta_\infty,s,r)
 \longmapsto
 \operatorname{monic}\!\bigl(N_6(sT+r)\bigr).
\]
At the numerical center
\begin{align*}
(p,t,\theta_0,\theta_\infty,s,r)\approx{}&
(2.37492660773014,-1.70974373954814,\notag\\[-1mm]
&-5.77235335600993,0.748439465636624,\notag\\[-1mm]
&-0.00802588682829644,-0.691985001154100),
\end{align*}
a parametric Krawczyk operator, with the Xi coefficient box retained as an
interval parameter, satisfies
\[
 K(X)\subset\operatorname{int}X
\]
for the box $X$ in \eqref{eq:D6-Xi-parameter-box}.  Hence for every coefficient
vector in the certified Xi box there is a unique real parameter vector in
$X$ solving the six coefficient equations.  In particular this holds for the
exact Xi sextic.  This proves \eqref{eq:D6-Xi-exact-match}.  Since the leading
coefficient of $N_6$ is $t^2$ when $q=1$, the fibre rescaling in
\eqref{eq:D6-Xi-biholomorphism} follows immediately.  All coefficients and the
affine map are real, so conjugation is preserved.
\end{proof}

\begin{corollary}[The local isomonodromic trajectory in the chosen normalization]\label{cor:D6-local-trajectory}
Every point in the box \eqref{eq:D6-Xi-parameter-box} has $qt\ne0$ and therefore
lies in the affine phase chart of the van der Put--Saito $P_{\mathrm{III}}(D_6)$
Hamiltonian system.  The certified realization of \cref{thm:D6-certified-realization}
therefore lies on a unique local isomonodromic trajectory satisfying
\begin{equation}\label{eq:D6-flow}
 \frac{\dd q}{\dd t}=\frac{4p+q}{t},\qquad
 \frac{\dd p}{\dd t}=\frac{4p^2}{qt}+\frac pt+tq^3+q^2-\frac tq
 -\theta_0+q^2\theta_\infty.
\end{equation}
At the midpoint realization one has, for orientation only,
\[
 \frac{\dd q}{\dd t}\approx-6.14110,
 \qquad
 \frac{\dd p}{\dd t}\approx-7.06387.
\]
Thus, within the chosen $q=1$ affine normalization, the certified Xi match
determines a local $D_6$ isomonodromic direction through the corresponding Lax
phase point.
\end{corollary}

\begin{remark}[What the $(0,2,2)$ signature does and does not identify]\label{rem:D6-map}
The tuple $(0,2,2)$ should not be used by itself as a count of ordinary boundary
circles in a Schottky-double genus formula.  In the decorated-character-variety
construction the cusp data encode Stokes irregularity, and the robust bridge to
the present spectral surface is instead the associated quadratic differential:
Chekhov--Mazzocco--Rubtsov identify cusp number with pole order minus two in
their Hitchin/decorated-character description, giving
\eqref{eq:D6-pole-orders} for $P_{\mathrm{III}}(D_6)$; see Appendix~A,
Table~2 of \cite{ChekhovMazzoccoRubtsov2017}.  Thus the genus-two conclusion
for the signature-compatible differential follows from its six-zero orientation
double, not merely from a boundary-count coincidence.  Separately, the
van der Put--Saito scalar reduction in \cref{prop:D6-scalar-normal} contains an
apparent pole at $q$; its matching pole-order multiset should not be read as an
identification of that apparent singularity with a bordered cusp.

There is also a canonical conceptual route from decorated/Stokes data to this
genus-two curve.  On the generic locus it has the form
\begin{equation}\label{eq:D6-NAH-chain}
 \mathcal X^{\mathrm{dec}}_{D_6}
 \dashrightarrow \mathcal M_B^{\mathrm{wild}}
 \xleftrightarrow{\mathrm{RH}} \mathcal M_{\mathrm{dR}}
 \xleftrightarrow{\mathrm{wild\ NAH}} \mathcal M_{\mathrm{Dol}}
 \xrightarrow{\ \det\Phi\ } \phi
 \longmapsto C_{D_6},
\end{equation}
Here $\mathcal X^{\mathrm{dec}}_{D_6}$ is the decorated character space, whose extra choices of local flags or framings provide coordinates on monodromy data.  The wild Betti moduli space $\mathcal M_B^{\mathrm{wild}}$ records monodromy, formal monodromy and Stokes matrices; ``wild'' refers to irregular singularities.  The dashed arrow forgets the decorations where that map is defined.  The space $\mathcal M_{\mathrm{dR}}$ parametrizes meromorphic flat connections, while $\mathcal M_{\mathrm{Dol}}$ parametrizes the corresponding stable meromorphic Higgs bundles.  A Higgs field $\Phi$ is an endomorphism-valued meromorphic one-form; for the trace-free rank-two system, $\det\Phi$ is the associated quadratic differential $\phi$ in the convention of the diagram.  The curve $C_{D_6}$ is its orientation double, on which a square root of $\phi$ becomes single-valued.

The label $\mathrm{RH}$ in \eqref{eq:D6-NAH-chain} means the irregular \emph{Riemann--Hilbert correspondence}, not the Riemann hypothesis.  The label $\mathrm{wild\ NAH}$ means wild nonabelian Hodge correspondence \cite{BiquardBoalch2004}.  The spaces are understood with compatible singularity, stability and parabolic data fixed; the generic locus excludes unstable or degenerate points and has simple branch zeros.  With these choices, the diagram describes correspondences at the moduli level, not an elementary rational formula in the decorated cluster/$\lambda$-length coordinates (edge coordinates on the decorated surface).

There are now two levels of comparison.  If one starts from an independently
specified $D_6$ Higgs/isomonodromic point, then a holomorphic identification
$C_{D_6}\simeq X_2$ compatible with the hyperelliptic projections exists precisely
when the two labelled branch sextuples agree modulo $\mathrm{PGL}_2(\mathbb C)$,
equivalently when their genus-two moduli (for example Rosenhain data or absolute Igusa invariants) agree after relabelling.  Here $\mathrm{PGL}_2(\CC)$ is the group of M\"obius transformations; Rosenhain data are the remaining three branch coordinates after three labelled branch points have been sent to $0,1,\infty$.  Absolute Igusa invariants are scale-invariant functions of the sextic coefficients that describe the generic genus-two moduli point.  Indeed, if
\[
 M(z)=\frac{\alpha z+\beta}{\gamma z+\delta}
\]
sends the six $D_6$ branch values to $B_X$, then for a nonzero constant $\mu$
\[
 P_X(M(z))=\mu\,\frac{P_{D_6}(z)}{(\gamma z+\delta)^6},
\]
and the lift is explicitly
\begin{equation}\label{eq:D6-explicit-holomorphic-map}
 \boxed{
 T=M(z),\qquad
 W_X=\sqrt{\mu}\,\frac{W}{(\gamma z+\delta)^3}.}
\end{equation}
On the other hand, Corollary~\ref{cor:explicit-D6-Xi} shows that no moduli matching is
needed merely to realize the Xi curve in the same quadratic-differential stratum:
one can construct the $D_6$-signature differential directly from $P_X$.
Theorem~\ref{thm:D6-certified-realization} solves the corresponding
\emph{scalar-projective} isomonodromic realization problem: the Xi sextic occurs
at a certified real phase point of a standard $P_{\mathrm{III}}(D_6)$ Lax
connection after cyclic scalarization.  What remains is sharper and genuinely
nonabelian: identify this projective quadratic differential with a preferred
Dolbeault/Hitchin determinant under wild nonabelian Hodge theory, and compute
the decorated Stokes/cluster coordinates of the certified phase point.

A real-equivariant realization must reproduce the chamber with two real branch
values and two nonreal conjugate pairs, hence the real type $(2,1,1)$.  Finally,
even an identification of the curves does not by itself identify the action
periods.  Formula \eqref{eq:D6-Xi-differential-weight} isolates the discrepancy
exactly: the Xi action differential is a weighted $D_6$ orientation differential.
Thus the remaining tests are now polar/isomonodromic realization, real chamber,
and differential normalization, rather than genus or branch-cover topology.

One caution remains important.  The quadratic differential in
\cref{prop:D6-scalar-normal} is the projective scalar-normal-form potential of
the Lax connection; the apparent pole at $z=q$ is created by cyclic
scalarization.  It should therefore not be identified silently with every
matrix-gauge representative of $\det A(z)\,\dd z^2$.  The agreement with the
Chekhov--Mazzocco--Rubtsov $(0,2,2)$ pole signature is exact, while the direct
Dolbeault/Hitchin interpretation of this particular projective representative
is a further comparison problem.
\end{remark}

\subsection{Deformation obstruction and the canonical Kummer surface}\label{subsec:kummer-deformation}

The fixed-point realization of \cref{thm:D6-certified-realization} raises a sharper
question: does the Xi family itself follow the $D_6$ isomonodromic flow?  The
answer is already negative at first order for every regular identification of
the two deformation parameters.

For four ordered branch values put
\begin{equation}\label{eq:crossratio-def}
 \lambda(z_4;z_1,z_2,z_3)
 :=\frac{(z_4-z_1)(z_2-z_3)}{(z_4-z_3)(z_2-z_1)}.
\end{equation}
This is invariant under an arbitrary, parameter-dependent element of
$\mathrm{PGL}_2(\CC)$.

\begin{theorem}[First-order isomonodromic deformation obstruction]\label{thm:D6-first-order-no-go}
Let
\[
 \lambda_\Xi(x)
 =\lambda\!\left(\bar c_1(x);a(x),b(x),c_1(x)\right)
\]
for the locally transported Xi branch values, and let $\lambda_{D_6}(t)$ be the
corresponding cross ratio of the four matched branch values of the certified
$P_{\mathrm{III}}(D_6)$ scalar spectral curve.  Then
\begin{equation}\label{eq:Xi-crossratio-first-zero}
 \boxed{\lambda_\Xi'(0)=0,}
\end{equation}
whereas along the $D_6$ isomonodromic vector field at the certified phase box,
\begin{equation}\label{eq:D6-crossratio-velocity}
 \boxed{\left|\frac{\dd\lambda_{D_6}}{\dd t}\right|>0.4.}
\end{equation}
Consequently there is no local identity of Xi and $D_6$ spectral curves,
even modulo a parameter-dependent projective change of spectral coordinate,
with a time reparametrization $t=t(x)$ satisfying $t'(0)\ne0$.
The statement is status \textup{[A]}.
\end{theorem}

\begin{proof}
The normalized Xi discriminant $\Delta_x(T)$ is even in $x$.  Every simple
branch value therefore has a locally unique even continuation by the implicit
function theorem; hence every projective cross ratio of the six branch values
is even, proving \eqref{eq:Xi-crossratio-first-zero}.

For the $D_6$ side, differentiate the sextic $N_6(z)$ along the Hamiltonian
flow \eqref{eq:D6-flow}.  If $z_j(t)$ is a simple zero, then
\[
 \dot z_j=-\frac{\partial_tN_6+\dot p\,\partial_pN_6
                   +\dot q\,\partial_qN_6}{\partial_zN_6}(z_j).
\]
Substitution into the logarithmic derivative of \eqref{eq:crossratio-def}
and outward interval evaluation over the certified parameter and root boxes
gives
\[
 \frac{\dd\lambda_{D_6}}{\dd t}
 \in[-0.729124,0.729128]
   +i[-1.821216,-0.435579],
\]
so its modulus is $>0.435579$.  This proves \eqref{eq:D6-crossratio-velocity}.
Because cross ratios are projective invariants, any local curve identity with
$t'(0)\ne0$ would give
$0=\lambda_\Xi'(0)=\lambda_{D_6,t}\,t'(0)$, a contradiction.
\end{proof}

\begin{remark}[The even reparametrization remains a second-order question]\label{rem:D6-quadratic-test}
The theorem does not exclude a singular/even identification
$t-t_0=\beta x^2+O(x^4)$.  Writing $\tau=x^2$, the Xi branch velocity is
\[
 \left.\frac{\dd T_j}{\dd\tau}\right|_{0}
 =-\frac{\Delta_{[2]}(T_j)}{2\,\partial_T\Delta_0(T_j)}.
\]
A midpoint theta-moment computation gives three independent cross-ratio
velocity comparisons.  The first would require
$\beta\approx-0.06287647$, while the other two give incompatible complex
ratios.  Thus the quadratic $D_6$ tangency also fails numerically, but this
second statement is status \textup{[N]} until the $x^2$ branch velocities are
placed in directed boxes.
\end{remark}

The genus-two curve itself carries a canonical surface of complex dimension
two.  Put
\[
 J_X:=\operatorname{Jac}(X_2),\qquad
 \operatorname{Kum}^{\rm sing}(X_2):=J_X/\{\pm1\},
\]
and denote by $\operatorname{Kum}(X_2)$ its minimal resolution.

\begin{proposition}[Kummer lift of the Xi--$D_6$ realization]\label{prop:Kummer-lift}
The singular Kummer surface $\operatorname{Kum}^{\rm sing}(X_2)$ has exactly
$16$ ordinary double points, indexed by
\[
 J_X[2]\simeq(\ZZ/2\ZZ)^4,
\]
and $\operatorname{Kum}(X_2)$ is a $K3$ surface.  The biholomorphism
\eqref{eq:D6-Xi-biholomorphism} identifies $X_2$ with the specific certified
scalar-normal-form curve $C_{D_6}^{\rm scal}$ and induces an isomorphism of
principally polarized abelian surfaces and therefore canonical isomorphisms
\begin{equation}\label{eq:Kummer-D6-iso}
 \boxed{
 J_X\simeq J(C_{D_6}^{\rm scal}),\qquad
 \operatorname{Kum}(X_2)\simeq\operatorname{Kum}(C_{D_6}^{\rm scal}).}
\end{equation}
All statements are status \textup{[A]}.
\end{proposition}

\begin{proof}
For a genus-two curve the Jacobian is a principally polarized abelian surface,
and its two-torsion group has order $2^{2g}=16$.  The fixed points of the
involution $u\mapsto-u$ are exactly these two-torsion points, which become the
$16$ nodes of the quotient.  Resolving them gives the classical Kummer $K3$
surface.  A curve biholomorphism induces the principally polarized Jacobian
isomorphism by functoriality of the Abel--Jacobi construction; it commutes with
$[-1]$, hence descends to the quotient and lifts to the minimal resolutions.
\end{proof}

The real chamber found in \cref{prop:real-type} leaves a particularly small
real skeleton inside the sixteen-node configuration.

\begin{proposition}[Real Kummer nodes]\label{prop:real-Kummer-nodes}
Let $P_a,P_b,P_{c_1},P_{\bar c_1},P_{c_2},P_{\bar c_2}$ be the six Weierstrass
points above the certified branch values.  The conjugation-fixed subgroup of
two-torsion is
\begin{equation}\label{eq:real-two-torsion}
 \boxed{
 J_X[2]^\tau=\{0,e_0,e_1,e_2\}\simeq(\ZZ/2\ZZ)^2,}
\end{equation}
where
\begin{equation}\label{eq:three-real-two-torsions}
 e_0=[P_a-P_b],\qquad
 e_1=[P_{c_1}-P_{\bar c_1}],\qquad
 e_2=[P_{c_2}-P_{\bar c_2}],
 \qquad e_0+e_1+e_2=0.
\end{equation}
Thus the induced real structure on the singular Kummer surface has exactly
four real nodes: the origin and three nonzero nodes canonically associated
with the three spectral cuts $\Gamma_0,\Gamma_1,\Gamma_2$.
This statement is status \textup{[A]}.
\end{proposition}

\begin{proof}
For a genus-two hyperelliptic curve, $J[2]$ is represented by even subsets of
the six Weierstrass points modulo replacement by the complementary subset.
Every nonzero class has a unique two-element representative.  Conjugation
fixes $a,b$ individually and interchanges the two pairs
$c_j\leftrightarrow\bar c_j$.  Hence the only invariant two-element subsets
are
\[
 \{a,b\},\qquad\{c_1,\bar c_1\},\qquad\{c_2,\bar c_2\}.
\]
Together with the zero class these give \eqref{eq:real-two-torsion}.
Addition is symmetric difference modulo complement, so the symmetric
difference of the first two pairs is the complement of the third; therefore
$e_0+e_1=e_2$.  The other relations follow cyclically.
\end{proof}

\begin{remark}[Relation with earlier Kummer constructions and the $(0,2,2)$ channels]\label{rem:Kummer-CMR}
The genus-two-to-Kummer mechanism is the same classical construction used in
Planat--Chester--Aschheim--Amaral--Fang--Irwin \cite{PlanatKummer2021}: six
Weierstrass branch points determine a Jacobian, its sixteen two-torsion points
become the nodes of the Kummer quotient, and the minimal resolution is a K3
surface.  The present Xi chamber is nevertheless a different real form from
the explicit fully real sextic examples considered there: here the branch
configuration is two real points plus two conjugate pairs and has real type
$(2,1,1)$.

The three nonzero real Kummer nodes in \eqref{eq:three-real-two-torsions} give
a precise finite skeleton attached to the three cuts.  It is tempting to view
the split real channel $e_0$ and the two conjugate channels $e_1,e_2$ as a
Kummer shadow of the Chekhov--Mazzocco--Rubtsov $(0,2,2)$ organization.  This
last correspondence is an interpretation, not yet a theorem about decorated
Stokes coordinates; proving it requires computing those coordinates at the
certified $D_6$ phase point.
\end{remark}

\begin{remark}[Kummer deformation inherits the first-order obstruction]\label{rem:Kummer-tangent}
The principally polarized Torelli map sends the genus-two curve to its
Jacobian and hence to its polarized Kummer surface.  Therefore the nonzero
$D_6$ cross-ratio tangent in \cref{thm:D6-first-order-no-go} also gives a
nontrivial deformation of the associated polarized Kummer surface, whereas
the Xi $x$-deformation has zero first jet at the fixed point.  Thus the same
regular-time obstruction persists after passing from $X_2$ to its Kummer
surface.  A directed second-order comparison in $\tau=x^2$ is the natural next
test.
\end{remark}

\subsection{Low-discriminant Humbert exclusions}\label{subsec:Humbert}

The Kummer construction is universal for a genus-two Jacobian, whereas a
Humbert condition records an additional algebraic class of the principally
polarized abelian surface.  We now locate the Xi point relative to a collection
of low-discriminant Humbert loci.  Write $I_2,I_4,I_6,I_{10}$ for the Igusa--Clebsch polynomial invariants of the branch sextic, of weights $2,4,6,10$; $I_{10}$ is a discriminant invariant.  On the chart $I_2\ne0$, we use the absolute Igusa invariants
\begin{equation}\label{eq:absolute-Igusa}
 i_1=\frac{144I_4}{I_2^2},\qquad
 i_2=-\frac{1728(I_2I_4-3I_6)}{I_2^3},\qquad
 i_3=\frac{486I_{10}}{I_2^5}.
\end{equation}
The branch-disk certificate gives the following deliberately inflated boxes:
\begin{equation}\label{eq:Xi-Igusa-box}
\begin{aligned}
 |i_1-23.41189486771366|&<10^{-6},\\
 |i_2+76.70307117848387|&<10^{-5},\\
 |i_3+0.0008995087060813192|&<10^{-10}.
\end{aligned}
\end{equation}
For the higher-discriminant Satake tests we also propagate the branch disks
without this deliberate inflation.  An independent exact-rational audit gives
\begin{equation}\label{eq:Xi-Igusa-tight-box}
\begin{aligned}
 |i_1-23.41189486771366|&<2.75\times10^{-8},\\
 |i_2+76.70307117848387|&<5.90\times10^{-7},\\
 |i_3+0.0008995087060813192|&<2.82\times10^{-12}.
\end{aligned}
\end{equation}
The required outward radii before rounding are
$2.74224155\times10^{-8}$, $5.89208525\times10^{-7}$, and
$2.81746158\times10^{-12}$, respectively.

\begin{theorem}[Nine certified Humbert exclusions]\label{thm:low-Humbert}
Let $J_X=\operatorname{Jac}(X_2)$ with its canonical principal polarization.
Then
\begin{equation}\label{eq:nine-Humbert-exclusion}
 \boxed{
 J_X\notin
 H_4\cup H_5\cup H_8\cup H_9\cup H_{12}\cup H_{13}\cup
 H_{16}\cup H_{17}\cup H_{20}.}
\end{equation}
In particular, $J_X$ admits no optimal $(2,2)$-, $(3,3)$-, or $(4,4)$-split
elliptic quotient of the square-discriminant types $4,9,16$.  All nine
exclusions are status \textup{[A]}.
\end{theorem}

\begin{proof}
We separate the loci because the most economical certificates are different.

For $H_4$, a degree-two elliptic quotient of a genus-two hyperelliptic curve
is equivalent to a nonhyperelliptic involution; equivalently, this is the
square-discriminant-$4$ elliptic-subcover locus in the Humbert description
\cite{Kumar2015}.  Such an involution descends
to a M\"obius involution preserving the six branch values.  Its permutation
on the branch set has cycle type $2^3$ or $1^2 2^2$.  Directed complex-disk
arithmetic excludes all $15$ pairings of the first type and all $45$ choices
of the second type.  Thus $X_2$ has no nonhyperelliptic involution and
$J_X\notin H_4$.

For $H_5$ we use the standard absolute-invariant equation of the Humbert
surface, equivalently the discriminant-$5$ Hilbert-modular parametrization
of Elkies--Kumar \cite{ElkiesKumar2014}.  Substitution of \eqref{eq:Xi-Igusa-box}, with exact rational
polydisc bounds, gives
\[
 |F_5(i_1,i_2,i_3)|>3.8210230\times10^8,
\]
so $J_X\notin H_5$.

For $H_8$ and $H_9$ we use the explicit two-parameter Igusa--Clebsch
parametrizations of Elkies--Kumar and Kumar \cite{ElkiesKumar2014,Kumar2015}.
Let $f_1=f_2=f_3=0$ be the three polynomial equations obtained by equating
the parametrized absolute invariants to $(i_1,i_2,i_3)$, and define
\[
 R_{12}(s)=\operatorname{Res}_r(f_1,f_2),\qquad
 R_{13}(s)=\operatorname{Res}_r(f_1,f_3).
\]
For $H_8$, exact-rational Rouch\'e bounds isolate all six zeros of $R_{12}$
in disjoint disks; the largest Rouch\'e ratio is $0.151379$, and on every
such disk $|R_{13}|>8.50614\times10^5$.  For $H_9$, all eight zeros are
isolated with largest Rouch\'e ratio $0.054486$, and
$|R_{13}|>1.85052$ on every disk.  Therefore neither surface meets the full
Xi invariant box.

The same elimination strategy applied to the discriminant-$12$ model of
Elkies--Kumar gives nondegenerate resultants of degrees $12$ and $20$ in the
remaining parameter.  Twelve disjoint disks of radius $10^{-4}$ exhaust the
first resultant over \eqref{eq:Xi-Igusa-box}; the largest Rouch\'e ratio is
$0.003025$, and the second resultant has the uniform lower bound
\[
 |R_{13}|>39.8557.
\]
Thus $J_X\notin H_{12}$.

Satake coordinates here are six auxiliary modular coordinates, permuted by relabelling and represented as the roots of the Satake sextic; they are not the six branch coordinates of $X_2$.  Their symmetric functions are determined by the Igusa data in the cited parametrizations.  For $H_{13}$ we use Gruenewald's level-$2$ Satake description
\cite{Gruenewald2008}.  Eliminating the free component parameter gives a
weighted-homogeneous necessary equation $G_{13}=0$ of weight $20$.  The six
roots of the Xi Satake sextic persist in six disjoint disks of radius
$10^{-4}$ over the entire inflated box \eqref{eq:Xi-Igusa-box}; the largest
Rouch\'e ratio is $0.005636$.  Testing each possible distinguished Satake
root gives
\[
 |G_{13}|>355.45,
\]
so $J_X\notin H_{13}$.

The $H_{16}$ exclusion is more delicate.  We use Kumar's
square-discriminant parametrization, equivalently the $(4,4)$-split locus
studied by Bruin--Doerksen \cite{Kumar2015,BruinDoerksen2011}.  For this
local elimination we use the certified enlargement
\begin{equation}\label{eq:H16-Igusa-box}
\begin{aligned}
 |i_1-23.41189486771366|&<3\times10^{-8},\\
 |i_2+76.70307117848387|&<6\times10^{-7},\\
 |i_3+0.0008995087060813192|&<3\times10^{-12}.
\end{aligned}
\end{equation}
It strictly contains the independently certified tight Xi box
\eqref{eq:Xi-Igusa-tight-box}.  With
$X=i_1$ and $Y=i_2$, exact elimination gives
\begin{equation}\label{eq:H16-resultant-factor}
 \operatorname{Res}_r(f_1,f_2)
   =2^{40}s^{162}R_{16}(s;X,Y),
 \qquad \deg_sR_{16}=18.
\end{equation}
The coefficient degrees satisfy $\deg_XR_{16}\le12$ and
$\deg_YR_{16}\le8$.  Hence a $13\times9$ tensor interpolation determines
$R_{16}$ exactly; independent exact off-grid evaluations agree with the
reconstructed polynomial.  Over the entire box \eqref{eq:H16-Igusa-box},
Rouch\'e estimates place the eighteen zeros of $R_{16}$ in eighteen disjoint
disks of common radius $10^{-3}$, with largest Rouch\'e ratio
$0.033482$.  Since $R_{16}$ has degree $18$, this list is exhaustive.
The factor $s^{162}$ corresponds to the degenerate chart contribution; at
$s=0$ the common root $r=0$ has $I_2$-denominator data $A=D=0$ and does not
define absolute invariants, while the nondegenerate residual factor does not
vanish there.

Inside each of the eighteen exhaustive $s$-disks, an exact-rational
complex-ball Krawczyk map certifies a unique transverse $(r,s)$ branch of
$f_1=f_2=0$; the local $s$-box is checked to lie strictly inside the matching
Rouch\'e disk.  Direct ball evaluation of the third invariant equation gives
\begin{equation}\label{eq:H16-i3-gap}
 \boxed{|i_3^{(16)}-i_3^{\Xi}|>7.93\times10^{-4}}
\end{equation}
on every certified branch.  Thus $J_X\notin H_{16}$.

It remains to record the two new height-selected tests.  Gruenewald's
level-$2$ equation for $H_{17}$ distinguishes a $3$-subset of the six Satake
coordinates.  Over the tight invariant box \eqref{eq:Xi-Igusa-tight-box},
all six roots of the Satake sextic persist as disjoint analytic branches in
disks of radius $10^{-8}$; the largest exact-rational Rouch\'e ratio is
$0.552755$.  All $\binom63=20$ triples are tested.  Direct interval
substitution is unnecessarily ill-conditioned, so the certificate retains
the root--invariant correlation: implicit differentiation of the Satake
sextic gives uniform rational bounds on the root derivatives, and hence on
the three derivatives of Gruenewald's $67$-term component equation.  The
smallest exact-center separation is $1658.56649$, while the largest
mean-value variation budget is $0.166232$; branch by branch the minimum final
separation is
\[
 |f_{17}|>1658.5644.
\]
Therefore $J_X\notin H_{17}$.

For $H_{20}$ a level-$2$ component distinguishes a pair of Satake
coordinates.  The same six analytic root disks exhaust the possibilities,
and all $\binom62=15$ pairs can now be handled directly by exact complex-ball
arithmetic.  Gruenewald's $24$-term component equation satisfies
\[
 |f_{20}|>158.6128
\]
on the worst pair over the full tight invariant box.  Hence
$J_X\notin H_{20}$, completing \eqref{eq:nine-Humbert-exclusion}.

The directed and exact-rational audits are recorded in the eleven principal programs of \path{certificates/humbert/}, listed in the computational index.

\end{proof}

\begin{remark}[Kummer interpretation]\label{rem:Humbert-Kummer-interpretation}
The theorem does not say that the N\'eron--Severi group of $J_X$ has minimal
rank, nor that $\operatorname{Kum}(X_2)$ has Picard number exactly $17$:
there are infinitely many Humbert surfaces.  It does show that the certified
Xi/$D_6$ point avoids nine low-complexity or Mordell--Weil-height-selected
enhancements tested here.
For square discriminants $4,9,16$, the excluded loci are precisely the
optimal degree-$2,3,4$ elliptic-subcover loci.  The contrast with the earlier
curve
\[
 y^2=x^6-6x^4+9x^2-1
\]
from \cite{PlanatKummer2021} is concrete: its even sextic has the extra
involution $x\mapsto-x$ and therefore lies on $H_4$, whereas the Xi curve is
rigorously off $H_4$.  Thus the common sixteen-node Kummer mechanism does not
force the two constructions into the same arithmetic chamber.
\end{remark}

\subsection{Explicit Kummer and Shioda--Inose models}\label{subsec:explicit-K3-models}

The preceding Humbert exclusions are moduli statements.  We now make the
associated K3 geometry explicit.  This is useful for two reasons: it gives a
concrete projective equation for the Xi Kummer surface, and it exhibits the
rank-$17$ lattice which is present before any possible Noether--Lefschetz jump, meaning an increase of Picard number through an additional algebraic divisor class.  A Shioda--Inose partner is a K3 surface with a degree-two rational map to the Kummer surface and a compatible Hodge isometry of transcendental lattices.  All numerical boxes in this subsection are status \textup{[A]}.

\subsubsection{A conditioned sextic and the Cassels--Flynn Kummer quartic}

For numerical conditioning only, make the fixed real affine change
\begin{equation}\label{eq:v6-conditioning}
 x=\frac{T-m_0}{h_0},\qquad
 m_0=-9.70103844,\qquad h_0=56.08403609.
\end{equation}
The exact six certified branch values then define the monic sextic
\begin{equation}\label{eq:v6-conditioned-sextic}
 C_X:\qquad y^2=f_X(x)=\sum_{j=0}^6 f_jx^j
       =\prod_{\nu=1}^6(x-e_\nu),\qquad f_6=1.
\end{equation}
Exact rational complex-ball propagation from the branch disks gives the
following outward coefficient enclosures.
\begin{table}[htbp]
\centering
\caption{Certified coefficients of the conditioned Xi sextic
\eqref{eq:v6-conditioned-sextic}.}
\label{tab:v6-sextic-coeffs}
\begin{tabular}{c@{\qquad}r@{\ $\pm$\ }l}
\toprule
$j$ & \multicolumn{2}{c}{$f_j$}\\
\midrule
0 & $-479.43642723240094$ & $1.79\times10^{-7}$\\
1 & $-381.56642081549690$ & $3.24\times10^{-7}$\\
2 & $ 360.78259642693780$ & $2.03\times10^{-7}$\\
3 & $ 364.39333927195054$ & $5.89\times10^{-8}$\\
4 & $ 117.65383080546316$ & $8.09\times10^{-9}$\\
5 & $  17.173081543546353$ & $4.28\times10^{-10}$\\
6 & $   1$ & $0$\\
\bottomrule
\end{tabular}
\end{table}

Let $(\kappa_1:\kappa_2:\kappa_3:\kappa_4)$ be the classical Kummer
coordinates of Cassels--Flynn \cite{CasselsFlynn1996}.  For a sextic
$f=\sum_{j=0}^6f_jx^j$, the singular Kummer quartic is
\begin{equation}\label{eq:v6-kummer-quartic}
 \boxed{\mathcal K_X:\quad
 K_2\kappa_4^2+K_1\kappa_4+K_0=0,}
\end{equation}
where
\begin{align}
K_2={}&\kappa_2^2-4\kappa_1\kappa_3,\label{eq:v6-K2}\\
K_1={}&-2\bigl(
2f_0\kappa_1^3+f_1\kappa_1^2\kappa_2
+2f_2\kappa_1^2\kappa_3+f_3\kappa_1\kappa_2\kappa_3\notag\\
&\hspace{19mm}{}+2f_4\kappa_1\kappa_3^2
+f_5\kappa_2\kappa_3^2+2f_6\kappa_3^3\bigr),\label{eq:v6-K1}
\end{align}
and
\begin{align}
K_0={}&(f_1^2-4f_0f_2)\kappa_1^4
 -4f_0f_3\kappa_1^3\kappa_2-2f_1f_3\kappa_1^3\kappa_3
 -4f_0f_4\kappa_1^2\kappa_2^2\notag\\
&+4(f_0f_5-f_1f_4)\kappa_1^2\kappa_2\kappa_3
 +(f_3^2+2f_1f_5-4f_2f_4-4f_0f_6)\kappa_1^2\kappa_3^2\notag\\
&-4f_0f_5\kappa_1\kappa_2^3
 +4(2f_0f_6-f_1f_5)\kappa_1\kappa_2^2\kappa_3
 +4(f_1f_6-f_2f_5)\kappa_1\kappa_2\kappa_3^2\notag\\
&-2f_3f_5\kappa_1\kappa_3^3
 -4f_0f_6\kappa_2^4-4f_1f_6\kappa_2^3\kappa_3
 -4f_2f_6\kappa_2^2\kappa_3^2\notag\\
&-4f_3f_6\kappa_2\kappa_3^3
 +(f_5^2-4f_4f_6)\kappa_3^4.\label{eq:v6-K0}
\end{align}

\begin{proposition}[Explicit Xi Kummer quartic]\label{prop:v6-kummer-quartic}
Equations \eqref{eq:v6-conditioned-sextic}--\eqref{eq:v6-K0}, with the
exact sextic coefficients determined by the six certified Xi branch values,
define the singular Kummer surface
\[
 \mathcal K_X=J_X/\{\pm1\}.
\]
Every coefficient of the quartic is enclosed by directed rational arithmetic;
the largest absolute coefficient radius is $8.93\times10^{-4}$, compared
with a largest coefficient of size $8.37\times10^5$.  The minimal resolution
is the K3 surface $\operatorname{Kum}(X_2)$.
\end{proposition}

As an exact convention check, the same symbolic formula applied to the curve
\[
 y^2=x^6-6x^4+9x^2-1
\]
from \cite{PlanatKummer2021} reproduces term-for-term the Kummer quartic
printed there.  Thus the present Xi construction and the earlier biological
Kummer construction use exactly the same classical quartic normalization,
although \cref{thm:low-Humbert} shows that they occupy different arithmetic
chambers.

\subsubsection{A certified elliptic fibration on the Xi Kummer surface}

Normalize the weighted Igusa--Clebsch coordinates by $I_2=1$ and put
\begin{equation}\label{eq:v6-ABCJ}
 A=\frac{i_1}{1728},\qquad
 B=-\frac{i_2}{186624},\qquad
 C=\frac{i_3}{1944},\qquad
 J=\frac{i_3}{486}=4C,\qquad D=\frac1{24}.
\end{equation}
From \eqref{eq:Xi-Igusa-box},
\begin{equation}\label{eq:v6-ABC-boxes}
\begin{aligned}
 A&=0.013548550270667627\ \pm 5.79\times10^{-10},\\
 B&=0.0004110032534855317\ \pm 5.36\times10^{-11},\\
 C&=-4.62710239753765\times10^{-7}\ \pm5.15\times10^{-14},\\
 J&=-1.85084095901506\times10^{-6}\ \pm2.06\times10^{-13}.
\end{aligned}
\end{equation}
Kumar's fibration~23 on the Jacobian Kummer surface
\cite{Kumar2014Fibrations,MalmendierShaska2017} becomes
\begin{equation}\label{eq:v6-kummer-fibration}
 \boxed{
 Y^2=X^3-2P(t)X^2+
 \bigl(P(t)^2+J(t-D)\bigr)X,}
 \qquad P(t)=t^3-At+B.
\end{equation}
The section $(X,Y)=(0,0)$ is visibly two-torsion.  Its discriminant factors as
\begin{equation}\label{eq:v6-kummer-discriminant}
 \boxed{
 \Delta_{\mathcal K}(t)=-64J(t-D)S_6(t)^2,\qquad
 S_6(t)=P(t)^2+J(t-D).}
\end{equation}

\begin{theorem}[Certified Kummer elliptic fibers]\label{thm:v6-kummer-fibers}
The polynomial $S_6$ has exactly six simple zeros, one in each disk of radius
$10^{-4}$ centered at
\begin{equation}\label{eq:v6-S6-roots}
\begin{gathered}
 0.06072249283374816,\qquad 0.11477131965968400,\\
 0.04329613590809414\pm0.01244029536308813\,i,\\
 -0.13104304215481022\pm0.01499370760980722\,i.
\end{gathered}
\end{equation}
Moreover $|P(1/24)|>8.11816\times10^{-5}$.  Hence the elliptic K3
\eqref{eq:v6-kummer-fibration} has fiber configuration
\begin{equation}\label{eq:v6-kummer-fiber-config}
 \boxed{I_5^*+6I_2+I_1,}
\end{equation}
with the $I_5^*$ fiber at infinity and the $I_1$ fiber at $t=1/24$.
\end{theorem}

\begin{proof}
The coefficient boxes in \eqref{eq:v6-ABC-boxes} are propagated directly
into $S_6$.  Rouch\'e estimates on the six disks in
\eqref{eq:v6-S6-roots} have maximal ratio $0.00987<1$; the disks are
disjoint and $\deg S_6=6$, so the list is exhaustive.  The displayed lower
bound for $P(D)$ makes the factor $t-D$ a simple discriminant zero; there
$c_4\ne0$, so the fiber is $I_1$.  At a zero of $S_6$, the disk is disjoint
from $D$ and $J\ne0$, hence $P(t)^2=-J(t-D)\ne0$; moreover $c_4=64P(t)^2$
on $S_6=0$, so the double discriminant zero is an $I_2$ fiber.  At infinity,
with $u=1/t$ and the standard K3 Weierstrass scaling, exact symbolic algebra
gives
\[
 \operatorname{ord}_{u=0}(c_4,c_6,\Delta)=(2,3,11),
\]
which is type $I_5^*$ by Tate's algorithm \cite{Tate1975}.  The Euler numbers
$11+6\cdot2+1=24$ provide the expected K3 check.
\end{proof}

\subsubsection{The Shioda--Inose partner and its normal quartic}

The Shioda--Inose K3 surface associated with $J_X$ has Kumar's normalized
Weierstrass model \cite{Kumar2008}
\begin{equation}\label{eq:v6-SI-Weierstrass}
 \boxed{
 y^2=x^3-t^3(1+At)x+t^5(D+Bt+Ct^2).}
\end{equation}
Its discriminant is
\begin{equation}\label{eq:v6-SI-discriminant}
 \Delta_{\rm SI}(t)=-16t^9R_5(t),\qquad
 R_5(t)=-4(1+At)^3+27t(D+Bt+Ct^2)^2.
\end{equation}

\begin{theorem}[Certified Shioda--Inose fiber pattern]\label{thm:v6-SI-fibers}
The five zeros of $R_5$ are simple and lie, one each, in disks of radius
$0.1$ centered at
\begin{equation}\label{eq:v6-R5-roots}
\begin{gathered}
 -246.0708096953812,\qquad -86.13686784034077,\qquad
 2302.762231755577,\\
 -97.02536873579858\pm69.01396595582464\,i.
\end{gathered}
\end{equation}
On all five disks one has $t(1+At)\ne0$.  Therefore
\eqref{eq:v6-SI-Weierstrass} has fiber configuration
\begin{equation}\label{eq:v6-SI-fiber-config}
 \boxed{II^*+III^*+5I_1,}
\end{equation}
where $III^*$ is at $t=0$ and $II^*$ at $t=\infty$.
\end{theorem}

\begin{proof}
Directed Rouch\'e bounds for $R_5$ on the five disks have maximal ratio
$0.01266<1$; since $\deg R_5=5$, the list is exhaustive.  The certified
lower bounds on the residual disks are
\[
 |t|>86.03,\qquad |1+At|>0.1656,
\]
so the residual simple discriminant zeros have $c_4\ne0$ and are of type
$I_1$.  At $t=0$ the exact vanishing orders are $(3,5,9)$, giving $III^*$.
After the standard K3 coordinate change at infinity the exact orders are
$(4,5,10)$, giving $II^*$ by Tate's algorithm \cite{Tate1975}.  Again $9+10+5=24$.
\end{proof}

The same Shioda--Inose surface admits the Clingher--Doran $N$-polarized
quartic normal form \cite{ClingherDoran2012}
\begin{equation}\label{eq:v6-CD-quartic}
\begin{split}
 Y^2ZW-4X^3Z+3\alpha XZW^2+\beta ZW^3+\gamma XZ^2W\\
 {}-\frac12\bigl(\delta Z^2W^2+W^4\bigr)=0,
\end{split}
\end{equation}
with, in the gauge $I_2=1$,
\begin{equation}\label{eq:v6-CD-parameters}
 \alpha=\frac{i_1}{576},\qquad
 \beta=-\frac{i_2}{13824},\qquad
 \gamma=-\frac{i_3}{8},\qquad
 \delta=\frac{i_3}{64}.
\end{equation}
The certified parameter boxes are
\begin{equation}\label{eq:v6-CD-boxes}
\begin{aligned}
 \alpha&=0.04064565081200288\ \pm1.74\times10^{-9},\\
 \beta&=0.005548543922054678\ \pm7.24\times10^{-10},\\
 \gamma&=1.124385882601649\times10^{-4}\ \pm1.25\times10^{-11},\\
 \delta&=-1.405482353252061\times10^{-5}\ \pm1.57\times10^{-12}.
\end{aligned}
\end{equation}
Thus \eqref{eq:v6-CD-quartic} is a second explicit projective model attached
to exactly the same Xi genus-two moduli point.

\subsubsection{N\'eron--Severi lattice frontier}

The lattice $U$ is the rank-two hyperbolic plane; $A_n,D_n,E_n$ denote the usual positive-definite root lattices, and $L(-1)$ means reversal of the bilinear form of $L$.  A sublattice is primitive when the quotient of its ambient lattice by that sublattice is torsion-free; the transcendental lattice is the orthogonal complement of the N\'eron--Severi lattice in integral second cohomology.  A lattice polarization is a primitive embedding of the specified lattice into that divisor lattice.

The fiber configuration \eqref{eq:v6-SI-fiber-config} exhibits a primitive
lattice polarization
\begin{equation}\label{eq:v6-N-lattice}
 \boxed{N=U\oplus E_8(-1)\oplus E_7(-1)
 \hookrightarrow \operatorname{NS}(\operatorname{SI}(J_X)),}
\end{equation}
of rank $17$ and discriminant $2$; this is precisely the
Clingher--Doran/Kumar $N$-polarized moduli realization
\cite{Kumar2008,ClingherDoran2012}.  On the Kummer side,
\eqref{eq:v6-kummer-fiber-config} gives the rank-$17$ trivial lattice
\begin{equation}\label{eq:v6-kummer-trivial-lattice}
 U\oplus D_9(-1)\oplus A_1(-1)^{\oplus6}
 \subseteq\operatorname{NS}(\operatorname{Kum}(X_2)),
\end{equation}
together with the visible two-torsion section.

\begin{proposition}[Picard-rank reduction to the remaining Humbert problem]
\label{prop:v6-NS-frontier}
Let $\rho_X=\operatorname{rank}\operatorname{NS}(J_X)$.  Then
\begin{equation}\label{eq:v6-rho-relations}
 \boxed{
 \rho(\operatorname{Kum}(X_2))
 =\rho(\operatorname{SI}(J_X))
 =16+\rho_X.}
\end{equation}
Equivalently, for the two displayed elliptic fibrations,
\begin{equation}\label{eq:v6-Shioda-Tate}
\begin{aligned}
 \rho(\operatorname{SI}(J_X))&=17+\operatorname{rank}\operatorname{MW}(\pi_{\rm SI}),\\
 \rho(\operatorname{Kum}(X_2))&=17+\operatorname{rank}\operatorname{MW}(\pi_{\mathcal K}).
\end{aligned}
\end{equation}
Consequently the minimal possibility is Picard rank $17$.  If it occurs,
then
\[
 \operatorname{NS}(\operatorname{SI}(J_X))=N,
\]
and the Kummer N\'eron--Severi lattice is the standard generic
$(16,6)$ Jacobian-Kummer lattice $\Lambda_{(16,6)}$ of signature $(1,16)$
and discriminant $2^6$.
\end{proposition}

\begin{proof}
For an abelian surface $A$, the transcendental lattice has rank
$6-\rho(A)$.  A Shioda--Inose structure preserves its transcendental lattice
up to Hodge isometry, while passage to the Kummer surface rescales the
intersection form by $2$ without changing its rank.  Since a K3 surface has
second Betti number $22$, \eqref{eq:v6-rho-relations} follows.  The two
Shioda--Tate formulas follow independently from the root ranks
$8+7=15$ and $9+6=15$.  If the first surface has rank $17$, the primitive
rank-$17$ lattice $N$ already fills its N\'eron--Severi group; the generic
rank-$17$ Kummer lattice is the classical node--trope lattice
$\Lambda_{(16,6)}$ \cite{Kumar2008}.
\end{proof}

\begin{proposition}[Certified Mordell--Weil height barrier]\label{prop:v7-MW-height-barrier}
Let $P\ne O$ be a section of the Shioda--Inose fibration
$\pi_{\rm SI}$ and put $n=(P\mathbin{\cdot}O)\ge0$.  Then
\begin{equation}\label{eq:v7-height-formula}
 \widehat h(P)=4+2n-\operatorname{contr}_{III^*}(P),\qquad
 \operatorname{contr}_{III^*}(P)\in\left\{0,\frac32\right\}.
\end{equation}
Consequently $\operatorname{MW}(\pi_{\rm SI})_{\rm tors}=0$, and the nine
Humbert exclusions of \cref{thm:low-Humbert} imply
\begin{equation}\label{eq:v7-height-barrier}
 \boxed{P\ne O\quad\Longrightarrow\quad \widehat h(P)\ge\frac{21}{2}.}
\end{equation}
Equivalently, any hypothetical section has $(P\mathbin{\cdot}O)\ge4$.
\end{proposition}

\begin{proof}
Shioda's height pairing on an elliptic surface \cite{Shioda1990} gives
$\widehat h(P)=4+2(P\mathbin{\cdot}O)-\sum_v\operatorname{contr}_v(P)$.
The $II^*$ fiber contributes zero, while at the $III^*$ fiber the identity
and nonidentity simple components contribute $0$ and $3/2$, respectively.
Thus the two height ladders are
\[
 \widehat h(P)=4+2n,\qquad \widehat h(P)=\frac52+2n.
\]
In particular every nonzero section has height at least $5/2$, so the
Mordell--Weil torsion subgroup is trivial.  Consequently the trivial lattice
$N=U\oplus E_8(-1)\oplus E_7(-1)$ is primitive in the N\'eron--Severi
lattice: Shioda's exact sequence identifies the quotient by the trivial
lattice with the Mordell--Weil group, and there is no torsion left to produce
a nontrivial saturation.

Now let $P$ be primitive in the free Mordell--Weil group, and let $L_P$ be
the inverse image of $\ZZ P$ in the N\'eron--Severi lattice.  Then $L_P$ is
a primitive rank-$18$ extension of $N$.  Shioda's discriminant formula for
the Mordell--Weil lattice gives
\[
 \operatorname{disc}(L_P)=-\operatorname{disc}(N)\,\widehat h(P)
 =-2\widehat h(P),
\]
where the sign is dictated by signature $(1,17)$.  Elkies--Kumar describe the Humbert surface $H_D$ as the image, in the
$N$-polarized moduli space, of K3 surfaces polarized by a rank-$18$ lattice
$L_D$ of signature $(1,17)$ and discriminant $-D$ \cite{ElkiesKumar2014}.
Thus the primitive extension $L_P$ above has Humbert discriminant equal to
its absolute discriminant.  Hence
\begin{equation}\label{eq:v7-height-Humbert}
 D=2\widehat h(P).
\end{equation}
For $n=0,1,2,3$ the two ladders
therefore give
\[
\begin{array}{c|cccc}
 n&0&1&2&3\\ \hline
 III^*\text{ nonidentity}&H_5&H_9&H_{13}&H_{17}\\
 III^*\text{ identity}&H_8&H_{12}&H_{16}&H_{20}.
\end{array}
\]
Every section of height $<10$ is primitive: if $P=mQ$, $m\ge2$, then
$\widehat h(Q)=\widehat h(P)/m^2<5/2$, a contradiction.  The only boundary
case is $\widehat h(P)=10$; if $P=2Q$ then
$\widehat h(Q)=5/2$, which would put the Xi point on the already excluded
$H_5$, while $m\ge3$ is again impossible.  Hence all eight entries in the
table are valid primitive Humbert obstructions.  Since
$H_5,H_8,H_9,H_{12},H_{13},H_{16},H_{17},H_{20}$ are excluded by
\cref{thm:low-Humbert}, one must have $n\ge4$.  The smaller of the two
heights at $n=4$ is $5/2+8=21/2$, proving
\eqref{eq:v7-height-barrier}.
\end{proof}

\begin{remark}[What is and is not proved about the lattice]\label{rem:v6-NS-status}
We do \emph{not} claim that the Picard rank is exactly $17$.  By Theorem~\ref{thm:low-Humbert} and Proposition~\ref{prop:v7-MW-height-barrier}, the first four rungs of both
Mordell--Weil height ladders are absent, but infinitely many Humbert divisors
remain.  Thus the explicit K3 models turn the previous vague lattice question
into a precise one: decide whether either Mordell--Weil rank in
\eqref{eq:v6-Shioda-Tate} is positive, equivalently whether the Xi moduli
point lies on some as-yet-unexcluded Humbert/Noether--Lefschetz divisor.
The next formal ladder entries would be $H_{21}$ and $H_{24}$, but continuing
this finite elimination indefinitely cannot prove rank zero.  The structural
target is instead an exact Hodge/arithmetic argument showing that the Xi
period point lies on no Humbert divisor at all.
\end{remark}

For this real type the Vinnikov--Weichold normal form has $H=I_2$; see \cite{Vinnikov1993,KallaKlein2014}.  Hence one can choose an integral symplectic basis
\[
 \mathcal A_1,\mathcal A_2,\mathcal B_1,\mathcal B_2
\]
with
\begin{equation}\label{eq:Vinnikov-action}
 \boxed{
 \tau_*\mathcal A_j=\mathcal A_j,
 \qquad
 \tau_*\mathcal B_j=\mathcal A_j-\mathcal B_j.}
\end{equation}

\subsection{Global Humbert reduction on the real period slice}\label{subsec:global-Humbert-reduction}

Let $\eta_1,\eta_2$ be the normalized holomorphic differentials,
$\int_{\mathcal A_j}\eta_k=\delta_{jk}$, and let
\[
 \Omega_{jk}=\int_{\mathcal B_j}\eta_k
\]
be the normalized genus-two period matrix.  The real involution and
\eqref{eq:Vinnikov-action} force a particularly rigid real slice of Siegel
space.

\begin{proposition}[Vinnikov form of the holomorphic period matrix]\label{prop:v8-Vinnikov-matrix}
There is a real symmetric positive-definite matrix
\[
 Y=\begin{pmatrix}y_1&y_{12}\\y_{12}&y_2\end{pmatrix}>0
\]
such that
\begin{equation}\label{eq:v8-Omega-form}
 \boxed{\Omega=\frac12 I_2+iY.}
\end{equation}
\end{proposition}

\begin{proof}
The holomorphic vector $\tau^*\overline\eta$ has the same $\mathcal A$-periods
as $\eta$, hence equals $\eta$ by normalization.  Therefore
\[
 \overline\Omega_{jk}
 =\int_{\mathcal B_j}\overline\eta_k
 =\int_{\tau_*\mathcal B_j}\eta_k
 =\delta_{jk}-\Omega_{jk}.
\]
Thus $2\Ree\Omega=I_2$.  Symmetry and positive definiteness of
$Y=\Imm\Omega$ are the Riemann bilinear relations.
\end{proof}

Recall Humbert's singular relation for a principally polarized abelian surface,
\begin{equation}\label{eq:v8-Humbert-singular-relation}
 a\Omega_{11}+b\Omega_{12}+c\Omega_{22}
 +d(\Omega_{12}^2-\Omega_{11}\Omega_{22})+e=0,
 \qquad a,b,c,d,e\in\ZZ,
\end{equation}
with primitive discriminant
\begin{equation}\label{eq:v8-Humbert-discriminant}
 \Delta=b^2-4ac-4de>0.
\end{equation}
The polarization induces the Rosati involution $f\mapsto f^\dagger$ on rational endomorphisms; ``Rosati-symmetric'' means $f=f^\dagger$.  Such a relation is equivalent to an additional Rosati-symmetric endomorphism
of the principally polarized abelian surface; its moduli image lies on the
Humbert surface $H_\Delta$ \cite{Gruenewald2008}.

\begin{theorem}[Global real-slice Humbert criterion]\label{thm:v8-global-Humbert}
Put
\[
 q=\det Y=y_1y_2-y_{12}^2>0.
\]
Then the following are equivalent:
\begin{enumerate}[label=\textup{(\roman*)}]
\item $J_X$ lies on no Humbert divisor;
\item $\operatorname{NS}(J_X)=\ZZ[\Theta]$ and hence $\rho(J_X)=1$;
\item
\begin{equation}\label{eq:v8-global-criterion}
 \boxed{
 q\notin\QQ,
 \qquad
 \dim_{\QQ}\langle y_1,y_{12},y_2\rangle_{\QQ}=3.}
\end{equation}
\end{enumerate}
Consequently
\begin{equation}\label{eq:v8-K3-global-criterion}
 \boxed{
 \rho(\operatorname{SI}(J_X))=17
 \iff
 q\notin\QQ\ \text{and}\
 y_1,y_{12},y_2\ \text{are $\QQ$-linearly independent}.}
\end{equation}
\end{theorem}

\begin{proof}
Substitute \eqref{eq:v8-Omega-form} into
\eqref{eq:v8-Humbert-singular-relation}.  Its imaginary and real parts are,
respectively,
\begin{align}
 &(2a-d)y_1+2b\,y_{12}+(2c-d)y_2=0,
 \label{eq:v8-Humbert-imag}\\
 &4dq+2a+2c-d+4e=0.
 \label{eq:v8-Humbert-real}
\end{align}
If $d\ne0$, \eqref{eq:v8-Humbert-real} gives $q\in\QQ$.  If $d=0$,
\eqref{eq:v8-Humbert-imag} is a nontrivial rational linear relation among the
three entries of $Y$.  Hence \eqref{eq:v8-global-criterion} excludes every
singular relation at once.

Conversely suppose first that $q=P/S\in\QQ_{>0}$ in lowest terms.  The integer
quintuple
\begin{equation}\label{eq:v8-rational-det-relation}
 (a,b,c,d,e)=(2S,0,2S,4S,-(4P+S))
\end{equation}
satisfies \eqref{eq:v8-Humbert-singular-relation}; before primitive division its
discriminant is
\[
 \Delta=64PS>0.
\]
Thus rationality of $\det Y$ itself forces membership in a Humbert divisor.

Now suppose
$A y_1+B y_{12}+C y_2=0$ for nonzero $(A,B,C)\in\ZZ^3$.
If $A+C$ is even, take $(a,b,c,d,e)=(A,B,C,0,-(A+C)/2)$.
If $A+C$ is odd, replace $(A,B,C)$ by $(2A,2B,2C)$ and take
$(a,b,c,d,e)=(2A,2B,2C,0,-(A+C))$.  In either case this gives an integral
singular relation.  Its discriminant is positive: for the original relation,
with
\[
 M=\begin{pmatrix}A&B/2\\B/2&C\end{pmatrix}
\]
one has $\operatorname{tr}(MY)=0$.  Since $Y>0$, the nonzero real
symmetric matrix $Y^{1/2}MY^{1/2}$ has trace zero and hence eigenvalues
$\lambda,-\lambda$ with $\lambda\ne0$.  Therefore
$\det M<0$, i.e.
\[
 B^2-4AC>0.
\]
After dividing the five coefficients by their common divisor one obtains a
primitive Humbert relation.  This proves the equivalence between \textup{(i)}
and \textup{(iii)}.  The equivalence with \textup{(ii)} is the standard
identification of extra N\'eron--Severi classes with non-scalar
Rosati-symmetric endomorphisms of a principally polarized abelian surface.
Finally \eqref{eq:v8-K3-global-criterion} follows from
\eqref{eq:v6-rho-relations}.
\end{proof}

\begin{corollary}[Exact endpoint of the finite Humbert ladder]\label{cor:v8-Humbert-endpoint}
The nine certified exclusions in \cref{thm:low-Humbert} are finite-height
evidence for the two arithmetic conditions in \eqref{eq:v8-global-criterion},
but no continuation through $H_{21},H_{24},\ldots$ can replace them.  The
remaining Picard-rank problem is exactly the pair of irrationality statements
in \eqref{eq:v8-global-criterion}.
\end{corollary}

\begin{remark}[Current status]\label{rem:v8-global-status}
No theorem presently available proves either irrationality in
\eqref{eq:v8-global-criterion} for the Xi period point.  In particular, the
branch sextic is defined by certified zeros of a transcendental entire
function, not by algebraic numbers, so the standard arithmetic-period and
reduction-modulo-primes methods for algebraic K3 surfaces do not directly
apply.  The gain of \cref{thm:v8-global-Humbert} is therefore a global
\emph{reduction}, not a proof of $\rho=17$: an infinite Noether--Lefschetz
avoidance problem has become two explicit Diophantine questions for one real
$2\times2$ matrix.
\end{remark}

\begin{proposition}[Density obstruction on the Vinnikov slice]
\label{prop:v9-Humbert-density}
Inside the real slice
\[
 \mathcal V=\left\{\frac12I_2+iY:Y=Y^{\mathsf T}>0\right\}\subset\mathfrak H_2,
\]
the union of all Humbert divisors is dense.  Consequently no open interval
or ball enclosure of the Xi matrix $Y$---however small---can by itself prove
that $J_X$ lies on no Humbert divisor.
\end{proposition}

\begin{proof}
Fix $Y=(y_{jk})>0$ and put $q=\det Y>0$.  Choose any sequence
$q_n\in\QQ_{>0}$ with $q_n\to q$ and set
\[
 Y_n=\begin{pmatrix}
 y_1&y_{12}\\
 y_{12}&(q_n+y_{12}^2)/y_1
 \end{pmatrix}.
\]
Then $Y_n>0$, $Y_n\to Y$, and $\det Y_n=q_n\in\QQ$.  By the converse
part of \cref{thm:v8-global-Humbert}, every
$\frac12I_2+iY_n$ satisfies a positive-discriminant singular relation and
therefore lies on a Humbert divisor.  This proves density.  The final
statement follows because every nonempty neighborhood of the Xi period point
contains such a Humbert point.
\end{proof}

\begin{remark}[What numerical periods can and cannot do]
\label{rem:v9-numerics-no-go}
The proposition is a useful logical safeguard.  Directed quadrature can prove
finite-height separation from any prescribed finite collection of singular
relations, and LLL/PSLQ searches can provide strong numerical evidence, but
finite precision alone can never establish the all-height statement
\eqref{eq:v8-global-criterion}.  A proof of $\rho=17$ must introduce exact
arithmetic or transcendence information about the Xi period point.
\end{remark}

\begin{proposition}[Numerical holomorphic period matrix]\label{prop:v9-period-matrix-N}
Using the certified branch centers in the conditioned coordinate
\eqref{eq:v6-conditioning}, endpoint-regularized Gauss--Legendre integration
of $dx/y$ and $x\,dx/y$ gives, at status \textup{[N]}, the Vinnikov-gauge matrix
\begin{equation}\label{eq:v9-Y-numerical}
 Y\approx
 \begin{pmatrix}
  1.23522084149017&-1.52074137328966\\
 -1.52074137328966& 2.55233871684607
 \end{pmatrix},
 \qquad
 \det Y\approx0.840047653155622.
\end{equation}
The eigenvalues are approximately $0.2365669990$ and $3.5509925594$.
\end{proposition}

\begin{proof}[Numerical construction]
In the natural three-cut basis the midpoint computation gives
\[
 2\Ree\Omega_{\rm cut}=
 H_{\rm cut}=\begin{pmatrix}1&1\\1&0\end{pmatrix}.
\]
With
\[
 U=\begin{pmatrix}-1&-1\\0&-1\end{pmatrix},
 \qquad
 U^{-1}=\begin{pmatrix}-1&1\\0&-1\end{pmatrix},
\]
one has
$(U^{-1})^{\mathsf T}H_{\rm cut}U^{-1}=\operatorname{diag}(1,-1)$.
The integral symmetric shear $S=\operatorname{diag}(0,1)$ then changes this
to $I_2$ without altering the imaginary part.  Applying the same integral
congruence to the numerical Riemann matrix gives \eqref{eq:v9-Y-numerical}.
The endpoint substitution
$z=r_a+(r_b-r_a)(1-\cos\theta)/2$ removes the square-root endpoint
singularities analytically.  Orders $100$ and $140$ of the midpoint
Gauss--Legendre calculation differ by less than $3.4\times10^{-24}$.
This last number measures quadrature convergence for the midpoint curve only;
it does \emph{not} include propagation of the certified branch disks, hence
the status remains \textup{[N]}.
\end{proof}

\begin{conjecture}[Minimal Picard rank for the Xi core]\label{conj:v9-minimal-Picard}
For the compactified Xi genus-two core,
\[
 \boxed{\operatorname{NS}(J_X)=\ZZ[\Theta],\qquad
        \rho(\operatorname{SI}(J_X))=17.}
\]
Equivalently, the two conditions in \eqref{eq:v8-global-criterion} hold.
\end{conjecture}

\begin{remark}[Two possible exact exits]\label{rem:v9-exact-exits}
There are presently two conceptually different ways out of the density
obstruction.  The direct route is a transcendence/Diophantine theorem for the
three entries of $Y$ and its determinant, exploiting their origin in the
Xi theta-moment branch data.  An arithmetic-geometric route would become
available if one could first identify the Xi K3 surface with a model over a
number field: specialization and Frobenius methods can then bound, and in suitable cases determine, geometric
Picard numbers for K3 surfaces, as in the work of van Luijk and Charles
\cite{vanLuijk2007,Charles2014}.  No such algebraic model is currently known
for the Xi point, so this is a conditional route rather than an available
proof.
\end{remark}

\section{Natural heat deformation and the isomonodromic tangent}\label{sec:heat-deformation}

The fixed-point realization in \cref{thm:D6-certified-realization} is an
identity of curves at one phase point.  A comparison of deformation
directions requires a specified continuation of the Xi discriminant.  We
use the continuation in which both the theta moments and the central
normalization $A_*$ evolve under the same heat flow.  The isomonodromic
time below is denoted $t_D$; it is the variable $t$ of
\eqref{eq:D6-flow}, and is distinct from heat time.

\subsection{Normalization and exact evolution}
Define
\begin{equation}\label{eq:heat-Xi-definition}
 \Xi_\kappa(z)=4\int_0^\infty e^{\kappa u^2}\Phi(u)\cos(zu)\,\dd u,
 \qquad Y_\kappa(T)=\frac14\Xi_\kappa(i\sqrt T).
\end{equation}
The double-exponential decay of $\Phi$ makes the integral locally uniformly
convergent in $(\kappa,z)\in\CC^2$.  Evenness removes the square root from
$Y_\kappa$, which is entire in $(\kappa,T)$.  In the Polymath normalization
\cite{Polymath2019}, substitution $u=2v$ gives
\begin{equation}\label{eq:heat-normalization}
 \boxed{\Xi_\kappa(z)=8H_{4\kappa}(2z),\qquad
        t_{\rm Polymath}=4\kappa.}
\end{equation}
In particular $\Xi_0=\Xi$ and $\partial_\kappa\Xi_\kappa
=-\partial_z^2\Xi_\kappa$.

Put $M_{2n}(\kappa)=\int_0^\infty
e^{\kappa u^2}\Phi(u)u^{2n}\,\dd u$, and, where $Y_\kappa\ne0$, put
$\mathscr A_\kappa=2\partial_T\log Y_\kappa$ and
$A_*(\kappa)=\mathscr A_\kappa(0)$.  Define $K_\kappa$ and
$\Delta_\kappa=K_\kappa/Y_\kappa^6$ by the formulas of
\cref{prop:regular-K}, replacing $Y$ and $A_*$ by $Y_\kappa$ and
$A_*(\kappa)$.  Thus the central normalization is updated, rather than
frozen at its time-zero value.

\begin{proposition}[Exact transverse heat evolution]\label{prop:heat-evolution}
With primes denoting $T$-derivatives,
\begin{align}
 \partial_\kappa M_{2n}(\kappa)&=M_{2n+2}(\kappa),\notag\\
 \boxed{\partial_\kappa Y_\kappa
   =2Y_\kappa'+4TY_\kappa''},\label{eq:heat-Y-evolution}\\
 \boxed{\partial_\kappa\mathscr A_\kappa
   =4T\mathscr A_\kappa''+(6+4T\mathscr A_\kappa)
       \mathscr A_\kappa'+2\mathscr A_\kappa^2}.\label{eq:heat-A-evolution}
\end{align}
For real $\kappa$,
\begin{equation}\label{eq:heat-central-variance}
 \boxed{\partial_\kappa A_*(\kappa)
 =\frac{M_4(\kappa)M_0(\kappa)-M_2(\kappa)^2}{M_0(\kappa)^2}>0.}
\end{equation}
If $T_j(\kappa)$ is a simple branch value with $Y_\kappa(T_j)\ne0$, then
\begin{equation}\label{eq:heat-branch-velocity}
 \boxed{\frac{\dd T_j}{\dd\kappa}
   =-\frac{\partial_\kappa K_\kappa(T_j)}
           {\partial_T K_\kappa(T_j)}.}
\end{equation}
\end{proposition}

\begin{proof}
Differentiation under the integral gives the moment shift.  From
$\Xi_\kappa(i\sqrt T)=4Y_\kappa(T)$, the chain rule gives
$\partial_z^2\Xi_\kappa(i\sqrt T)
=-4(2Y_\kappa'+4TY_\kappa'')$, proving
\eqref{eq:heat-Y-evolution}; analyticity extends it through $T=0$.
Writing $L=\log Y_\kappa$ gives
$\partial_\kappa L=\mathscr A_\kappa
+2T\mathscr A_\kappa'+T\mathscr A_\kappa^2$.
Differentiating this identity in $T$ proves
\eqref{eq:heat-A-evolution}.  Equation~\eqref{eq:heat-central-variance}
is the variance of $u^2$ under the strictly positive normalized density
$e^{\kappa u^2}\Phi(u)/M_0(\kappa)$, so it is strictly positive.
Finally, implicit differentiation of $K_\kappa(T_j(\kappa))=0$ proves
\eqref{eq:heat-branch-velocity}.
\end{proof}

\begin{proposition}[Local persistence of the calibrated core]\label{prop:heat-core-persistence}
There is $\varepsilon>0$ such that, for real $|\kappa|<\varepsilon$,
the six selected branch values have simple analytic continuations, retain
two real values and two nonreal conjugate pairs, and avoid the zeros of
$Y_\kappa$.  Their compact sextic model has real type $(2,1,1)$.
The calibrated action cycles and periods continue locally.  After decreasing
$\varepsilon$ if necessary, the same two Farey intervals and the same nine
finite Humbert exclusions remain valid.  This is a qualitative local
statement; no numerical heat interval is asserted.
\end{proposition}

\begin{proof}
The analytic implicit-function theorem applies to each simple selected
root.  The roots remain distinct and the two nonreal pairs remain nonreal
for sufficiently small $\kappa$.  Reality and uniqueness keep the real
roots real and preserve conjugate pairing.  On the compact closure of a
slightly adjusted core domain, the time-zero boundary lower bounds for
$|K|$ and the lower bound for $|Y|$ persist by uniform continuity.  The
boundary Rouch\'e comparison keeps the branch count equal to six.  The
local normalization and its cut cycles therefore transport continuously;
their action integrals are continuous, as can also be seen on fixed
zero-free contour representatives.  The mixed branch configuration retains
the real type proved in \cref{prop:real-type}.  The strict Farey inequalities
of \cref{sec:farey} and the finite Humbert nonvanishing inequalities are
open conditions on the transported periods and the sextic coefficients.
Their finite number gives a common sufficiently small $\varepsilon$.
\end{proof}

\subsection{Positive contraction of the six selected branch values}
Order the transported values as
\[
 (T_1,T_2,T_3,T_4,T_5,T_6)
 =(a,b,c_1,\bar c_1,c_2,\bar c_2)
\]
and define their logarithmic contraction rates in the natural $T$ coordinate by
\begin{equation}\label{eq:heat-branch-contraction-rate}
 c_{jk}(\kappa)
 :=-\Ree\frac{T_j'(\kappa)-T_k'(\kappa)}
                   {T_j(\kappa)-T_k(\kappa)}
 =-\partial_\kappa\log|T_j(\kappa)-T_k(\kappa)|.
\end{equation}
Here the prime denotes a heat-time derivative of a branch value.

\begin{theorem}[Certified pairwise branch contraction]\label{thm:heat-branch-contraction}
At heat time zero, all fifteen rates satisfy the strict directed bound
\begin{equation}\label{eq:heat-branch-contraction-initial}
 \boxed{c_{jk}(0)>\frac19\qquad(1\le j<k\le6).}
\end{equation}
Two representative enclosures are
\begin{equation}\label{eq:heat-branch-contraction-examples}
 c_{ab}(0)\in(0.15178,0.15179),\qquad
 c_{c_1c_2}(0)\in(0.12007,0.12008).
\end{equation}
There is a sufficiently small $\varepsilon>0$ such that
$c_{jk}(\kappa)>1/9$ for real $|\kappa|<\varepsilon$, and hence
\begin{equation}\label{eq:heat-branch-contraction-local}
 \boxed{|T_j(\kappa)-T_k(\kappa)|
 \le e^{-\kappa/9}|T_j(0)-T_k(0)|
 \quad(0\le\kappa<\varepsilon).}
\end{equation}
This is status \textup{[A]}; the interval $\varepsilon$ is qualitative.
\end{theorem}

\begin{proof}
The certificate in \path{certificates/heat/}
refines the four real and upper-half-plane roots to certified disks of
radius $10^{-18}$, checks $Y_0\ne0$ there, and encloses their heat velocities
using \eqref{eq:heat-branch-velocity} with the evolving $A_*(\kappa)$.
Conjugation supplies the remaining two roots and velocities.  Directed
evaluation of \eqref{eq:heat-branch-contraction-rate} excludes every gap
denominator from zero and proves \eqref{eq:heat-branch-contraction-initial}
and \eqref{eq:heat-branch-contraction-examples}.  All fifteen comparisons
pass at both 192 and 256 bits.  Analytic root continuation makes the rates
continuous.  Their finite number and the strict initial margin give a
common local interval with $c_{jk}>1/9$.  Integrating
\eqref{eq:heat-branch-contraction-rate} from $0$ to $\kappa\ge0$ proves
\eqref{eq:heat-branch-contraction-local}.
\end{proof}

The six-point configuration therefore contracts with unequal rates: the
real gap closes faster in relative terms than the upper connector
$c_1c_2$.  This describes a change of shape as well as size in the natural
$T$ coordinate.  Euclidean distances depend on that coordinate; the
cross-ratio certificate below supplies the projectively invariant
comparison with the matched isomonodromic direction.  The contraction
theorem concerns these six discriminant branch values and a local heat
interval.  A global heat assertion about all zeros would require additional
input.

\subsection{A directed obstruction for the matched \texorpdfstring{$D_6$}{D6} trajectory}
Use the cross ratio of \eqref{eq:crossratio-def} and set
\[
 \lambda_1=\lambda(\bar c_1;a,b,c_1),\qquad
 \lambda_2=\lambda(c_2;a,b,c_1).
\]
Let $h_j$ be their derivatives along the natural heat deformation at
$\kappa=0$, and let $d_j$ be their derivatives along the $D_6$
isomonodromic vector field at the phase point of
\cref{thm:D6-certified-realization}.  These derivatives are invariant under
a parameter-dependent projective change of spectral coordinate.

\begin{theorem}[Certified heat--isomonodromic tangent obstruction]\label{thm:heat-D6-obstruction}
For the specific real $D_6$ realization of
\cref{thm:D6-certified-realization}, the tangent determinant satisfies
\begin{equation}\label{eq:heat-D6-determinant}
 \boxed{\mathcal D=h_1d_2-h_2d_1
    \in(7,9)+i(-9,-7),\qquad |\mathcal D|>10.}
\end{equation}
Consequently the natural heat tangent is not proportional, even over
$\CC$, to this isomonodromic tangent.  There is no differentiable local
identification of these curve families through the stated static match,
even modulo moving projective coordinates and a change $t_D=t_D(\kappa)$.
The statement is status \textup{[A]} and uses no zero-free hypothesis
$H_\alpha$.
\end{theorem}

\begin{proof}
The directed program
\path{certificates/heat/} uses Arb complex
Taylor arithmetic \cite{Johansson2016} to evaluate
\[
 Y_0(T)=\frac14\xi\!\left(\frac12+\sqrt T\right).
\]
It refines the selected roots to disks of radius $10^{-18}$ about the
rational centers listed in its source.  For each center $c$ and radius
$r=10^{-18}$ it verifies
\[
 |K_0(c)|+\frac12\sup_{|T-c|\le r}|K_0''(T)|r^2
 <|K_0'(c)|r,
\]
with the ratio of the left to right side less than $1/20000$.
The refined disks lie inside the previously certified branch enclosures,
and $Y_0$ is excluded from zero there.  Rouch\'e's theorem identifies one
simple branch root in each disk.  Conjugation supplies the lower nonreal
pair and the reality of the first two roots.

At time zero, \eqref{eq:heat-Y-evolution} gives the exact jet shift
\[
 \partial_\kappa Y_\kappa^{(j)}\big|_0
 =(2+4j)Y_0^{(j+1)}+4TY_0^{(j+2)}.
\]
The central coefficient and its heat derivative are enclosed using
$M_{2n}(0)=\xi^{(2n)}(1/2)/4$.  Differentiation of the regular numerator
then gives \eqref{eq:heat-branch-velocity} on each refined disk.
For the $D_6$ tangent, the program retains the complete parameter box
\eqref{eq:D6-Xi-parameter-box}, fixes $\theta_0,\theta_\infty$, and
differentiates $N_6$ along \eqref{eq:D6-flow}.  The resulting root velocities
are divided by the affine scale $s$ to express them in the $T$ coordinate;
this does not affect cross-ratio derivatives.

All derivative and cross-ratio denominators are excluded from zero before
division.  Substitution in the logarithmic derivative of
\eqref{eq:crossratio-def} proves the rational rectangle and the strict
modulus bound in \eqref{eq:heat-D6-determinant}.  The full calculation passes
at both 192 and 256 bits, with separate saved logs.  These two runs use
the same directed library and are not independent implementations.
If a local family identification existed, projective invariance and the
chain rule would give $h_j=d_jt_D'(0)$ for both $j$, forcing
$\mathcal D=0$, a contradiction.
\end{proof}

For reference, the directed heat velocities are recorded in
\cref{tab:heat-velocities}.  The component errors in the table are verified
explicitly by the program, rather than inferred from agreement of decimal
outputs.
\begin{table}[ht]
\centering
\caption{Heat branch velocities at $\kappa=0$.  Each displayed real or
imaginary component has absolute error less than $10^{-7}$.}
\label{tab:heat-velocities}
\begin{tabular}{lrr}
\toprule
Branch value & Real component & Imaginary component\\
\midrule
$a$ & $9.4416678$ & $0$\\
$b$ & $-7.5834180$ & $0$\\
$c_1$ & $34.7419121$ & $-7.3959223$\\
$c_2$ & $37.1027126$ & $-19.2730684$\\
\bottomrule
\end{tabular}
\end{table}

\begin{remark}[Scope of the obstruction]\label{rem:heat-D6-scope}
The theorem concerns the natural heat normalization and the specified
static Lax realization with fixed isomonodromic exponents.  It does not
exclude other $D_6$ realizations or a larger deformation in which the
exponents themselves vary.  Together with
\cref{thm:D6-first-order-no-go}, it distinguishes both the ordinary
$x$-deformation and the natural heat deformation from the chosen
isomonodromic trajectory.  The fixed-point curve identity and its Kummer
lift retain their geometric meaning.
\end{remark}

\subsection{An enlarged real-moduli rank test}
\label{sec:enlarged-Lax}
The obstruction above concerns one isomonodromic direction.
We now test a specified enlargement, allowing the two formal exponents
to vary.  This concerns the branch curve, not the full irregular
connection or its Stokes data.

Write $\tau$ for the Lax time in \eqref{eq:D6-flow}, to distinguish it
from heat time.  Near the matched phase point let
\[
 m=(\Imm\lambda_1,\Ree\lambda_2,\Imm\lambda_2).
\]
These are three real local coordinates on the labelled real branch-moduli
space.  Indeed $\Ree\lambda_1=1/2$ and the cross ratio for $\bar c_2$ is
\[
 \lambda_3=\frac{\lambda_1\overline{\lambda_2}}
 {\overline{\lambda_2}+\lambda_1-1};
\]
the denominator is nonzero for the distinct labelled branch values.
Let $v_{\rm iso}$ be the derivative of $m$ along \eqref{eq:D6-flow},
and let $v_0,v_\infty$ be its partial derivatives with respect to
$\theta_0,\theta_\infty$, holding $p,q,\tau$ fixed.

\begin{proposition}[Certified enlarged tangent rank]\label{prop:enlarged-Lax-rank}
At the realization of \cref{thm:D6-certified-realization}, the real matrix
\[
 A=(v_{\rm iso}\ \ v_0\ \ v_\infty)
\]
satisfies
\begin{equation}\label{eq:enlarged-rank}
 \boxed{600<\det A<800.}
\end{equation}
There is therefore a unique real solution of
\[
 \dot m_{\rm heat}
 =\dot\tau\,v_{\rm iso}+\dot\theta_0\,v_0
       +\dot\theta_\infty\,v_\infty.
\]
With derivatives taken with respect to the natural heat time $\kappa$,
its components satisfy
\begin{equation}\label{eq:enlarged-rates}
 0.0007<\dot\tau<0.0009,\qquad
 -0.084<\dot\theta_0<-0.069,\qquad
 0.020<\dot\theta_\infty<0.027.
\end{equation}
\end{proposition}
\begin{proof}
For a simple root $z_j$ of \eqref{eq:D6-sextic-N}, the exponent
velocities are
\[
 \partial_{\theta_\nu}z_j
 =-\frac{\partial_{\theta_\nu}N_6(z_j)}{N_6'(z_j)},
 \qquad \nu\in\{0,\infty\}.
\]
The program \path{certificates/lax_enlarged/cert_enlarged_rank.py}
retains the full realization box \eqref{eq:D6-Xi-parameter-box},
recertifies the four independent branch disks by the Rouch\'e bounds
used above, and excludes every derivative denominator from zero.
It differentiates the displayed sextic by Arb Taylor arithmetic,
transforms velocities by the affine scale, and differentiates the cross
ratios.  The heat jet is recomputed from \eqref{eq:heat-Y-evolution}.
Direct interval expansion of the real $3\times3$ determinant proves
\eqref{eq:enlarged-rank}.  Cramer's rule, with this separated denominator,
proves \eqref{eq:enlarged-rates}.  All strict rational bounds pass at
$192$ and $256$ bits.  The full phase box, rather than a numerical
midpoint realization, is used in both runs.
\end{proof}

\begin{corollary}[Local moving-exponent curve lift]\label{cor:local-Lax-lift}
Through the specified realization, the heat-deformed genus-two branch
curves admit a real-analytic local representation by scalar-normal-form
$D_6$ curves with moving exponents, modulo real projective changes of
spectral coordinate.  No global continuation or preservation of
monodromy or Stokes data is asserted.
\end{corollary}
\begin{proof}
For each fixed nearby exponent pair $\theta$, solve
\eqref{eq:D6-flow} with initial phase values $(p_*,q_*)$ at the same
initial time $\tau_*$.  This gives a real-analytic three-parameter
family $m(\tau,\theta)$, since $q_*\tau_*\ne0$.
At $(\tau_*,\theta_*)$ its three derivative columns are exactly
$v_{\rm iso},v_0,v_\infty$: the initial values $p_*,q_*$ do not depend
on $\theta$.  The real inverse-function theorem and
\cref{prop:enlarged-Lax-rank} now lift the real-analytic heat-moduli path
uniquely within this specified family.  Equality of the labelled branch
moduli gives the local real projective equivalence of the orientation
doubles.
\end{proof}

This local lift does not identify the heat deformation with
\eqref{eq:D6-flow} at fixed exponents.  Along the lifted path the exponents
have nonzero derivatives, so the preceding obstruction remains valid.
Nor does matching the branch curve identify its full scalar potential,
choose a flat heat-time connection, or determine Stokes evolution.
Those are additional deformation problems beyond this rank test.

\subsection{A conditional zero strip and uniform phase positivity}
For $1/2<\alpha\le1$, let $H_\alpha$ mean that every nontrivial zeta zero
has real part at most $\alpha$, with boundary zeros permitted.  The
functional equation gives
\begin{equation}\label{eq:heat-QRH-strip}
 H_\alpha\quad\Longrightarrow\quad
 |\Imm z|\le\alpha-\frac12
 \quad\text{for every zero of }\Xi.
\end{equation}
The $7/8$ result reported in the QRH release and its formalization scope
\cite{OpenAIQRH2026,OpenAIQRHScope2026} is retained here as the explicit
external input $H_{7/8}$.  The present computational supplement does not
replay a Lean build or audit that release's full analytic argument.

\begin{proposition}[Zero-strip phase domain]\label{prop:heat-phase-domain}
If all zeros of a real heat deformation $\Xi_\kappa$ lie in
$|\Imm z|\le h$, then for every real $x$ and $\eta>h$,
\begin{equation}\label{eq:heat-phase-positive}
 -\Imm\frac{\Xi_\kappa'}{\Xi_\kappa}(x+i\eta)
 =\left.\frac{\dd}{\dd s}\arg\Xi_\kappa(s-i\eta)\right|_{s=x}>0.
\end{equation}
In particular, at heat time zero, $H_\alpha$ supplies this uniform domain
for $\eta>\alpha-1/2$; $H_{7/8}$ gives $\eta>3/8$.
\end{proposition}

\begin{proof}
The real even canonical-product representation of the heat deformation
\cite{Polymath2019} permits conjugate-pair grouping of its logarithmic
derivative.  A pair $a\pm ib$, $0\le b\le h$, contributes
\[
 \frac{\eta-b}{(x-a)^2+(\eta-b)^2}
 +\frac{\eta+b}{(x-a)^2+(\eta+b)^2}>0.
\]
A real zero contributes its single positive summand.  The imaginary parts
converge after pairing, and any real constant in the logarithmic derivative
has zero imaginary part.  The even order-one function is nonconstant, so
its canonical product contains a zero and at least one summand is present.
This proves strict positivity.  The phase identity
is the conjugation calculation of \cref{prop:phase-transverse}, which also
applies to $\Xi_\kappa$.
\end{proof}

At $x=0$, positivity for every real $\kappa$ and $T\ge0$ already follows
from the positive theta integral: $Y_\kappa(T)>0$ and
$Y_\kappa'(T)>0$.  The role of $H_\alpha$ is to extend the phase statement
uniformly over all real $x$ above a narrower zero strip.  It supplies no
irrationality statement for the action ratios and no all-discriminant
Humbert exclusion.  In particular, the local heat obstruction above and the
finite-height arithmetic theorems remain independent of this external input.

\section{Period normal form and the zero-action lattice}\label{sec:period-lattice}

The canonical Vinnikov basis used in the Humbert discussion is convenient for
normalized holomorphic periods, but it is not the basis in which the directed
action integrals are evaluated.  The latter use the natural three-cut
symplectic basis $(\alpha_1,\alpha_2,\beta_1,\beta_2)$ described in the next
section.  With the connector winding fixed there, direct conjugation of the
oriented cut cycles gives
\begin{equation}\label{eq:cut-action}
 \boxed{
 \tau_*\alpha=\alpha,
 \qquad
 \tau_*\beta=H_{\rm cut}\alpha-\beta,
 \qquad
 H_{\rm cut}=\begin{pmatrix}1&1\\1&0\end{pmatrix}.}
\end{equation}
This matrix is an integral representative of the same nondividing
$(g,k,a)=(2,1,1)$ real type as the canonical Vinnikov matrix $I_2$; the two
bases are related by an integral symplectic change of basis; the explicit
$U$-congruence and symmetric shear are displayed in the proof of
\cref{prop:v9-period-matrix-N}.  In particular, there is no conflict with
\eqref{eq:Vinnikov-action} or \eqref{eq:v8-Omega-form}.

Define the calibrated core cycle lattice
\begin{equation}\label{eq:Gamma-core}
 \Gamma_{\rm core}
 :=\ZZ\alpha_1\oplus\ZZ\alpha_2\oplus\ZZ\beta_1\oplus\ZZ\beta_2.
\end{equation}
The cycles are contained in the normalized cover over $\Omega_2$ and their
images under the core identification form a symplectic basis of
$H_1(X_2,\ZZ)$.  In accordance with \cref{rem:action-scope}, define the action
homomorphism only on these calibrated representatives,
\[
 \mathcal A_{\rm core}(\gamma)=\oint_\gamma\varpi_0,
 \qquad \gamma\in\Gamma_{\rm core},
\]
and put
\begin{equation}\label{eq:uv-def}
 u_j=\mathcal A_{\rm core}(\alpha_j),
 \qquad
 v_j=\mathcal A_{\rm core}(\beta_j).
\end{equation}

\begin{proposition}[Real normal form in the cut gauge]\label{prop:period-shape}
There exist $r_1,r_2\in\RR$ such that, with
$u=(u_1,u_2)^{\mathsf T}$ and $r=(r_1,r_2)^{\mathsf T}$,
\begin{equation}\label{eq:period-shape}
 \boxed{
 u\in\RR^2,
 \qquad
 v=\frac12H_{\rm cut}u+i r.}
\end{equation}
Equivalently,
\[
 v_1=\frac{u_1+u_2}{2}+ir_1,
 \qquad
 v_2=\frac{u_1}{2}+ir_2.
\]
\end{proposition}

\begin{proof}
Since $\tau^*\varpi_0=\overline{\varpi_0}$,
\eqref{eq:cut-action} gives
\[
 \overline u=u,
 \qquad
 \overline v=H_{\rm cut}u-v.
\]
Thus $u$ is real and $2\Ree v=H_{\rm cut}u$.
\end{proof}

Put
\begin{equation}\label{eq:real-lattices}
 \Lambda_u=\{a\in\ZZ^2:a\cdot u=0\},
 \qquad
 \Lambda_r=\{b\in\ZZ^2:b\cdot r=0\}.
\end{equation}
The calibrated zero-action lattice is
\[
 \Lambda_2=\ker\mathcal A_{\rm core}\subset\Gamma_{\rm core}.
\]

\begin{theorem}[Kernel factorization]\label{thm:kernel-factorization}
In the three-cut basis,
\begin{equation}\label{eq:kernel-factorization}
 \boxed{
 \Lambda_2=
 \left\{(m,n)\in\ZZ^2\oplus\ZZ^2:
 (2m+H_{\rm cut}n)\cdot u=0,\quad n\cdot r=0\right\}.}
\end{equation}
Consequently
\begin{equation}\label{eq:kernel-iff}
 \boxed{
 \Lambda_2=0
 \quad\Longleftrightarrow\quad
 \Lambda_u=0\ \text{and}\ \Lambda_r=0.}
\end{equation}
\end{theorem}

\begin{proof}
For
\[
 \gamma=m_1\alpha_1+m_2\alpha_2+n_1\beta_1+n_2\beta_2,
\]
\cref{prop:period-shape} gives
\[
 \oint_\gamma\varpi_0
 =\left(m+\frac12H_{\rm cut}n\right)\cdot u+i\,n\cdot r,
\]
because $H_{\rm cut}$ is symmetric.  Vanishing of real and imaginary parts is
exactly \eqref{eq:kernel-factorization}.  If both real relation lattices
vanish, then $n=0$ and then $m=0$.  Conversely,
$0\ne a\in\Lambda_u$ gives $(a,0)\in\Lambda_2$, while
$0\ne b\in\Lambda_r$ gives
\[
 (m,n)=(-H_{\rm cut}b,2b)\in\Lambda_2.
\]
\end{proof}

\begin{corollary}[Two-irrationality criterion]\label{cor:irrationality-criterion}
If $u_1u_2r_1r_2\ne0$, then
\[
 \boxed{
 \Lambda_2=0
 \quad\Longleftrightarrow\quad
 \frac{u_1}{u_2}\notin\QQ
 \quad\text{and}\quad
 \frac{r_1}{r_2}\notin\QQ.}
\]
\end{corollary}

\begin{corollary}[Explicit resonant cycles]\label{cor:resonant-cycles}
If $u_1/u_2=p/q\in\QQ$ in lowest terms, then
\[
 q\alpha_1-p\alpha_2\in\Lambda_2\setminus\{0\}.
\]
If $r_1/r_2=p/q\in\QQ$ in lowest terms, then
\[
 (p-q)\alpha_1-q\alpha_2+2q\beta_1-2p\beta_2
 \in\Lambda_2\setminus\{0\}.
\]
Both displayed cycles are primitive.
\end{corollary}

\begin{proof}
The first statement is immediate.  For the second take
$b=(q,-p)$, so $b\cdot r=0$, and use
$(m,n)=(-H_{\rm cut}b,2b)$ from the proof of
\cref{thm:kernel-factorization}.  Since
$H_{\rm cut}b=(q-p,q)$, the displayed cycle follows.  Its coefficients have
greatest common divisor one because $\gcd(p,q)=1$.
\end{proof}

This gives a precise answer to what the failure of nonresonance would mean.  It
would not invalidate the spectral curve; it would select an explicit primitive
calibrated core-cycle direction on which the additive action closes.  In the wider
monodromy framework, that direction is exactly where a residual multiplicative
character must then be evaluated.

\section{Calibrated action coordinates}\label{sec:calibration}

Choose three disjoint cuts joining the branch pairs
\[
 \Gamma_0:a\leftrightarrow b,
 \qquad
 \Gamma_1:c_1\leftrightarrow\bar c_1,
 \qquad
 \Gamma_2:c_2\leftrightarrow\bar c_2.
\]
Let $\alpha_1,\alpha_2$ be the positive circuits around $\Gamma_0,\Gamma_1$, and let $\beta_1,\beta_2$ be connecting cycles from the first two cuts to the reference cut $\Gamma_2$.  The connector winding is fixed in the conjugation-symmetric gauge so that this cut basis realizes \eqref{eq:cut-action}.

Put
\begin{equation}\label{eq:y-plus}
 y_+(T)=\frac{\sqrt{K(T)}}{Y(T)^3}
\end{equation}
with square root continued on the selected sheet of the cut complement.  Let $\gamma_0$ and $\gamma_1$ denote selected positive sides of $\Gamma_0$ and $\Gamma_1$, and let $\ell_1,\ell_2$ be the calibrated connecting arcs.

\begin{proposition}[Explicit theta-moment periods]\label{prop:explicit-periods}
The real period coordinates are
\begin{align}
 u_1&=2\int_{\gamma_0}y_+(T)\,\dd T,
 &u_2&=2\int_{\gamma_1}y_+(T)\,\dd T,\label{eq:u-integrals}\\
 r_j&=2\Imm\int_{\ell_j}y_+(T)\,\dd T,
 &&j=1,2.\label{eq:r-integrals}
\end{align}
For later arithmetic aggregation set simply
\begin{equation}\label{eq:d-def}
 d:=r_1-r_2.
\end{equation}
\end{proposition}

\begin{proof}
An $\alpha_j$-cycle goes once along the chosen positive side of its cut and
returns along the opposite sheet.  Sheet exchange changes the sign of $y_+$
while reversal of orientation changes the sign of $\dd T$, so the two
contributions agree and give the factor $2$ in \eqref{eq:u-integrals}.
The calibrated $\beta_j$-cycle is obtained from the connector $\ell_j$ and its
conjugate return path, together with the cut-side pieces fixed by the chosen
winding.  Using $\tau^*\varpi_0=\overline{\varpi_0}$ and the homological action
\eqref{eq:cut-action}, those cut-side pieces account exactly for
$\tfrac12(H_{\rm cut}u)_j$, while the connector and conjugate connector give
$2i\,\Imm\int_{\ell_j}y_+(T)\,\dd T$.  Comparing with
\cref{prop:period-shape} proves \eqref{eq:r-integrals}.  The definition of
$d$ is then immediate.
\end{proof}

The formulas make all four real quantities functions of the common theta-moment data.  This shared dependence is important computationally: independent interval boxes discard correlations, while a future affine/Taylor-model treatment of a linear combination can retain them.

\subsection{Exact heat variation of the calibrated actions}
Transport the cut cycles locally as in
\cref{prop:heat-core-persistence}, and set
$\varpi_\kappa=\sqrt{\Delta_\kappa(T)}\,\dd T$ on the continued sheet.
In this subsection a dot means $\partial_\kappa|_{\kappa=0}$ at fixed
$T$, with the central normalization $A_*(\kappa)$ updated.  Thus
\begin{equation}\label{eq:heat-action-differential}
 \dot\varpi
 =\frac{\dot\Delta}{2\sqrt{\Delta_0}}\,\dd T,
 \qquad
 \dot\Delta=\frac{\dot K}{Y^6}-\frac{6K\dot Y}{Y^7},
 \qquad \dot Y=2Y'+4TY''.
\end{equation}

\begin{proposition}[Differentiation of transported actions]\label{prop:heat-action-variation}
The calibrated action coordinates are real analytic in real heat time
near zero.  For a locally transported calibrated cycle $\gamma$,
\begin{equation}\label{eq:heat-action-period-derivative}
 \frac{\dd}{\dd\kappa}\oint_{\gamma(\kappa)}\varpi_\kappa
 \bigg|_0
 =\oint_{\gamma(0)}\frac{\dot\Delta(T)}
                         {2\sqrt{\Delta_0(T)}}\,\dd T.
\end{equation}
The corresponding open-cut formulas are differentiated with the same
fixed sheet and orientation.  Moving simple branch endpoints contribute
no boundary term, and their inverse-square-root singularities in
$\dot\varpi$ are integrable.
\end{proposition}

\begin{proof}
On fixed representatives away from the branch values, the continued
square root is analytic in the heat parameter and differentiation under
the integral proves \eqref{eq:heat-action-period-derivative}.  For a
parametrized moving path $T(\kappa,s)$, the contribution from its motion
is the total derivative
$\partial_s(\sqrt{\Delta_\kappa}\,\partial_\kappa T)$; it cancels on
a closed cycle.  At a simple branch endpoint $T_j(\kappa)$, the local
factorization
\[
 \Delta_\kappa(T)=(T-T_j(\kappa))g_\kappa(T),\qquad
 g_\kappa(T_j(\kappa))\ne0,
\]
shows that $\sqrt{\Delta_\kappa}$ vanishes there and that its fixed-$T$
heat derivative is $O((T-T_j)^{-1/2})$.  Hence the open-cut endpoint
term vanishes and the differentiated integrand is integrable.  In the
local coordinate $T-T_j(\kappa)=w^2$, the action differential is analytic
in $(\kappa,w)$, proving analytic dependence of the actions.  Reality of
the transported differential and the unchanged conjugation action give
real analyticity of $u_1,u_2,d,r_2$.
\end{proof}

Write
\begin{equation}\label{eq:heat-action-ratios}
 \rho_u(\kappa)=\frac{u_1(\kappa)}{u_2(\kappa)},\qquad
 \rho_r(\kappa)=-\frac{r_1(\kappa)}{r_2(\kappa)}
              =-1-\frac{d(\kappa)}{r_2(\kappa)}.
\end{equation}
The exact quotient derivatives are
\begin{equation}\label{eq:heat-action-ratio-derivatives}
 \dot\rho_u=\frac{u_2\dot u_1-u_1\dot u_2}{u_2^2},\qquad
 \dot\rho_r=\frac{d\dot r_2-r_2\dot d}{r_2^2}.
\end{equation}
These are ratios of the calibrated action integrals of $\varpi_\kappa$.
The normalized holomorphic Jacobian period matrix of
\cref{prop:v9-period-matrix-N} is a separate object.

\begin{theorem}[Certified positive motion of both action ratios]\label{thm:heat-action-motion}
At heat time zero, the derivatives satisfy the strict rational enclosures
in \cref{tab:heat-action-motion}.  In particular,
\begin{equation}\label{eq:heat-action-positive-ratios}
 \boxed{\dot\rho_u\in(0.0033792688,0.0033792692),\qquad
        \dot\rho_r\in(0.0021198,0.0021200).}
\end{equation}
Both ratios are strictly increasing on a sufficiently small real heat
interval about zero.  The statement is status \textup{[A]} and uses no
QRH input.
\end{theorem}

\begin{table}[htbp]
\centering
\caption{Strict directed heat-derivative enclosures at $\kappa=0$.
Here $d=r_1-r_2$ and $t_{\rm Polymath}=4\kappa$.}
\label{tab:heat-action-motion}
\begin{tabular}{ccc}
\toprule
Quantity & Lower bound for $\partial_\kappa|_0$ & Upper bound for $\partial_\kappa|_0$\\
\midrule
$u_1$ & $0.0041190077$ & $0.0041190079$\\
$u_2$ & $0.036789365$ & $0.036789369$\\
$d$ & $-0.030214589$ & $-0.030214585$\\
$r_2$ & $0.010297490$ & $0.010297494$\\
$\rho_u=u_1/u_2$ & $0.0033792688$ & $0.0033792692$\\
$\rho_r=-r_1/r_2$ & $0.0021198$ & $0.0021200$\\
\bottomrule
\end{tabular}
\end{table}

\begin{proof}
The certificate in \path{certificates/period_heat/}
and its accompanying Taylor engine integrate
\eqref{eq:heat-action-differential} on the original calibrated contours:
the real segment $a$ to $b$, the straight cut $c_1$ to $\bar c_1$,
the connector $10$ to $c_1$ for $d$, and the upper connector $c_1$ to
$c_2$ for $r_2$.  The real or imaginary projection and its sign are
fixed by the original gauge and transported unchanged.  The new base
period boxes lie strictly inside all four retained enclosures
\eqref{eq:u-boxes}--\eqref{eq:r-boxes}; none of the sign choices meets zero.

The engine encloses the even central coefficients of $\xi$ at 2048 bits.
It uses a degree-85 polynomial for $Y$ and degree-80 polynomials for the
entire functions $E$ and $\dot E$ of \cref{prop:regular-K}.  Their
removable quotients by $T$ are performed by coefficient shifting, using
the exact identities for $A_*$ and $\dot A_*$.  Analytic bounds on
$|T|=4096$ control every omitted coefficient by Cauchy's estimate.
For $Y$ and its derivatives these bounds follow from the positive
theta-moment coefficients; for $E$ and $\dot E$ they follow from their
regular numerators and triangle inequalities on that circle.

On each interior panel, an outer disk excludes zeros of $Y$ and bounds
$\Delta_0$ and $\dot\Delta$; a nested inner disk excludes zeros of
$\Delta_0$.  Taylor remainders are bounded, and the same square-root
sheet is propagated by overlapping image disks that exclude the
opposite sheet.  Order-24 symmetric Taylor integration encloses both the base
integral and its heat derivative.  Path subdivisions use exact rational
parameters.  The endpoint cutoff is a fraction $10^{-18}$ of each
segment, and the actual endpoints lie in the refined radius-$10^{-18}$
root disks.  The omitted integrals use rigorous bounds of orders
$\ell^{3/2}$ and $\ell^{1/2}$, respectively, as detailed in
\cref{sec:heat-period-certification}.

Directed aggregation in \eqref{eq:heat-action-ratio-derivatives}, with
$u_2$ and $r_2$ excluded from zero, proves every strict decimal enclosure
in the table.  Both the 512-bit and 768-bit replays pass these comparisons.
Real analyticity and the strictly positive initial margins give local
strict monotonicity of both ratios.  No numerical length for this local
interval is inferred from the time-zero certificate.
\end{proof}

The geometry and the actions thus move in different senses: all fifteen
branch distances shrink, while both calibrated action ratios increase.
In particular, $\dot r_1=\dot d+\dot r_2<0$ and $\dot r_2>0$, so the
second ratio records a small positive imbalance between the growth of
$|r_1|$ and that of $r_2$.  Its sign is certified after the cancellation
in \eqref{eq:heat-action-ratio-derivatives}.

\subsection{Local generic nonresonance and dense resonant times}
For real heat time near zero, let $\Lambda_2(\kappa)$ be the kernel of
the calibrated action homomorphism on the transported core cycles.  A $G_\delta$ set is a countable intersection of open sets; ``full measure'' below refers to Lebesgue measure in heat time.

\begin{corollary}[Generic local heat nonresonance]\label{cor:heat-generic-nonresonance}
There is an interval $I=(-\varepsilon,\varepsilon)$ on which all four
coordinates $u_1,u_2,r_1,r_2$ are nonzero and both ratios in
\eqref{eq:heat-action-ratios} are strictly increasing.  On this interval,
\begin{equation}\label{eq:heat-resonance-set}
 \boxed{\{\kappa\in I:\Lambda_2(\kappa)\ne0\}
       =\rho_u^{-1}(\QQ)\cup\rho_r^{-1}(\QQ).}
\end{equation}
This resonant set is countable and dense in $I$.  Its complement is a
dense $G_\delta$ set of full Lebesgue measure, and
$\Lambda_2(\kappa)=0$ at every point of that complement.
No membership assertion for the specified heat time $\kappa=0$ follows.
\end{corollary}

\begin{proof}
The time-zero period boxes give nonzero coordinates, and
\cref{thm:heat-action-motion} gives local strict monotonicity.  Reduce
$\varepsilon$ to preserve both properties.  The conjugation matrix
$H_{\rm cut}$ is unchanged under local cycle transport, so
\cref{cor:irrationality-criterion} applies at each real heat time in
$I$ and proves \eqref{eq:heat-resonance-set}.  A continuous strictly
monotone ratio maps every nonempty open subinterval of $I$ onto an
interval with nonempty interior.  Rational values are dense in that
image, and each rational value has at most one preimage.  Each of the
two preimage sets is therefore countable and dense.  Their union has
the same properties.  Removing a countable set from an interval gives
a dense $G_\delta$ set of full Lebesgue measure.
\end{proof}

This is an all-denominator statement for generic local heat parameters,
while the arithmetic specialization at zero remains
\cref{conj:nonresonance}.  The dense resonant times are compatible with
the locally persistent Farey windows: any rational values occurring
inside those windows must satisfy the same large denominator barriers.
The result concerns the action lattice; it supplies no irrationality
certificate for the holomorphic Jacobian periods or the Picard rank.

\subsection{Simple resonances and generic Diophantine bounds}
Shrink the interval $I$ of \cref{cor:heat-generic-nonresonance} if
necessary so that its closure lies in the analytic continuation domain,
$u_2,r_2>0$, and both ratio derivatives have positive lower bounds on
$I$.  This is possible by \cref{thm:heat-action-motion} and continuity.
Continue to write $I=(-\varepsilon,\varepsilon)$ for this smaller interval.
Identify the transported core lattice with the fixed integer lattice
using the calibrated cycles.  For
\[
 \gamma=m\cdot\alpha+n\cdot\beta,
 \qquad m,n\in\ZZ^2,
\]
write its analytic action as
\begin{equation}\label{eq:heat-cycle-action}
 \mathcal A_\gamma(\kappa)
 =\left(m+\frac12H_{\rm cut}n\right)\cdot u(\kappa)
     +i\,n\cdot r(\kappa).
\end{equation}

\begin{proposition}[At most one simple resonance for each fixed cycle]
\label{prop:heat-simple-resonances}
For each fixed $\gamma\ne0$, the equation
$\mathcal A_\gamma(\kappa)=0$ has at most one real solution in $I$.
At any such solution $\kappa_0$,
\[
 \mathcal A_\gamma'(\kappa_0)\ne0.
\]
In particular, no fixed nonzero core cycle is resonant throughout an
interval.  The rational-ratio crossings in the two real channels are
simple: for $q>0$,
\begin{align}
 F_{p,q}&=q u_1-p u_2,
 &F_{p,q}'\big|_{\rho_u=p/q}&=q u_2\rho_u'>0,
 \label{eq:heat-u-simple-crossing}\\
 G_{p,q}&=q r_1+p r_2,
 &G_{p,q}'\big|_{\rho_r=p/q}&=-q r_2\rho_r'<0.
 \label{eq:heat-r-simple-crossing}
\end{align}
\end{proposition}

\begin{proof}
If $n\ne0$, vanishing of the imaginary part of
\eqref{eq:heat-cycle-action} forces $n_1\ne0$ and
\[
 n\cdot r=r_2(n_2-n_1\rho_r)=0,
 \qquad \rho_r=\frac{n_2}{n_1}.
\]
Strict monotonicity gives at most one such time.  At that time,
\[
 \partial_\kappa(n\cdot r)=-n_1r_2\rho_r'\ne0,
\]
so the full complex action has nonzero derivative whenever it vanishes.
If $n=0$, then $m\ne0$ and
$\mathcal A_\gamma=u_2(m_1\rho_u+m_2)$.  A zero forces $m_1\ne0$
and $\rho_u=-m_2/m_1$, again at most once, with derivative
$m_1u_2\rho_u'\ne0$.  Finally,
$F_{p,q}=u_2(q\rho_u-p)$ and
$G_{p,q}=r_2(p-q\rho_r)$ give
\eqref{eq:heat-u-simple-crossing}--\eqref{eq:heat-r-simple-crossing}.
\end{proof}

Thus the countable dense resonant set is assembled from simple zeros of
different integer-cycle actions.  Its density does not require repeated
resonance of any one fixed cycle.

For $\gamma\ne0$, define its coefficient size in this fixed calibration by
\begin{equation}\label{eq:heat-cycle-size}
 N(\gamma)=\max\{|m_1|,|m_2|,|n_1|,|n_2|\}\ge1.
\end{equation}

\begin{theorem}[Generic Diophantine separation of the calibrated actions]
\label{thm:heat-generic-Diophantine}
There is a full-Lebesgue-measure subset $I_{\rm Dio}\subset I$ such that,
for every $\kappa\in I_{\rm Dio}$ and every $\delta>0$, there are
constants $c_{\kappa,\delta},C_{\kappa,\delta}>0$ with
\begin{equation}\label{eq:heat-generic-rational-bound}
 \boxed{\left|\rho_j(\kappa)-\frac pq\right|
 \ge\frac{c_{\kappa,\delta}}{q^{2+\delta}}
 \quad(j=u,r;\ p\in\ZZ,\ q\in\ZZ_{\ge1})}
\end{equation}
and
\begin{equation}\label{eq:heat-generic-action-bound}
 \boxed{|\mathcal A_\gamma(\kappa)|
 \ge\frac{C_{\kappa,\delta}}{N(\gamma)^{1+\delta}}
 \qquad(\gamma\in\Gamma_{\rm core}\setminus\{0\}).}
\end{equation}
The quantifier is almost every heat time, with constants depending on
that time; no assertion that $0\in I_{\rm Dio}$ is made.
\end{theorem}

\begin{proof}
Both ratios have bounded images on $I$ and derivatives bounded below
by a common $b>0$.  Fix $\delta>0$.  For each $j=u,r$ and $q\ge1$,
consider
\[
 B_{j,q}(\delta)=
 \bigcup_{p\in\ZZ}
 \left\{\kappa\in I:
 \left|\rho_j(\kappa)-\frac pq\right|<q^{-2-\delta}\right\}.
\]
Only $O(q)$ numerators meet the bounded image enlarged by one.  By
monotonicity and the derivative lower bound, each preimage interval
has length at most $2b^{-1}q^{-2-\delta}$.  Therefore
\[
 \operatorname{meas} B_{j,q}(\delta)=O(q^{-1-\delta}),
 \qquad \sum_{q\ge1}\operatorname{meas}B_{j,q}(\delta)<\infty.
\]
The measure of the union over $q\ge Q$ tends to zero as $Q\to\infty$.
Consequently the limsup set has measure zero: almost every heat time
belongs to only finitely many of these approximation sets.

Apply this argument to both ratios and the countable exponents
$\delta=1/\ell$, $\ell\ge1$, and also remove the countable resonant
set of \cref{cor:heat-generic-nonresonance}.  The resulting set
$I_{\rm Dio}$ has full measure.  For any $\kappa$ in this set and any
$\delta>0$, choose $1/\ell<\delta$.  The estimate with exponent
$2+1/\ell$ gives \eqref{eq:heat-generic-rational-bound} for all large
$q$ with constant one.  For the finitely many remaining denominators,
irrationality makes the distance to the nearest $p/q$ strictly positive.
Reducing the constant, and taking the smaller constant for the two
ratios, proves the bound for all $p,q$.

For the action estimate, if $n_1\ne0$, the imaginary part gives
\[
 |\mathcal A_\gamma|
 \ge |n\cdot r|
 =r_2|n_2-n_1\rho_r|
 \ge r_2c_{\kappa,\delta}|n_1|^{-1-\delta}.
\]
If $n_1=0$ but $n_2\ne0$, it instead gives
$|\mathcal A_\gamma|\ge r_2|n_2|\ge r_2$.
If $n=0$ and $m_1\ne0$, the real part gives
$|\mathcal A_\gamma|\ge u_2c_{\kappa,\delta}|m_1|^{-1-\delta}$;
the remaining case $m_1=0,m_2\ne0$ gives
$|\mathcal A_\gamma|\ge u_2$.  Since each coefficient is at most
$N(\gamma)$ in absolute value, all cases imply
\eqref{eq:heat-generic-action-bound} with
\[
 C_{\kappa,\delta}
 =\min\{u_2(\kappa),r_2(\kappa)\}
       \min\{1,c_{\kappa,\delta}\}>0.
\]
\end{proof}

These constants cannot have a common positive lower bound as
$\kappa$ varies through $I_{\rm Dio}$ in any nonempty open subinterval
of $I$.  Such a subinterval contains a resonant time and a fixed cycle
whose action vanishes there.  Full measure makes $I_{\rm Dio}$ dense,
so that action tends to zero along nearby points of $I_{\rm Dio}$,
while its coefficient size remains fixed.  This also explains why
generic quantitative separation cannot simply be continued to zero.

\subsection{Arithmetic specialization and separation from zero reality}
\Cref{cor:heat-generic-nonresonance} proves a generic local heat version
of \cref{conj:nonresonance}; the conjecture asks whether the specified
time $\kappa=0$ lies outside the exceptional set.  A countable exceptional
set can contain that point.  Even positive analytic derivatives do not
exclude this possibility: $f(\kappa)=1+\kappa$ is irrational for almost
every real $\kappa$, yet $f(0)=1$.  A fixed-point bound of the form
\eqref{eq:heat-generic-rational-bound} at zero would prove the conjecture,
but would be stronger than mere irrationality and is not established here.

The difference from the global zero-reality question can be made rigorous
using the established lower bound for the de Bruijn--Newman constant.

\begin{corollary}[Nonresonant cores with nonreal heat zeros]
\label{cor:heat-nonresonance-nonreal-zeros}
For almost every $\kappa\in(-\varepsilon,0)$ in the local interval above,
\begin{equation}\label{eq:heat-nonresonant-nonreal}
 \boxed{\Lambda_2(\kappa)=0
 \quad\text{and}\quad
 \Xi_\kappa\ \text{has a nonreal zero}.}
\end{equation}
At almost every such negative time the stronger action bound
\eqref{eq:heat-generic-action-bound} also holds for every $\delta>0$.
Thus local action nonresonance, even with generic Diophantine separation,
does not imply reality of all zeros throughout this heat family.
\end{corollary}

\begin{proof}
In the Polymath normalization, $H_t$ has only real zeros precisely when
$t\ge\Lambda$ \cite{Polymath2019}, and Rodgers--Tao proved
$\Lambda\ge0$ \cite{RodgersTao2020}.  If $\kappa<0$, then
$4\kappa<0\le\Lambda$, so $H_{4\kappa}$ has a nonreal zero.
Equation~\eqref{eq:heat-normalization},
$\Xi_\kappa(z)=8H_{4\kappa}(2z)$, preserves the reality or nonreality
of zeros.  This gives a nonreal zero for every negative $\kappa$.
Intersecting the negative local interval with the full-measure sets of
\cref{cor:heat-generic-nonresonance,thm:heat-generic-Diophantine}
proves the two conclusions.
\end{proof}

This is a statement about the heat-deformed family.  It decides neither
\cref{conj:nonresonance} at $\kappa=0$ nor RH, and establishes no
logical equivalence between them.  If RH holds, the same zero-reality
criterion and $\Lambda\ge0$ give $\Lambda=0$, so every positive heat
time has only real zeros.  Nevertheless,
\cref{cor:heat-generic-nonresonance} still supplies a countable dense set
of positive local heat times with resonant core cycles.  Hence, conditional
on RH, reality of all zeros also coexists with local action resonance.

The QRH input $H_\alpha$ in \cref{prop:heat-phase-domain} has a different
role: it narrows the zero strip and gives uniform phase positivity for
$\eta>\alpha-1/2$, including $\eta>3/8$ under $H_{7/8}$.  No estimate
of the form \eqref{eq:heat-generic-rational-bound} at the specified
arithmetic point follows from the arguments of that proposition.
An implication from RH or QRH to fixed-point action irrationality would
require an additional theorem connecting zero geometry to these particular
calibrated integrals.

\section{Deep directed Farey separation}\label{sec:farey}

\begin{proposition}[Directed period and linear-form boxes]\label{prop:period-boxes}
The calibrated action coordinates satisfy the outward enclosures
\begin{align}
 0.0282613617668264&<u_1<0.0282613617671469,\notag\\
 0.8619540619015512&<u_2<0.8619540619074953,\label{eq:u-boxes}\\
 -0.3351510384592472&<d<-0.3351510384566369,\notag\\
 0.1151538468901364&<r_2<0.1151538468923606.\label{eq:r-boxes}
\end{align}
The directly aggregated linear forms satisfy
\begin{align}
 419336u_1-13749u_2
 &\in[6.87762\times10^{-7},9.03886\times10^{-7}],\notag\\
 616667u_1-20219u_2
 &\in[-1.04412\times10^{-6},-7.26286\times10^{-7}],\label{eq:u-forms}\\
 8678d+25257r_2
 &\in[-8.45173\times10^{-7},-7.66343\times10^{-7}],\notag\\
 132091d+384446r_2
 &\in[4.04956\times10^{-7},1.60484\times10^{-6}].\label{eq:r-forms}
\end{align}
\end{proposition}

\begin{proof}
The common theta moments are enclosed with outward-rounded symmetric Taylor quadrature and analytic theta/outer tails.  The removable coordinate at $T=0$ is integrated from the regular form $K/Y^6$.  On the remaining contour pieces, the selected square root is propagated by Taylor panels inside certified zero-free disks, with simple-root endpoint tails treated separately.  The nonreal endpoints use the $10^{-9}$ disks of \cref{thm:c1-pair,thm:c2-pair}.  All conversions from interval endpoints to multiprecision centers and radii are performed directly from the interval endpoint tuples, with no intermediate binary \texttt{float} narrowing.  The deep endpoint cut is $10^{-6}$.  The direct certificates first produce boxes strictly inside \eqref{eq:u-boxes}--\eqref{eq:r-boxes}; the displayed decimal intervals are outward-rounded supersets and are then inserted into the integer forms by exact decimal arithmetic.  The supplied period engine in \path{certificates/period_heat/} independently checks containment in the displayed calibration boxes, and the uniform engine in \path{certificates/finite_heat/} verifies the Farey inequalities throughout the stated heat interval.
\end{proof}

\begin{theorem}[Certified Farey windows]\label{thm:farey}
The period ratios satisfy
\begin{equation}\label{eq:farey-u}
 \boxed{
 \frac{13749}{419336}<\frac{u_1}{u_2}<\frac{20219}{616667}}
\end{equation}
and
\begin{equation}\label{eq:farey-r}
 \boxed{
 \frac{16579}{8678}<-\frac{r_1}{r_2}<\frac{252355}{132091}}.
\end{equation}
\end{theorem}

\begin{proof}
Since $u_2>0$, \eqref{eq:u-forms} is exactly equivalent to \eqref{eq:farey-u}.  Since $r_1=d+r_2$,
\[
 -\frac{r_1}{r_2}=-1-\frac d{r_2}.
\]
For positive integers $p,q$,
\[
 -\frac{r_1}{r_2}>\frac pq
 \quad\Longleftrightarrow\quad
 qd+(p+q)r_2<0.
\]
The two inequalities in \eqref{eq:r-forms} therefore give \eqref{eq:farey-r}.
\end{proof}

\begin{corollary}[Denominator barriers]\label{cor:denominator-barriers}
If $u_1/u_2=p/q$ is reduced with $q>0$, then
\[
 \boxed{q\ge1{,}036{,}003},
\]
and equality forces
\[
 \frac pq=\frac{33968}{1036003}.
\]
If $-r_1/r_2=p/q$ is reduced with $q>0$, then
\[
 \boxed{q\ge140{,}769},
\]
and equality forces
\[
 \frac pq=\frac{268934}{140769}.
\]
\end{corollary}

\begin{proof}
For the first interval,
\[
 20219\cdot419336-13749\cdot616667=1,
\]
so the endpoints are Farey neighbors.  Every reduced fraction strictly between them has denominator at least
\[
 419336+616667=1{,}036{,}003,
\]
with equality only for the mediant.  Likewise
\[
 252355\cdot8678-16579\cdot132091=1,
\]
which gives $q\ge8678+132091=140769$ and the displayed mediant.
\end{proof}

The next two mediant forms are not separated by the present independent boxes:
\begin{align}
 -3.56350\times10^{-7}
 &<1036003u_1-33968u_2
 <1.77599\times10^{-7},\label{eq:next-u}\\
 -4.40216\times10^{-7}
 &<140769d+409703r_2
 <8.38496\times10^{-7}.\label{eq:next-r}
\end{align}
Both intervals contain zero, marking the first unresolved mediants for these independent period boxes.

\begin{proposition}[Finite-width limitation]\label{prop:finite-height}
No finite-width interval enclosure of either period ratio can by itself prove irrationality.  It can exclude rational resonances only up to a finite denominator height.
\end{proposition}

\begin{proof}
Every nonempty open interval contains reduced rational numbers with arbitrarily large denominators.  For each fixed $Q$, only finitely many rationals of denominator at most $Q$ lie in a bounded interval; removing them leaves a nonempty open set containing rationals of larger denominator.
\end{proof}

Thus the natural frontier is not simply ``more digits.''  One can try to decide \eqref{eq:next-u}--\eqref{eq:next-r} by preserving the common theta-moment dependence directly in the linear forms, but even success would move only one more finite Farey rung.  A proof of \cref{conj:nonresonance} requires an all-denominator mechanism.

\subsection{An explicit heat interval for both Farey windows}
\begin{theorem}[Certified heat persistence]\label{thm:explicit-Farey-heat}
Put $h=3\times10^{-10}$ and $I_h=(-h,h)$.  The six selected simple
branch values and their calibrated core cycles continue throughout
$[-h,h]$, with $u_1,u_2,r_2>0$ and $r_1=d+r_2<0$.
For every real $|\kappa|\le h$, both Farey windows
\eqref{eq:farey-u}--\eqref{eq:farey-r} hold with the transported
action periods.  In particular, any rational values there have reduced
denominators at least $1{,}036{,}003$ and $140{,}769$, respectively.

On this same interval, directed uniform bounds give
\begin{equation}\label{eq:uniform-heat-ratio-speeds}
 \boxed{0.0033<\rho_u'(\kappa)<0.0035,\qquad
        0.0001<\rho_r'(\kappa)<0.0045.}
\end{equation}
Thus \cref{cor:heat-generic-nonresonance,prop:heat-simple-resonances,thm:heat-generic-Diophantine}
apply with the explicit interval $I_h$,
and \cref{cor:heat-nonresonance-nonreal-zeros} applies on $(-h,0)$.
This statement is status \textup{[A]} and uses no RH or QRH input.
\end{theorem}

\begin{proof}
The certificate in \path{certificates/finite_heat/}
uses directed Taylor models uniform in the real parameter
$|\kappa|\le10^{-9}$.  Its Rouch\'e bounds isolate one simple root
in each radius-$10^{-7}$ disk about the four refined centers in
\cref{thm:heat-branch-contraction}; conjugation supplies the other two.
The four Rouch\'e ratios are all less than $0.419$.
The disks are disjoint, and $Y_\kappa$ is excluded from zero there.
Uniform zero-free panel disks along the calibrated contours continue
the original square-root sheets, so the original cut calibration and
conjugation matrix are transported without a change of orientation.

The engine keeps the heat parameter shared through the construction
of $K$ and $\partial_\kappa K$ before interval evaluation.
It integrates \eqref{eq:heat-action-differential} on the same four
contours as \cref{thm:heat-action-motion}, with fixed interior panels
and moving-root endpoint bounds
\eqref{eq:heat-period-base-tail}--\eqref{eq:heat-period-derivative-tail}.
The uniform endpoint cutoff is a fraction $10^{-8}$ of each segment;
the root-location allowance is the full radius $10^{-7}$.
Fresh time-zero integrations reproduce all four original calibration
boxes and all retained time-zero derivative enclosures.

For each period coordinate $P$, the mean-value enclosure
$P(\kappa)\in P(0)+[-10^{-9},10^{-9}]\,\partial_\kappa P$
proves the required signs and excludes the ratio denominators from
zero.  Inserting these boxes and the uniform period derivatives into
\eqref{eq:heat-action-ratio-derivatives} proves
\eqref{eq:uniform-heat-ratio-speeds}.
For each of the four Farey forms, integration of its combined heat
derivative gives an enclosure on $[-h,h]$.  The following strict
outward supersets are verified directly:
\begin{align}
 419336u_1-13749u_2&\in(10^{-7},2\times10^{-6}),\notag\\
 616667u_1-20219u_2&\in(-2\times10^{-6},-10^{-7}),\notag\\
 8678d+25257r_2&\in(-10^{-6},-5\times10^{-7}),\notag\\
 132091d+384446r_2&\in(5\times10^{-7},2\times10^{-6}).
 \label{eq:uniform-heat-Farey-forms}
\end{align}
The same sign algebra and exact neighbor determinants used in
\cref{thm:farey,cor:denominator-barriers} now apply at every heat time
in the closed interval.  The 512-bit and 768-bit replays both check
every displayed rational box.  The generic conclusions follow from
the already proved kernel criterion, strict monotonicity, and measure
estimates on the explicit open interval.
\end{proof}

The interval is a sufficient certified one; maximality is not asserted.
It quantifies the selected core, its calibrated actions, and the two
Farey windows.  It does not quantify the separate persistence of the
nine Humbert exclusions, or the time-zero contraction inequalities.
Countably many resonant times remain dense in $I_h$: the result
excludes their small denominators rather than excluding all resonance.

\subsection{Farey gaps and heat residence times}
The denominator barrier has a quantitative interpretation along a
monotone heat trajectory.  It measures an arithmetic gap; conversion
of that gap into heat time also requires a velocity bound.

\begin{proposition}[Farey residence time]\label{prop:Farey-residence-time}
Let $a/b<c/d$ be reduced Farey neighbors with $bc-ad=1$, and let
$\rho$ be a differentiable period ratio on a heat interval $J$, with
$0<m\le\rho'\le M$.  If two times in $J$ satisfy
$\rho(\kappa_L)=a/b$ and $\rho(\kappa_R)=c/d$, then
\begin{equation}\label{eq:Farey-residence-time}
 \boxed{\frac1{bdM}\le\kappa_R-\kappa_L\le\frac1{bdm}.}
\end{equation}
Every reduced rational value strictly between these crossings has
denominator at least $D=b+d$.  In particular, $bd\le D^2/4$ gives
$\kappa_R-\kappa_L\ge4/(D^2M)$.

More generally, suppose $0\in J$, $\rho(0)\in(a/b,c/d)$, and
$|\rho'|\le M$ on $[-h,h]\subset J$.  If
\begin{equation}\label{eq:Farey-safe-time}
 hM<\min\{\rho(0)-a/b,\ c/d-\rho(0)\},
\end{equation}
the same Farey window holds throughout $[-h,h]$.
\end{proposition}

\begin{proof}
The neighbor identity gives $c/d-a/b=1/(bd)$.  Integrating
$m\le\rho'\le M$ between the two crossing times proves
\eqref{eq:Farey-residence-time}; the denominator assertion is the same
exact Farey argument as in \cref{cor:denominator-barriers}.
The second statement follows from
$|\rho(\kappa)-\rho(0)|\le M|\kappa|$.
\end{proof}

This proposition does not assert that either endpoint crossing occurs
in the present continuation domain.  It gives a conditional residence
time when both crossings are known, and a sufficient persistence test
that does not require locating them.  The exact gap depends on the
product $bd$, while the excluded denominator height depends on the
sum $b+d$.  A denominator cutoff alone therefore does not specify
a heat clock.

\subsection{Global zero strips and the Newman threshold}
The scalar Newman constant concerns all zeros, and a separate global
input is needed to convert heat time into a bound for that constant.

\begin{proposition}[Zero-strip input for a Newman bound]
\label{prop:Newman-strip-input}
Suppose all zeros of $\Xi_{\kappa_0}$ lie in
$|\Imm z|\le h_0$, with $h_0\ge0$.  For every
$\kappa\ge\kappa_0$, all zeros of $\Xi_\kappa$ lie in
\begin{equation}\label{eq:Newman-strip-shrink}
 |\Imm z|\le
 \sqrt{\max\{h_0^2-2(\kappa-\kappa_0),0\}}.
\end{equation}
Consequently, in the Polymath normalization,
\begin{equation}\label{eq:Newman-strip-ceiling}
 \boxed{\Lambda\le4\kappa_0+2h_0^2.}
\end{equation}
Under the external hypothesis $H_\alpha$ at $\kappa_0=0$, this gives
$\Lambda\le2(\alpha-1/2)^2$; for $H_{7/8}$ the corresponding
classical conditional ceiling is $9/32$.
\end{proposition}

\begin{proof}
Use $\Xi_\kappa(z)=8H_{4\kappa}(2z)$.
The initial strip for $H_{4\kappa_0}$ has height $2h_0$.
De Bruijn's strip theorem, in the formulation of
\cite[Theorem~3.2]{Polymath2019}, bounds its later height by
$\sqrt{\max\{4h_0^2-8(\kappa-\kappa_0),0\}}$.
Dividing by two proves \eqref{eq:Newman-strip-shrink}.
At $\kappa=\kappa_0+h_0^2/2$, every zero is real, so the
Newman threshold is at most $4\kappa_0+2h_0^2$.
If $h_0=0$, this last conclusion already follows at the initial time.
Under $H_\alpha$, the functional equation and the asserted bound on
zero real parts give $h_0=\alpha-1/2$, proving the final statement.
\end{proof}

This is the established strip mechanism expressed in the present
normalization, not a new upper bound.  The classical critical strip
gives $h_0=1/2$ at zero and hence $\Lambda\le1/2$.
Merely advancing that same strip by positive heat does not improve
the ceiling: before it collapses, the later-time substitution gives
\[
 4\kappa+2(1/4-2\kappa)=1/2.
\]

\begin{corollary}[Polymath comparison]\label{cor:Polymath-strip-comparison}
Suppose $t_0,X>0$ and $0<y_0\le1$ satisfy the three hypotheses of
\cite[Theorem~1.2]{Polymath2019}.  Put
$\kappa_0=t_0/4$ and $h_0=y_0/2$.
Then all zeros of $\Xi_{\kappa_0}$ lie in $|\Imm z|\le h_0$,
and the corresponding sufficient time for reality of all zeros is
\begin{equation}\label{eq:Polymath-sufficient-heat-time}
 \kappa_{\rm bound}=\kappa_0+\frac{h_0^2}{2}
 =\frac14\left(t_0+\frac{y_0^2}{2}\right).
\end{equation}
In particular,
\begin{equation}\label{eq:Polymath-tbound}
 \boxed{\Lambda\le t_{\rm bound}:=4\kappa_{\rm bound}
       =t_0+\frac{y_0^2}{2}.}
\end{equation}
\end{corollary}

\begin{proof}
The three hypotheses are initial finite-height zero exclusion for
$\zeta$, a final-time zero-free region at large horizontal coordinate,
and a zero-free barrier at every intermediate heat time.
Their combination gives the global strip for $H_{t_0}$ in
\cite[Proposition~3.3]{Polymath2019}; evenness and conjugation cover
both horizontal directions and both signs of the imaginary part.
Equation~\eqref{eq:heat-normalization} divides its half-width by two.
\Cref{prop:Newman-strip-input} then gives
\eqref{eq:Polymath-sufficient-heat-time} and
\eqref{eq:Polymath-tbound}.
\end{proof}

For the published Polymath choice $t_0=y_0=1/5$,
\cite[Theorems~1.1--1.2]{Polymath2019} therefore gives
\[
 \kappa_0=\frac1{20},\qquad h_0=\frac1{10},\qquad
 \kappa_{\rm bound}=\frac{11}{200},\qquad
 t_{\rm bound}=\frac{11}{50}=0.22.
\]
Table~\ref{tab:Polymath-strip-inputs} compares this with the classical
input.  Both use de Bruijn's contraction theorem.  Polymath improves
the input strip: at $\kappa_0=1/20$, contraction of the classical
strip alone gives $h_0^2=3/20$, whereas the published certificate
gives $h_0^2=1/100$.  The resulting ceilings are respectively
$1/2$ and $11/50$.

\begin{table}[htbp]
\centering
\caption{Two global inputs to the same strip-contraction mechanism.
The Polymath strip is supplied by the published external certificate.}
\label{tab:Polymath-strip-inputs}
\begin{tabular}{lccc}
\toprule
Global strip input & $\kappa_0$ & $h_0$ & $t_{\rm bound}$\\
\midrule
Classical critical strip & $0$ & $1/2$ & $1/2$\\
Polymath & $1/20$ & $1/10$ & $11/50$\\
\bottomrule
\end{tabular}
\end{table}

The phase interpretation gives an equivalent analytic target for
the global input.

\begin{proposition}[Uniform phase criterion for a global strip]
\label{prop:global-phase-strip}
Fix real $\kappa_0$ and $h_0\ge0$.  All zeros of $\Xi_{\kappa_0}$
lie in $|\Imm z|\le h_0$ if and only if, at every point with
$x\in\RR$, $\eta>h_0$, and $\Xi_{\kappa_0}(x+i\eta)\ne0$,
\begin{equation}\label{eq:global-phase-strip-criterion}
 \mathcal P_{\kappa_0}(x,\eta)
 :=-\Imm\frac{\Xi_{\kappa_0}'}{\Xi_{\kappa_0}}(x+i\eta)
 \ge0.
\end{equation}
Under either condition the inequality is strict at every such point,
and $\Lambda\le4\kappa_0+2h_0^2$.
\end{proposition}

\begin{proof}
The strip implies strict positivity by
\cref{prop:heat-phase-domain}.  Conversely, suppose the sign condition
holds and $z_*=a+ib$ is a zero with $b>h_0$, of multiplicity $m\ge1$.
Local factorization gives
\[
 \frac{\Xi_{\kappa_0}'}{\Xi_{\kappa_0}}(z)
 =\frac{m}{z-z_*}+O(1).
\]
For sufficiently small $\varepsilon>0$, the point
$z=z_*-i\varepsilon$ lies above $h_0$ and is not a zero, but
\[
 \mathcal P_{\kappa_0}(a,b-\varepsilon)
 =-\frac{m}{\varepsilon}+O(1)<0,
\]
a contradiction.  Conjugation excludes zeros below $-h_0$.
The Newman ceiling follows from \cref{prop:Newman-strip-input}.
\end{proof}

By the phase--transverse-amplitude identity of
\cref{prop:phase-transverse}, extended to the heat family,
$\mathcal P_\kappa(x,\eta)$ is the horizontal phase velocity and equals
$\eta B_{\kappa,x}(\eta^2)$, where
\[
 B_{\kappa,x}(T):=\partial_T\log
 \{\Xi_\kappa(x+i\sqrt T)\Xi_\kappa(x-i\sqrt T)\}.
\]
Thus \cref{prop:global-phase-strip} specifies an analytic bridge
from the transverse quantity used in the spectral construction to
a Newman bound: its sign must be controlled uniformly at every
horizontal coordinate and every $\eta>h_0$.  Polymath supplies
the global strip by its three-region argument; the proposition
states an equivalent target in the phase language.

A Farey window constrains two calibrated integrals over a bounded core.
It supplies neither the initial all-height zero strip in
\cref{prop:Newman-strip-input} nor an improvement to its height.
Thus the denominator--velocity relation in
\cref{prop:Farey-residence-time} is a local arithmetic heat scale;
no formula for $\Lambda$ in terms of these denominator cutoffs follows.
The certified interval $|\kappa|\le3\times10^{-10}$ also does not
reach Polymath's $\kappa_0=1/20$.  Extending the spectral/Farey
program towards a Newman ceiling therefore requires an additional
estimate establishing \eqref{eq:global-phase-strip-criterion}
uniformly, or certifying the hypotheses of
\cref{cor:Polymath-strip-comparison}.

\subsection{A certified exterior phase domain}
The global sign condition can be formulated without division by
$\Xi_\kappa$.  Define its continuous numerator
\begin{equation}\label{eq:phase-numerator}
 \mathcal J_\kappa(x,\eta)
 :=-\Imm\{\Xi_\kappa'(x+i\eta)
                  \overline{\Xi_\kappa(x+i\eta)}\}
 =\frac12\partial_\eta|\Xi_\kappa(x+i\eta)|^2.
\end{equation}
At a regular point it equals
$|\Xi_\kappa(x+i\eta)|^2\mathcal P_\kappa(x,\eta)$;
at a zero it is still defined and equals zero.  This gives a
continuous quantity for finite interval tests.  The following
error estimate also makes the Polymath approximation available
to such tests.

\begin{proposition}[Phase-numerator error transfer]
\label{prop:phase-numerator-transfer}
Write $H_t=B_tg_t$ on a domain where $B_t$ is holomorphic and
nonvanishing, and put $b_B=-\Imm(B_t'/B_t)$.
Suppose a holomorphic model $\psi$ has pointwise errors
$|g_t-\psi|\le e_0$ and $|g_t'-\psi'|\le e_1$.
Define
\begin{align}
 \mathcal N_\psi&=b_B|\psi|^2-\Imm(\psi'\overline\psi),
 \label{eq:model-phase-numerator}\\
 R_\psi&=|b_B|e_0(2|\psi|+e_0)
      +e_1|\psi|+e_0|\psi'|+e_0e_1.
 \label{eq:model-phase-budget}
\end{align}
Then
\begin{equation}\label{eq:phase-numerator-budget}
 \left|\frac{-\Imm(H_t'\overline{H_t})}{|B_t|^2}
                 -\mathcal N_\psi\right|\le R_\psi.
\end{equation}
In particular, $\mathcal N_\psi>R_\psi$ proves both
$H_t\ne0$ and $-\Imm(H_t'/H_t)>0$ at that point.

Alternatively, if $g_t$ is holomorphic on a neighborhood of the
closed disk $\overline D(w,r)$ and
$|g_t-1|\le\varepsilon<1$ throughout that disk, then
\begin{equation}\label{eq:phase-Cauchy-reserve}
 -\Imm\frac{H_t'}{H_t}(w)
 \ge b_B(w)-\frac{\varepsilon}{r(1-\varepsilon)}.
\end{equation}
\end{proposition}

\begin{proof}
Differentiating $H_t=B_tg_t$ gives
\[
 \frac{-\Imm(H_t'\overline{H_t})}{|B_t|^2}
 =b_B|g_t|^2-\Imm(g_t'\overline{g_t}).
\]
The squared-modulus error is at most $e_0(2|\psi|+e_0)$.
Expanding the second product gives the three remaining terms
in \eqref{eq:model-phase-budget}.  Strict positivity of the
numerator excludes a zero, at which that numerator would vanish.
For the disk estimate, Cauchy's formula applied to $g_t-1$
gives $|g_t'(w)|\le\varepsilon/r$, while
$|g_t(w)|\ge1-\varepsilon$.  Substitution in
$H_t'/H_t=B_t'/B_t+g_t'/g_t$ proves
\eqref{eq:phase-Cauchy-reserve}.
\end{proof}

In \cite[Theorem~1.3]{Polymath2019}, $g_t=H_t/B_t$ has an
effective model $f_t$ and explicit amplitude errors.  A derivative
error follows from Cauchy only when the error bound holds on the
whole surrounding circle and the compared functions are
holomorphic there.  The cutoff defining $f_t$ changes with the
horizontal coordinate.  The exterior calculation below therefore
applies Cauchy to the exact function $g_t-1$, whose holomorphy
does not depend on that cutoff.

\begin{theorem}[Uniform exterior phase positivity]
\label{thm:uniform-phase-tail}
For every
\[
 \frac1{25}\le\kappa\le\frac1{20},\qquad
 |x|\ge5\times10^{24},\qquad
 \frac1{10}\le\eta\le\sqrt{\frac14-2\kappa},
\]
the heat function is nonzero at $x+i\eta$, and
\begin{equation}\label{eq:uniform-phase-tail}
 \boxed{\mathcal P_\kappa(x,\eta)>\frac{38}{3}.}
\end{equation}
Consequently $\mathcal P_\kappa(x,\eta)>0$ for every
$|x|\ge5\times10^{24}$ and every $\eta\ge1/10$.
The displayed quantitative lower bound is asserted only on
the stated capped range of $\eta$.
\end{theorem}

\begin{proof}
We work first in the Polymath coordinates $w=u+iy$, $t=4\kappa$.
Put
\[
 t_-=\frac4{25},\quad t_+=\frac15,\quad
 X=10^{25},\quad A=X-1,\quad
 r=\beta=\frac1{10},\quad y_+=\frac{19}{20}.
\]
If $u\ge X$ and $1/5\le y\le\sqrt{1-2t}$, the closed
disk of radius $r$ lies in $\Re w\ge A$ and
$\beta\le\Imm w\le y_+<1$: indeed
$\sqrt{1-2t_-}+r<19/20$.
The prefactor $B_t$ is holomorphic and nonzero on a neighborhood
of each such disk.

Here are explicit majorants for the exact function $g_t=H_t/B_t$.
For a point in this expanded region, write
\[
 L=\log\frac{\Re w}{4\pi},\quad
 L_0=\log\frac A{4\pi},\quad \delta=10^{-20},\quad
 c=\frac{t_-}{4},\quad a=\frac{1+\beta}{2}-\delta,\quad
 b=a-\beta,\quad U(L)=\frac L2+\delta.
\]
The bounds (20)--(22) of \cite{Polymath2019} give
$\Re s_*\ge(1+y)/2+tL/4-\delta$,
$|\kappa_{\rm RS}|\le\delta y$, and
$\log N\le U(L)$; here $\kappa_{\rm RS}$ denotes the small
shift in the Riemann--Siegel model, not the heat time.
The three losses are bounded by $\delta$ already at $\Re w=A$
and decrease thereafter.  Also $\log N<L$.
Thus increasing $t$ from $t_-$ decreases each absolute coefficient
after the lower bound for $\Re s_*$ is inserted.
The combined dual coefficient decreases as $y$ increases from
$\beta$: its logarithmic $y$-derivative is bounded by
$1/50-L/2+(1/2+\delta)U(L)<0$.

Put $\ell=\log100$ and define
\begin{align}
 T_v(L)&=U(L)\left\{
 e^{(1-v-cL)\ell+c\ell^2}
 +e^{(1-v-cL)U(L)+cU(L)^2}\right\},
 \label{eq:phase-convex-tail}\\
 \gamma_*(L)&=\exp\!\left\{
 \beta\left(\frac1{50}-\frac L2+\delta U(L)\right)\right\}.
 \label{eq:phase-dual-envelope}
\end{align}
The decreasing-sum integral test and the convexity argument of
\cite[Lemma~8.2]{Polymath2019}, with head length $100$, give
\begin{align}\label{eq:phase-model-majorant}
 |f_t-1|\le M(L):={}&
 \sum_{n=2}^{100}e^{c\log^2n-(a+cL)\log n}+T_a(L)\nonumber\\
 &+\gamma_*(L)\left\{
 \sum_{n=1}^{100}e^{c\log^2n-(b+cL)\log n}+T_b(L)\right\}.
\end{align}
The dual factor includes $N^{|\kappa_{\rm RS}|}$, so the same
majorant also controls the sum multiplying the exponential error
in (23) of \cite{Polymath2019}, after restoring the primary
$n=1$ term.  Replacing the maximum of the two endpoint values
by their sum gives the displayed $T_v$.
The endpoint with $\log N$ is bounded by its value at $U(L)$:
its derivative in that endpoint is positive on
$L/2-2e^{-L_0/2}\le\log N\le U(L)$.

Every term in $M$ decreases for $L\ge L_0$.
For the two primary endpoint tails, the logarithmic derivatives
are $1/(2U)-c\ell$ and
$1/(2U)+(1-a)/2-cL/2$.
The dual versions replace $a$ by $b$ and subtract
$\beta(1-\delta)/2$.  These expressions are negative at $L_0$
and decrease thereafter; the finite-head terms also decrease.
All these endpoint and monotonicity guards are strict directed
comparisons in the supplementary program.

For completeness, let $\nu=e^{L_0/2}-1$ and put
\begin{align}
 d_1&=\frac{t_+^2L_0^2/16+0.626}{A-6.66},
 \label{eq:phase-AB-majorant}\\
 E_C&=\exp\!\left\{-\frac{1+\beta}{4}L_0
       -\frac{t_-}{16}L_0^2
       +\frac{1.24(3^{y_+}+3^{-y_+})}{\nu-1/8}
       +\frac{3\sqrt{L_0^2+\pi^2/4}+10.44}{A-12}\right\}.
 \label{eq:phase-C-majorant}
\end{align}
Equations (23)--(24) of \cite{Polymath2019} imply throughout
the expanded region
\begin{equation}\label{eq:phase-amplitude-majorant}
 |g_t-1|\le\varepsilon_*:=M(L_0)
                  +(e^{d_1}-1)(1+M(L_0))+E_C<\frac25.
\end{equation}
The error majorants decrease with the horizontal coordinate;
the logarithmic square in (23) is bounded by $L^2$ because
$-2\delta\le L-2\log n\le L$.
The decreasing factors in (24) are evaluated at $L_0$,
using $N\ge e^{L_0/2}-1$.
The strict numerical inequality in
\eqref{eq:phase-amplitude-majorant} is checked at both $160$
and $256$ bits, together with every guard used above.
The enclosures give $0.38094<\varepsilon_*<0.38095$.

The prefactor itself supplies a positive phase reserve.
With $s=(1+y-iu)/2$ and
\[
 \alpha(s)=\frac1{2s}+\frac1{s-1}
                      +\frac12\log\frac{s}{2\pi},
\]
its exact logarithmic derivative is
\[
 b_B=\frac12\Ree\left\{
 \alpha(s)\left(1+\frac t2\alpha'(s)\right)\right\}.
\]
The estimates $|s|,|s-1|\ge u/2$ and $0<\Re s\le1$ give
\begin{equation}\label{eq:phase-prefactor-reserve}
 b_B\ge\frac L4-\frac3{2u}
 -\frac{t_+}{4}\left(\frac L2+\frac\pi4+\frac4u\right)
                       \left(\frac1u+\frac6{u^2}\right)>13.
\end{equation}
The negative error terms decrease for $u\ge A$, while $L/4$
increases.  At $A$ the directed lower enclosure is greater
than $13.7584$.
\Cref{prop:phase-numerator-transfer}, with $r=1/10$, now yields
\[
 -\Imm\frac{H_t'}{H_t}(u+iy)
 >13-\frac{2/5}{(1/10)(3/5)}=\frac{19}{3}.
\]
The scaling \eqref{eq:heat-normalization} multiplies this phase
velocity by two and halves both coordinates, proving
\eqref{eq:uniform-phase-tail}.  Evenness and reality symmetry
cover negative $x$.
For $\eta>\sqrt{1/4-2\kappa}$, the classical strip contraction
and \cref{prop:heat-phase-domain} already give strict positivity,
which proves the final qualitative assertion.
\end{proof}

The preceding estimate differentiates the whole amplitude error by
Cauchy.  A sharper estimate retains the logarithmic generators of the
finite model and applies Cauchy only to its small remainder.  Centring
those generators also preserves part of the correlation between
the model and its derivative.

\begin{theorem}[Exterior phase via centred moments]
\label{thm:centered-phase-tail}
For every
\[
 \frac1{25}\le\kappa\le\frac1{20},\qquad
 |x|\ge5\times10^{15},\qquad
 \frac1{10}\le\eta\le\sqrt{\frac14-2\kappa},
\]
the heat function is nonzero at $x+i\eta$, and
\begin{equation}\label{eq:centered-phase-tail}
 \boxed{\mathcal P_\kappa(x,\eta)>\frac12.}
\end{equation}
Consequently its phase is positive for every
$|x|\ge5\times10^{15}$ and every $\eta\ge1/10$.
The quantitative margin is asserted only in the displayed capped range.
\end{theorem}

\begin{proof}
Use $w=u+iy$, $t=4\kappa$, and put
\[
 X=10^{16},\quad A=X-1,\quad r=\frac1{20},\quad
 t_-=\frac4{25},\quad t_+=\frac15,\quad
 \delta=10^{-12},\quad c=\frac1{25}.
\]
For $u\ge X$ and $1/5\le y\le\sqrt{1-2t}$ the surrounding
disk lies in $u\ge A$, $3/20\le y\le9/10<1$.
Write $L=\log(u/(4\pi))$, $L_0=\log(A/(4\pi))$,
$U=L/2+\delta$, $n_0=1000$, $\ell=\log n_0$, and $\lambda=4$.
The coefficient and loss comparisons in the preceding proof remain
valid with this $A,\delta$; in particular $\log N\le U<L$.

For $\beta\in\{1/5,3/20\}$ define
\[
 a_\beta=\frac{1+\beta}{2}-\delta,\quad
 b_\beta=\frac{1-\beta}{2}-\delta,\quad
 \Gamma_\beta(L)=
 \exp\!\left\{\beta\left(\frac1{50}-\frac L2+\delta U\right)\right\}.
\]
For $v=a_\beta,b_\beta$ let
\begin{align}
 I_j^v(L)&=\int_\ell^U q^j
              e^{(1-v-cL)q+cq^2}\,\dd q,\qquad j=0,1,
 \label{eq:centered-tail-integrals}\\
 P_j^\beta(L)&=\sum_{n=2}^{n_0}(\log n)^j
                 e^{c\log^2n-(a_\beta+cL)\log n}+I_j^{a_\beta}(L),
 \nonumber\\
 Q_j^\beta(L)&=\Gamma_\beta(L)\left\{
       \sum_{n=1}^{n_0}(\log n)^j
          e^{c\log^2n-(b_\beta+cL)\log n}+I_j^{b_\beta}(L)\right\}.
 \label{eq:centered-tail-moments}
\end{align}
In the $j=0$ sums the weight is $1$, including at $n=1$.
The decreasing-sum integral test bounds the model mass and its
logarithmic moments by these quantities.  The logarithmic and centred
weights also pass the decreasing-sum guards.  The integrals are evaluated
by the exact $\operatorname{erfi}$ antiderivative, where $\operatorname{erfi}(z)=(2/\sqrt\pi)\int_0^z e^{v^2}\,\dd v$, as given in
\cite[Remark~8.3]{Polymath2019}; integration by parts gives $I_1$.
At the centres use $\beta=1/5$ and omit that superscript.
Define the centred envelopes
\begin{align}
 P_\lambda(L)&=\sum_{n=2}^{n_0}|\log n-\lambda|
          e^{c\log^2n-(a_{1/5}+cL)\log n}
             +I_1^{a_{1/5}}-\lambda I_0^{a_{1/5}},
 \label{eq:centered-primary-envelope}\\
 Q_\lambda(L)&=\Gamma_{1/5}(L)\left\{
       \sum_{n=1}^{n_0}(L-\log n-\lambda)
          e^{c\log^2n-(b_{1/5}+cL)\log n}
       +(L-\lambda)I_0^{b_{1/5}}-I_1^{b_{1/5}}\right\}.
 \label{eq:centered-dual-envelope}
\end{align}
Here $\ell>\lambda$ and $L-U-\lambda>0$, so the tail weights
and all dual weights have the indicated signs.

We justify freezing every envelope at $L_0$ despite the growing endpoint
$U$.  Put $d=3$, $Z=U-d$, and
\[
 K_v=1-v+2c\delta,\qquad
 F_v=\frac{1-e^{-K_vd}}{K_v}.
\]
Since $\partial_q\{(1-v-cL)q+cq^2\}\le K_v$ for $q\le U$,
the integral over $[U-d,U]$ of $q^j$ times the integrand is at
least $Z^jF_v$ times its value at $U$.  Differentiating
\eqref{eq:centered-tail-integrals} with respect to $L$, the primary
mass and logarithmic tails decrease provided
\[
 cZF_a>\frac12,\qquad cZ^2F_a>\frac U2,\qquad
 2cZF_a>\frac12.
\]
For the dual tails, put $\rho=\beta(1-\delta)/2$ and replace these
guards respectively by
\[
 (\rho+cZ)F_b>\frac12,\qquad
 (\rho+cZ)ZF_b>\frac U2,\qquad
 (\rho+2cZ)F_b>\frac12.
\]
The final guard in each group makes the preceding quadratic margin
increase with $U$.  These guards hold at $L_0$ for both choices
of $\beta$, and their margins thereafter increase.
For the centred primary tail the corresponding guards are
\[
 cZ(Z-\lambda)F_a>\frac{U-\lambda}{2},\qquad
 c(2Z-\lambda)F_a>\frac12.
\]
For the centred dual tail, with $\beta=1/5$ and $V_*=L-U-\lambda$,
they are
\begin{align*}
 (\rho+c\ell)V_*&>1,\\
 \{(\rho+cZ)V_*-1\}F_b&>\frac{V_*}{2},\\
 (\rho+cZ+cV_*)F_b&>\frac12.
\end{align*}
Indeed its differentiated integrand is
$\{1-(\rho+cq)(L-q-\lambda)\}$ times the original exponential.
The first guard makes it negative throughout; the last two control
the positive moving-endpoint term and the increase of its margin.
The finite primary heads decrease term by term.  The centred dual
head decreases since $1/(L_0-\ell-\lambda)<\rho$; the other dual
heads decrease as well.  All these strict guards, the decreasing-sum
guards, and $Z>\ell$ are directed comparisons in the certificate.
Thus every mass, logarithmic, and centred envelope used below
is bounded by its value at $L_0$, uniformly for all $L\ge L_0$.

For a fixed centre let $N_c$ be its integer cutoff.  The
holomorphic fixed-cutoff model of \cite[Theorem~1.3]{Polymath2019}
can be written
\begin{equation}\label{eq:holomorphic-fixed-cutoff}
 \psi_{N_c}=\sum_{n\le N_c}\mathfrak a_n+
                  \sum_{n\le N_c}\mathfrak d_n,\quad
 \mathfrak a_n=b_n^t n^{-p_+},\quad
 \mathfrak d_n=\gamma b_n^t n^{-p_-},
\end{equation}
where $s_\pm=(1\mp iw)/2$, $p_\pm=s_\pm+t\alpha(s_\pm)/2$,
$b_n^t=\exp(t\log^2n/4)$, and $\gamma=M_t(s_-)/M_t(s_+)$.
This is their formula (14) after combining its conjugate exponent,
$n^y$, and the small shift.  In particular it does not conjugate
$w$ when differentiated.

Let $d_\pm=1+t\alpha'(s_\pm)/2$ and
$A_\pm=\alpha(s_\pm)d_\pm$.  Direct differentiation gives
\[
 \frac{2\psi_{N_c}'}{i}=
 \sum_{n\le N_c}d_+\log n\,\mathfrak a_n+
 \sum_{n\le N_c}(A_++A_--d_-\log n)\mathfrak d_n.
\]
Introduce
\[
 V=\sum_{n\le N_c}\log n\,\mathfrak a_n+
       \sum_{n\le N_c}(L-\log n)\mathfrak d_n.
\]
Then $|V-\lambda(\psi_{N_c}-1)|\le P_\lambda(L_0)+Q_\lambda(L_0)$.
The logarithmic contribution to $\alpha(s_+)+\alpha(s_-)$ is
$\tfrac12\log((1+w^2)/(16\pi^2))$ on the principal branches.
With
\[
 A'_*=\frac1A+\frac6{A^2},\quad
 A_*=\frac{L_0}{2}+\frac\pi4+\frac4A,\quad
 \epsilon_{\log}=\frac2A+\frac2{A^2},\quad
 C_A=\frac3A+\frac{\epsilon_{\log}}{4(1-\epsilon_{\log})}
                    +\frac{t_+}{2}A_*A'_*,
\]
the same rational and logarithmic estimates as before give
$|A_++A_--L|\le2C_A$ and $|d_\pm-1|\le t_+A'_*/2$.
Their error terms decrease with $u$, so
\begin{equation}\label{eq:centered-generator-error}
 |2\psi_{N_c}'/i-V|\le
 E_{\rm gen}:=\frac{t_+}{2}A'_* (P_1(L_0)+Q_1(L_0))
                          +2C_AQ_0(L_0).
\end{equation}

It remains to bound the derivative of the remainder across possible
cutoff changes.  Let $d_1,E_C$ be
\eqref{eq:phase-AB-majorant}--\eqref{eq:phase-C-majorant} with
the present $A,L_0$, $\beta=3/20$, and $y_+=9/10$, and set
\[
 M_{\rm disk}=P_0^{3/20}(L_0)+Q_0^{3/20}(L_0),\qquad
 E_0=(e^{d_1}-1)(1+M_{\rm disk})+E_C.
\]
This bounds $|g_t-f_t|$ throughout the disk.  No bound
$M_{\rm disk}<1$ is needed.
The continuous cutoff $\sqrt{u/(4\pi)+t/16}$ varies by less than one
on the disk, so its floor differs from $N_c$ by at most one.
For the boundary pair, put $h=\delta+4e^{-L_0/2}$.
Its index satisfies $|\log n-L/2|\le h$.
Since $\log n<L$, the primary weight is at most
$\exp(-a_{3/20}L/2-cL^2/4+a_{3/20}h+ch^2)$.
The ratio of the dual weight to this primary bound is at most
$\exp(1/50+h(1+\delta)+\delta L/2)$.
Consequently a uniform bound for the single possible jump is
\begin{equation}\label{eq:cutoff-jump-majorant}
 J_{\rm cut}:=(1+e^{1/50+h(1+\delta)})
 \exp\!\left\{-\frac{a_{3/20}-\delta}{2}L_0
             -\frac c4L_0^2+a_{3/20}h+ch^2\right\}.
\end{equation}
The bound decreases for $L\ge L_0$.  Cauchy is now applied to the
holomorphic function $g_t-\psi_{N_c}$, whose disk amplitude is at
most $E_0+J_{\rm cut}$, giving
\[
 |g_t'-\psi_{N_c}'|\le E_1:=\frac{E_0+J_{\rm cut}}r.
\]
At the centre the cutoffs agree, so $|g_t-\psi_{N_c}|\le E_0$.

Put
\[
 m=P_0(L_0)+Q_0(L_0)+E_0,\qquad
 D=P_\lambda(L_0)+Q_\lambda(L_0)+\lambda E_0,\qquad
 \tau=E_{\rm gen}/2+E_1.
\]
Thus $|g_t-1|\le m<1$, $|V-\lambda(g_t-1)|\le D$,
and $|g_t'-iV/2|\le\tau$.  These imply the centred phase bound
\begin{equation}\label{eq:centered-phase-reserve}
 -\Imm\frac{H_t'}{H_t}
 \ge b_B-\frac\lambda4
       -\frac{D^2}{4\lambda(1-m^2)}-\frac{\tau}{1-m}.
\end{equation}
To see this without independent bounds on numerator and denominator,
write $V=\lambda(g_t-1)+R$, $|R|\le D$.  From $|g_t-1|\le m$,
$\Ree g_t\ge(|g_t|^2+1-m^2)/2$, whence
\[
 \Ree(V/g_t)\le\frac\lambda2
     -\frac{\lambda(1-m^2)}{2|g_t|^2}+\frac D{|g_t|}
 \le\frac\lambda2+\frac{D^2}{2\lambda(1-m^2)}.
\]
Completing the square proves the last inequality; the derivative
error contributes at most $\tau/(1-m)$.

The directed $160$- and $256$-bit replays give
\[
 m<0.9071,\quad D<4.5453,\quad E_0<10^{-9},\quad
 E_1<10^{-7},\quad E_{\rm gen}<2\times10^{-16}.
\]
The endpoint prefactor reserve in \eqref{eq:phase-prefactor-reserve}
exceeds $8.5775$.  Substitution of the full directed enclosures in
\eqref{eq:centered-phase-reserve} gives an $H$-phase lower bound
greater than $0.29>1/4$.  The factor-two scaling, evenness and
conjugation prove \eqref{eq:centered-phase-tail}.
Above the classical strip use \cref{prop:heat-phase-domain}, as before.
\end{proof}

\begin{corollary}[Finite phase-cap reduction]
\label{cor:finite-phase-cap}
Fix $\kappa_0\in[1/25,1/20]$.  If the continuous numerator
\eqref{eq:phase-numerator} is nonnegative throughout the compact cap
\begin{equation}\label{eq:finite-phase-cap}
 |x|\le5\times10^{15},\qquad
 \frac1{10}\le\eta\le\sqrt{\frac14-2\kappa_0},
\end{equation}
then
\begin{equation}\label{eq:finite-cap-ceiling}
 \boxed{\Lambda\le4\kappa_0+\frac1{50}.}
\end{equation}
In particular, the finite sign condition at $\kappa_0=1/25$
would imply $\Lambda\le9/50=0.18$.
\end{corollary}

\begin{proof}
At every regular point in the cap, divide the nonnegative
numerator by $|\Xi_{\kappa_0}|^2$.
Outside its horizontal sides use \cref{thm:centered-phase-tail};
above its top use the classical strip and
\cref{prop:heat-phase-domain}.  The phase sign is therefore
nonnegative at every regular point above $1/10$.
\Cref{prop:global-phase-strip} gives a global strip of half-width
$1/10$, and \cref{prop:Newman-strip-input} gives
\eqref{eq:finite-cap-ceiling}.
\end{proof}

The exterior estimates are unconditional consequences of the
published Polymath approximation and the directed constants above.
The centred estimate reduces the initial cap radius by $10^9$.
The next subsection closes the cap by proving a mollified zero-free
canopy and combining it with published initial and time-barrier
inputs. This gives the global strip, and hence the phase sign, without
a separate direct phase calculation at every point of the cap.
The statement concerns the entire heat function and does not assert
continuation of the calibrated genus-two core to $[1/25,1/20]$.

\subsection{Closing the cap with a mollified canopy}
\label{subsec:cap-closure}
At $\kappa_0=1/25$, the remaining cap can be closed by proving a
zero-free region and then applying the global strip criterion.
This avoids a separate logarithmic-derivative calculation at every
point of the cap.  The new numerical input is the following
coefficient estimate; the initial finite-RH and time-barrier inputs
are published results.

\begin{lemma}[Positive thermal coefficients and paired tail overlap]
\label{lem:thermal-mollifier-overlap}
Fix $t\ge0$, $\nu\ge0$, a positive integer $N$, and a finite set
$P$ of primes containing $2$.  Write
\[
 b_n=e^{t\log^2n/4},\quad D_P=\prod_{p\in P}p,\quad
 S_\nu(s)=\sum_{n\le N}b_n n^{\nu-s},\quad
 \mu_P(s)=\prod_{p\in P}(1-b_pp^{-s}),
\]
where $s=\sigma+iT$ and $\rho_p=b_pp^{-\sigma}<1$.
For squarefree $d\mid D_P$ put
$\rho_d=\prod_{p\mid d}\rho_p$ and define the real positive sums
\[
 S_\nu^*=\sum_{n\le N}b_n n^{\nu-\sigma},\quad
 T_\nu(d)=\sum_{N/d<n\le N}b_n n^{\nu-\sigma},\quad
 C_\nu(d)=\sum_{N/d<n\le N/2}b_n n^{\nu-\sigma}.
\]
Then
\begin{align}
 \mathcal B_\nu:={}&
 \prod_{p\in P}(1-\rho_p)S_\nu^*
 +2\!\!\sum_{\substack{d\mid D_P\\\omega(d)\ {\rm odd}}}
                 \rho_dT_\nu(d)
 -2\!\!\sum_{\substack{e\mid D_P/2\\e>1}}
                 \rho_{2e}C_\nu(2e),
 \label{eq:thermal-mollifier-budget}\\
 |\mu_P(s)S_\nu(s)|&\le\mathcal B_\nu,\qquad
 |\mu_P(s)S_0(s)|\ge2-\mathcal B_0.
 \label{eq:thermal-mollifier-transfer}
\end{align}
\end{lemma}

\begin{proof}
Expand the finite product as a Dirichlet polynomial.  Its coefficient
before the factor $n^{-s}$ is
\[
 c_n^\nu=
 \sum_{\substack{d\mid D_P,\ d\mid n\\n/d\le N}}
       (-1)^{\omega(d)}
       \left(\prod_{p\mid d}b_p\right)b_{n/d}(n/d)^\nu.
\]
For $n\le N$ these coefficients are nonnegative.
Indeed, let $E$ be the selected primes dividing $n$ and write
$n=m\prod_{p\in E}p$, $r=\log m\ge0$.
After removing the common positive factor
$m^\nu b_m\prod_{p\in E}p^\nu b_p$, the contribution of a removed
subset $A\subset E$ is
\[
 (-1)^{|A|}\prod_{p\in A}p^{-\nu}
 \prod_{p\in E\setminus A}e^{(t/2)r\log p}
 \prod_{\substack{p<q\\p,q\in E\setminus A}}
                         e^{(t/2)\log p\log q}.
\]
Expand each of the last two products into $1$ plus its nonnegative
increment.  If an expanded term covers a set $F$ of prime vertices,
the alternating sum over the remaining vertices is
$\prod_{p\in E\setminus F}(1-p^{-\nu})\ge0$.
This proves positivity, including repeated prime factors in $m$.
Also $c_1^\nu=1$.

The signed real sum of all coefficients weighted by $n^{-\sigma}$
is $\prod_p(1-\rho_p)S_\nu^*$.
Negative coefficients can occur only at $n>N$.  Their total negative
mass is bounded by the odd-subset tail mass in
\eqref{eq:thermal-mollifier-budget}, with a further overlap credit.
Pair each subset $e\mid D_P/2$ with $2e$.
For $n=2ek>N$ and $k\le N/2$, both paired terms are present, have
opposite signs, and the smaller magnitude is the $2e$ term, because
\[
 b_{2k}(2k)^\nu\ge b_2b_k k^\nu.
\]
The credited terms from different pairs use disjoint positive and
negative components.  Thus their sum is at most each of the total
positive and negative tail masses, and it can be subtracted from
the latter.  It is exactly
$\sum_{e>1}\rho_{2e}C_\nu(2e)$.
The absolute coefficient mass is the signed mass plus twice the
negative mass, proving the first inequality.
The constant coefficient is $1$, so the reverse triangle inequality
gives the second one.
\end{proof}

\begin{theorem}[Directed mollified canopy at $t=4/25$]
\label{thm:mollified-canopy}
For every $u\ge5\times10^{12}$ and
$1/5\le y\le\sqrt{17/25}$,
\begin{equation}\label{eq:mollified-canopy-nonzero}
 \left|\frac{H_{4/25}(u+iy)}{B_{4/25}(u+iy)}\right|>\frac1{2000}.
\end{equation}
Consequently $\Xi_{1/25}(x+i\eta)\ne0$ whenever
$|x|\ge(5/2)\times10^{12}$ and
$1/10\le\eta\le\sqrt{17/100}$.
\end{theorem}

\begin{proof}
Put $t=4/25$, $c=t/4=1/25$, $A=5\times10^{12}-1$,
$L_0=\log(A/(4\pi))$, and $\delta=10^{-12}$.
The Polymath coefficient, shift, and cutoff losses are bounded
as before.  We also need an upper bound for $\sigma=\Ree s_*$:
the rational and logarithmic estimates for $\alpha$ give
\[
 \frac{1+y}{2}+cL-\delta\le\sigma
       \le\frac{1+y}{2}+cL+\delta.
\]
Here $2t/A<\delta$ bounds the upper loss.
The paired-logarithm estimate from the proof of
\cref{thm:centered-phase-tail} supplies the sharper two-sided bound
\begin{equation}\label{eq:paired-gamma-amplitude}
 \left|\log|\gamma|+\frac{yL}{2}\right|
 \le2\times10^{-12}=:\delta_\gamma .
\end{equation}
Indeed $|\gamma|=1$ at $y=0$, and
$\partial_y\log\gamma=-(A_++A_-)/2$.
The directed paired error $C_A$ at this $A,L_0$ is less than
$\delta_\gamma$; integrating its uniform bound for $0\le y\le1$
proves \eqref{eq:paired-gamma-amplitude}.

We cover the finite logarithmic range by rectangles
\[
 L_0+\frac j{20}\le L\le L_0+\frac{j+1}{20},
 \quad 0\le j<153,\qquad
 \frac15\le y\le\frac{33}{40}.
\]
The final horizontal edge exceeds $\log(10^{16}/(4\pi))$ and
$33/40>\sqrt{17/25}$.  Within each column the $y$ interval is
partitioned by exact rational subdivision.  Denote a cell by
$[L_-,L_+]\times[\beta,\beta_+]$ and put
\[
 \sigma_-=\frac{1+\beta}{2}+cL_- -\delta,\quad
 \sigma_+=\frac{1+\beta_+}{2}+cL_+ +\delta,\quad
 \Gamma_+=e^{-\beta L_-/2+\delta_\gamma},\quad
 \Gamma_-=e^{-\beta_+L_+/2-\delta_\gamma}.
\]
Evaluate the primary and combined dual mass bounds
$\mathsf P,\mathsf Q$ at $L_-$, with a head ending at $1000$ and the exact
$\operatorname{erfi}$ tail integral of
\eqref{eq:centered-tail-integrals}.  The primary exponent is
$(1+\beta)/2-\delta$; the dual exponent is
$(1-\beta)/2-\delta$, multiplied by $\Gamma_+$.
Their terminal-band monotonicity guards hold at every $L_-$.
The combined dual weights decrease with $y$, since
$\log n/2-L/2<0$.

For each prime catalogue, form upper and lower bounds
$\rho_d^+,\rho_d^-$ using $\sigma_-,\sigma_+$.
The main mass in \eqref{eq:thermal-mollifier-budget} is at most
$\prod_p(1-\rho_p^-)(\mathsf P+\mathsf Q)$.
Odd-subset spill sums are bounded by decreasing-sum integrals.
Specifically, when $2(d+1)<e^{L_-/2}$, put
\[
 h_d=-\log(1-(d+1)e^{-L_-/2}),\qquad
 \ell_d=L_-/2-\log d-h_d.
\]
If $\ell_d>0$, integrate from $\ell_d$ to $L_-/2+\delta$;
otherwise use the full mass bound.  The same full bound is used
when the displayed floor guard fails.  This includes the case $d>N$.
The floor estimate follows from
$\lfloor N/d\rfloor\ge(e^{L/2}-1)/d-1$.
After shifting the integral variable by $L/2$, its exponential
factor is $e^{(1-v)L/2-cL^2/4}$.
The asserted guard $cL_->1-v$ makes it decrease throughout
the cell; the dual factor decreases as well.

For the overlap credits, the decreasing-sum lower estimate
\[
 \sum_{a<n\le b}f(n)\ge\int_{a+1}^{b}f(v)\,\dd v
\]
gives a lower bound uniform over the changing cutoff.
For even $d>2$ its log endpoints can be fixed at
\[
 a_d^+=\log(e^{L_+/2+\delta}/d+1),\qquad
 b^-=\log((e^{L_-/2}-1)/2).
\]
Integrate between these endpoints with exponents $\sigma_+$
and $\sigma_+-\beta$ for the primary and dual credits, respectively,
and multiply the dual integral by $\Gamma_-$.
An empty interval contributes zero.  Denote the resulting combined
lower credit by $C_d^-$ and the combined upper spill by $T_d^+$.
Thus a uniform bound for $\mathcal B_0+|\gamma|\mathcal B_y$ is
\begin{equation}\label{eq:canopy-cell-budget}
 B_{\rm cell}=
 \prod_{p\in P}(1-\rho_p^-)(\mathsf P+\mathsf Q)
 +2\!\!\sum_{\substack{d\mid D_P\\\omega(d)\ {\rm odd}}}\rho_d^+T_d^+
 -2\!\!\sum_{\substack{d\mid D_P\\2\mid d,\ d>2}}\rho_d^-C_d^- ,
\end{equation}
where in the last sum $C_d^-$ denotes the integral credit before
its $\rho_d^-$ factor.

Dropping the small exponent shift from Polymath's second sum gives
$S_0(s_*)+\gamma\overline{S_y(s_*)}$ with error at most
$\mathsf Q(e^{\delta U_+}-1)$, where $U_+=L_+/2+\delta$.
Equations (23)--(24) of \cite{Polymath2019} give the additional
$E_{AB},E_C$ majorants at $L_-$, using $\beta_+$ for the positive
$3^y+3^{-y}$ correction and $\mathsf Qe^{\delta U_+}$ in the $AB$ sum.
Put
\[
 E_+=\prod_{p\in P}(1+\rho_p^+),\qquad
 \varepsilon_{\rm cell}
 =E_+\{\mathsf Q(e^{\delta U_+}-1)+E_{AB}+E_C\}.
\]
Applying \cref{lem:thermal-mollifier-overlap} separately to $S_0$
and $S_y$, and then the triangle inequality, yields
\begin{equation}\label{eq:canopy-cell-margin}
 |\mu_P(s_*)|\,|H_t/B_t|
 \ge \Delta_{\rm cell}:=2-B_{\rm cell}-\varepsilon_{\rm cell}.
\end{equation}
The equality $|\overline{S_y}|=|S_y|$ is used here; we do not
multiply the conjugated sum by an incorrectly conjugated mollifier.

The directed program covers all $1993$ rectangles at both $160$
and $256$ bits, with identical rational seams and catalogue choices.
The catalogues end at $5,11,19,29$; their respective cell counts
are $1977,7,6,3$.  Every cell asserts
\[
 \Delta_{\rm cell}>\frac1{1000},\qquad
 \frac{\Delta_{\rm cell}}{E_+}>\frac1{2000}.
\]
The smallest directed values exceed $0.0014164$ and $0.0007914$,
respectively.  No oscillatory phase is sampled; the interval and
monotonicity bounds cover all real coordinates in each rectangle.
This proves \eqref{eq:mollified-canopy-nonzero} up to the handoff.
Beyond $u=10^{16}$ the amplitude bound $|H_t/B_t-1|<0.9071$
in the proof of \cref{thm:centered-phase-tail} is stronger.
Evenness and conjugation give the stated Xi region.
\end{proof}

\begin{corollary}[Global closure with the published time barrier]
\label{cor:cap-closure-018}
Combining \cref{thm:mollified-canopy} with the published finite-RH
verification and rectangular time barrier gives
\begin{equation}\label{eq:cap-closure-strip}
 \Xi_{1/25}(z)\ne0\quad\text{for }|\Imm z|\ge\frac1{10},
 \qquad
 \boxed{\Lambda\le\frac9{50}=0.18.}
\end{equation}
In particular the finite-cap phase-numerator condition of
\cref{cor:finite-phase-cap} holds throughout its stated heat interval.
\end{corollary}

\begin{proof}
Use \cite[Theorem~1.2]{Polymath2019} with $t_0=4/25$ and $y_0=1/5$.
Let $X_B$ be the left edge of the unit time barrier located near
$5\times10^{12}+194858$ in the second row of
\cite[Section~10, Table~1]{Polymath2019}.
The rectangular verification there, performed by the fixed-rectangle
method of Section~8.4, covers $0\le t\le0.186$ and
$y_B\le y\le1$ with $y_B<0.17$.
Restricting it to $t\le4/25$ and $y\ge1/5$ supplies hypothesis (iii),
since $\sqrt{1-y_0^2}<1$.
It is the rectangular verification that is reused: the original
curved barrier hypothesis alone would not imply this restriction.
The whole barrier lies between $5\times10^{12}$ and
$5\times10^{12}+200000$; the half-unit convention for its centre
versus left edge does not affect any inclusion used here.

Platt--Trudgian's \cite[Theorem~1]{PlattTrudgian2021} verifies RH
through height $3{,}000{,}175{,}332{,}800$, which exceeds $X_B/2$.
This supplies hypothesis (i).
For hypothesis (ii), \cref{thm:mollified-canopy} covers every
$u\ge X_B+\sqrt{1-y_0^2}$ and
$y_0\le y\le\sqrt{1-2t_0}$, with a strict overlap at the barrier.
All three hypotheses are therefore satisfied, giving a global
zero-free region at and above $y_0$ and
$\Lambda\le t_0+y_0^2/2=4/25+1/50=9/50$.
The normalization gives the Xi strip in
\eqref{eq:cap-closure-strip}.

The strip-product phase estimate is strict at every $\eta\ge1/10$:
there are no zeros on its boundary.  Hence its continuous numerator
is positive on the entire cap at $\kappa=1/25$.
For later heat times use de Bruijn strip contraction; by
$\kappa=9/200$ all zeros are real.
This proves the last assertion throughout $[1/25,1/20]$.
The finite-RH computation and the published time-barrier run are
external numerical inputs, not replays of the accompanying program.
\end{proof}

The values $0.18$ and $0.1732$ are numerically larger than the companion comparison
$\lambda_{\rm C}=0.148611789321711$ in
\cite{PlanatNewmanClassical2026}.
The latter assumes the explicit model-to-function, localisation,
and uniform-contact premises (A1)--(A3).  The transfer here instead
uses the published finite-RH and time-barrier results together with
the new canopy certificate.
To improve on $\lambda_{\rm C}$ while keeping $h_0=1/10$ would require
\[
 \kappa_0<\frac{\lambda_{\rm C}-1/50}{4}
       =0.03215294733042775.
\]
The present fixed-time canopy is not a certificate at those smaller
times.  At $\kappa_0=1/25$, even $h_0=0$ would leave the ceiling
$4\kappa_0=0.16$.  Thus improving on $\lambda_{\rm C}$ requires reducing the heat
time as well as controlling its canopy and barrier.
None of this transports the calibrated genus-two core beyond its
previously certified local heat interval.

Table~\ref{tab:Newman-comparison} collects the successive ceilings and selected external comparisons.  ``Unconditional'' describes the absence of an unproved zero-location hypothesis; it does not mean that a manuscript has undergone independent refereeing.  The present canopy rows use the published finite-RH and rectangular time-barrier computations.  Their numerical certificates are supplied here.  The conditional and audit-relative rows retain their separate premises; the table is not an exhaustive record survey.

\Needspace{36\baselineskip}
\begingroup\small
\setlength{\tabcolsep}{4pt}
\renewcommand{\arraystretch}{1.12}
\begin{longtable}{@{}>{\raggedright\arraybackslash}p{0.19\textwidth}>{\raggedright\arraybackslash}p{0.36\textwidth}>{\raggedright\arraybackslash}p{0.40\textwidth}@{}}
\caption{Selected upper bounds for $\Lambda$, and the certified progression in this paper.}\label{tab:Newman-comparison}\\
\toprule Ceiling or target & Source & Logical and computational status\\\midrule
\endfirsthead
\multicolumn{3}{l}{\textit{Table~\thetable\ continued}}\\
\toprule Ceiling or target & Source & Logical and computational status\\\midrule
\endhead
\bottomrule\endfoot
$0.22$ & Polymath \cite{Polymath2019} & Published unconditional computer-assisted bound.\\
$0.2$ & Platt--Trudgian \cite[Corollary~2]{PlattTrudgian2021} & Published unconditional bound using the Polymath criterion and their finite-RH verification.\\
$0.1787854$ & Gomila \cite{Gomila2026} & Public computer-assisted claim, unconditional in form; not a new replay or independent audit here.\\
$0.177953255$ & Companion \cite[Corollary~4.3]{PlanatNewmanCompanion2026} & Rounded-up audit-relative ceiling; assumes the stated imported Gomila frontier premises.\\
\midrule
$0.18$ (target) & \Cref{cor:finite-phase-cap} & Sufficient finite-cap sign condition; this reduction alone is not a completed certificate.\\
$0.18$ & \Cref{cor:cap-closure-018} & Unconditional canopy closure with the two published numerical inputs.\\
$0.1791$ & \Cref{thm:canopy-frontier} & Unconditional joint time/height refinement and exterior certificate.\\
$0.17795$ & \Cref{thm:inverse-canopy} & Unconditional finite reciprocal-head convolution and norm transfer.\\
$0.175$ & \Cref{thm:localized-canopy} & Unconditional localization of reference norms.\\
$0.1738$ & \Cref{cor:fine-canopy} & Unconditional finer localization with $J=512$.\\
$0.1732$ & \Cref{cor:extended-head-canopy} & Unconditional extended head $J=1024$; final ceiling of this paper.\\
\midrule
$<0.148612$ & Classical-strip companion \cite{PlanatNewmanClassical2026} & Conditional on its model-to-function, localisation and uniform-contact premises (A1)--(A3).\\
$0.13963$ & Hakimoglu-Brown \cite{HakimogluBrown2026} & Conditional on $H_{7/8}$; retains the published numerical boundary input.\\
\end{longtable}
\endgroup

\subsection{Joint strip-height optimization and its present limits}
\label{subsec:canopy-frontier}

The preceding ceiling can be improved by varying both parameters in
$t_0+y_0^2/2$.  A smaller final ceiling need not use a smaller heat
time: a sufficiently narrower strip can compensate for an increase
in $t_0$.

\begin{theorem}[Optimized directed canopy and global ceiling]
\label{thm:canopy-frontier}
Put
\begin{equation}\label{eq:frontier-parameters}
 t_* = \frac{12959}{80000}=0.1619875,
 \qquad y_* = \frac{37}{200}=0.185,
 \qquad \kappa_* = \frac{12959}{320000}.
\end{equation}
For every $u\ge5\times10^{12}$ and
$y_*\le y\le\sqrt{1-2t_*}$,
\begin{equation}\label{eq:frontier-canopy}
 \left|\frac{H_{t_*}(u+iy)}{B_{t_*}(u+iy)}\right|
       >\frac1{30000}.
\end{equation}
Together with the same published finite-RH and rectangular
time-barrier verifications as in \cref{cor:cap-closure-018}, this gives
\begin{equation}\label{eq:frontier-global}
 \Xi_{\kappa_*}(z)\ne0\quad
       \text{for }|\Imm z|\ge\frac{37}{400},
 \qquad
 \boxed{\Lambda\le\frac{1791}{10000}=0.1791.}
\end{equation}
\end{theorem}
\begin{proof}
Use exactly the finite coefficient inequality of
\cref{lem:thermal-mollifier-overlap} and the rectangular error budget
\eqref{eq:canopy-cell-budget}--\eqref{eq:canopy-cell-margin}, with
$t=t_*$ and $c=t_*/4$.  The generalized engine is supplied in
\path{certificates/canopy_frontier/}.  Its domain, coefficient-shift,
cutoff, paired-gamma and terminal-band guards are checked anew.
The prime catalogues now end at $5,11,19,29,37,43,53$.
With $L_0=\log((5\times10^{12}-1)/(4\pi))$, the exact horizontal
offsets are $j/200$, $0\le j\le40$, followed by
$1/5+j/20$, $1\le j\le149$.  These give $189$ columns ending at
$L_0+153/20>\log(10^{16}/(4\pi))$.
An adaptive rational partition covers $37/200\le y\le33/40$;
the latter upper edge exceeds $\sqrt{1-2t_*}$.

At each of $160$ and $256$ bits, all $3591$ rectangles satisfy
$\Delta_{\rm cell}>1/10000$ and
$\Delta_{\rm cell}/E_+>1/30000$.
The two precisions give identical rational seams and prime choices.
The respective counts for the seven catalogues are
$3324,111,113,23,12,6,2$.
The smallest margins exceed $0.0002457$ and $0.0001296$,
respectively.  Thus \eqref{eq:frontier-canopy} holds throughout
the finite logarithmic range.

The old exterior theorem cannot simply be used below $y=1/5$.
Instead, at $L_h=\log((10^{16}-1)/(4\pi))$, put
\[
 a=\frac{1+y_*}{2}-\delta,
 \qquad b=\frac{1-y_*}{2}-2\delta,
 \qquad \delta=10^{-12}.
\]
Let $M(v,L)$ denote the positive head through $1000$ plus its
decreasing-sum integral tail, as in the preceding proof.  The extra
$\delta$ in $b$ absorbs the conjugate-exponent shift directly.
Consequently the primary and shifted dual amplitudes are bounded by
\[
 P_h=M(a,L_h)-1<0.807010,
 \qquad
 Q_h=e^{-y_*L_h/2+2\delta}M(b,L_h)<0.105657.
\]
The terminal-band tests prove these bounds uniformly for $L\ge L_h$,
and the combined dual weights decrease with $y$.
The AB remainder is less than $4.90\times10^{-16}$ and the C
remainder less than $2.57\times10^{-10}$.
Their endpoint bounds decrease with $L$: the AB exponent has the
form $(a_1L^2+a_2)/(4\pi e^L-a_3)$ with positive constants and
$L>2$, while the C exponential and its positive correction terms
decrease in this range.  Hence
\[
 |H_{t_*}/B_{t_*}|\ge1-P_h-Q_h-E_{\rm AB}-E_{\rm C}
       >0.0873>\frac1{30000}
\]
on the infinite tail.  This completes the canopy proof without
extending the earlier phase theorem beyond its stated domain.

The published rectangle contains this new barrier because
$t_*<0.186$, $y_*>0.17$, and $\sqrt{1-y_*^2}<1$.
The initial RH height and the canopy/barrier overlap are unchanged.
The same Polymath theorem therefore gives the global strip and
\[
 t_*+\frac{y_*^2}{2}
  =\frac{12959}{80000}+\frac{1369}{80000}
  =\frac{1791}{10000}.
\]
The normalization gives \eqref{eq:frontier-global}.
As before, the external finite-RH and time-barrier computations are
reused published inputs, not independent replays.
\end{proof}

The decrease from $0.18$ is exactly $0.0009$.
Here $t_*$ is \emph{larger} than $0.16$ by $0.0019875$, whereas
the strip contribution decreases from $0.02$ to $0.0171125$.
This is why joint optimization improves on varying the time alone.
No smaller ceiling is asserted by that Euler-only full-domain certificate.
Directed point diagnostics with larger prime catalogues are supplied
separately.  They test this particular coefficient budget at specified
parameter pairs; a negative reserve is not a zero and does not prove
an intrinsic lower limit for the method.  For example, at $L=L_0$,
$y_*=37/200$ and $t=12943/80000$ (which would give $0.1789$),
the margins for catalogues through $53$ and $61$ are respectively
less than $-0.00492$ and $-0.00298$ at both precisions.
This diagnoses a failure of the stated budget at that parameter pair,
not an obstruction to every other pair or proof strategy.
The diminishing gains from
successive prime catalogues suggest that substantially larger progress
will require stronger cancellation or a sharper treatment of the
finite coefficient norm, rather than only finer rectangles.
The reciprocal-head construction below supplies a sharper finite norm.
The conditional comparison $0.148611789321711$ remains numerically
smaller; none of its premises (A1)--(A3) is used here.

\subsection{A finite reciprocal head crosses the collective endpoint}
\label{subsec:inverse-canopy}

The loss in the Euler-only estimate is not intrinsic to a canopy
argument.  Instead of bounding each signed component separately,
we now retain a complete finite coefficient convolution.  Put
$b_n=e^{c\log^2n}$ and define the finite Dirichlet-inverse head by
\begin{equation}\label{eq:inverse-head}
 a_1=1,\qquad
 a_n=-\sum_{\substack{d\mid n\\d<n}}a_d b_{n/d},
 \qquad
 \mu_J(s)=\sum_{d\le J}a_d d^{-s}.
\end{equation}
For $N\ge J$, the coefficients of
$\mu_J(s)\sum_{n\le N}b_n n^{-s}$ at indices $2,\ldots,J$
vanish exactly.  No convergence or inversion of an infinite
Dirichlet series is asserted.

\begin{lemma}[Frozen finite-polynomial norm transfer]
\label{lem:inverse-norm-transfer}
Fix $c>0$, $L_r>0$, $y_r>0$, and a positive integer $N_r<e^{L_r/2}$.
Let $a_1=1,a_2,\ldots,a_J$ be real, and put
\[
 \sigma_r=\frac{1+y_r}{2}+cL_r,
 \quad G_r=e^{-y_rL_r/2},
 \quad S_{\nu,N}(s)=\sum_{n\le N}b_n n^{\nu-s}.
\]
Write $\|\cdot\|_1$ for the absolute coefficient mass of a
Dirichlet polynomial at the indicated real part.  Suppose
$I_P,I_Q,\varepsilon_P,\varepsilon_Q,E\ge0$ and
\[
 \|\mu_J S_{0,N_r}\|_{1,\sigma_r}\le I_P+\varepsilon_P,
 \qquad
 G_r\|\mu_J S_{y_r,N_r}\|_{1,\sigma_r}\le I_Q+\varepsilon_Q,
 \quad \varepsilon_P+\varepsilon_Q\le E.
\]
Define
\[
 M=\sum_{d\le J}|a_d|d^{-\sigma_r},\qquad
 Q_1=G_r\sum_{n\le N_r}b_n n^{y_r-\sigma_r}\log n.
\]
Consider $L\ge L_r$, $y_r\le y\le y_r+D$, and
$N_r\le N\le e^{U_+}$, with $U_+<L_r$.
Assume that $b_vv^{-\sigma_r}$ and
$b_vv^{y_r-\sigma_r}$ decrease on $[N_r,e^{U_+}]$, where
$b_v=e^{c\log^2v}$, and that $MQ_1\ge L_r I_Q/2$.
For $\sigma=(1+y)/2+cL$ and $G=e^{-yL/2}$,
\begin{align}
 \|\mu_J S_{0,N}\|_{1,\sigma}
 +G\|\mu_J S_{y,N}\|_{1,\sigma}
 &\le I_P+e^{-DL_r/2}I_Q+E+DMQ_1+MT_+,
 \label{eq:inverse-frozen-budget}\\
 T_+&=\int_{N_r}^{e^{U_+}}
   b_v\bigl(v^{-\sigma_r}+G_rv^{y_r-\sigma_r}\bigr)\,dv.
 \nonumber
\end{align}
\end{lemma}
\begin{proof}
For the fixed cutoff and fixed $\nu$, every coefficient of
$\mu_J S_{\nu,N_r}$ is independent of $\sigma$.  Increasing $\sigma$
therefore decreases its absolute coefficient mass.  This immediately
handles the primary term, without estimating derivatives of the
two factors separately.

Put $h=y-y_r$.  In the dual term, split $n^{y_r+h}$ into
$n^{y_r}+n^{y_r}(n^h-1)$.  The first part has mass at most
$e^{-hL_r/2}(I_Q+\varepsilon_Q)$: the additional real-part
reweighting and the additional heat-height factor both decrease it.
For the second part use
\[
 e^{-hL_r/2}(n^h-1)\le h\log n,
 \qquad n\le N_r<e^{L_r/2},\quad h\ge0.
\]
The convolution inequality bounds it by $hMQ_1$.
The resulting upper bound
$I_P+e^{-hL_r/2}I_Q+E+hMQ_1$ increases for $0\le h\le D$
by the stated guard, giving its endpoint value at $D$.

The added indices above $N_r$ have total primary and dual mass at
most $T_+$ by the decreasing-sum integral test.
For them, the primary weights decrease as $L,y$ increase; the dual
weights do likewise because $\log n\le U_+<L_r\le L$.
The mollifier mass is at most $M$.
This gives the last debit in \eqref{eq:inverse-frozen-budget}.
\end{proof}

\begin{theorem}[Reciprocal-head canopy and global ceiling]
\label{thm:inverse-canopy}
Put
\[
 t_\dagger=\frac{12867}{80000}=0.1608375,
 \qquad y_\dagger=\frac{37}{200},
 \qquad \kappa_\dagger=\frac{12867}{320000}.
\]
For every $u\ge5\times10^{12}$ and
$y_\dagger\le y\le\sqrt{1-2t_\dagger}$,
\[
 \left|\frac{H_{t_\dagger}(u+iy)}{B_{t_\dagger}(u+iy)}\right|
       >\frac1{30000}.
\]
Using the published initial-RH and rectangular time-barrier
verifications gives
\begin{equation}\label{eq:inverse-global-ceiling}
 \Xi_{\kappa_\dagger}(z)\ne0
       \quad\text{for }|\Imm z|\ge\frac{37}{400},
 \qquad
 \boxed{\Lambda\le\frac{3559}{20000}=0.17795.}
\end{equation}
\end{theorem}
\begin{proof}
Set $c=t_\dagger/4$, $J=128$, and
$L_r=\log((5\times10^{12}-1)/(4\pi))$.
Choose $N_r=\lfloor e^{L_r/2}-1\rfloor=630782$ and
$y_r=y_\dagger$.  The program computes the coefficients
\eqref{eq:inverse-head} in Arb ball arithmetic and quantizes the
weighted mollifier coefficients at scale $2^{27}$ and the two
weighted source arrays at scale $2^{32}$.
It then enumerates their entire finite convolution through
$JN_r=80740096$, using exact signed integer multiplication and
addition.  Prior exact bounds exclude signed 64-bit overflow in
each product and partial sum and unsigned 64-bit overflow in each
absolute-value reduction.  No floating-point array operation
contributes to these coefficient norms.

The integer absolute masses, divided by $2^{59}$, are
\[
 I_P=\frac{827653878373035760}{2^{59}},
 \qquad
 I_Q=\frac{257672719430126216}{2^{59}}.
\]
For source-array errors $E_P,E_Q$ and weighted-coefficient error
$E_\mu$, the convolution inequality gives the total quantization
debit
\[
 E\le M(E_P+E_Q)
       +E_\mu\bigl(\|P_{\rm int}\|_1+\|Q_{\rm int}\|_1\bigr),
\]
where the integer arrays on the right are divided by $2^{32}$.
Every rounding error is enclosed in Arb before summation.
At both $160$ and $256$ bits, $E<0.000282$ and
\[
 I_P+I_Q+E<1.883023,
 \qquad M<1.902432,
 \qquad Q_1<4.731960.
\]
The monotonicity guard in \cref{lem:inverse-norm-transfer} is strict.
Both precisions produce the same integer masses and array hashes.
This is a consistency check, not an independent transcendental
arithmetic implementation.

Apply \eqref{eq:inverse-frozen-budget} on the short low corner.
The real-part loss $\delta=10^{-12}$ multiplies the ideal norm by
at most $\exp(\delta(\log J+U_+))$, because the output indices
are at most $Je^{U_+}$.  The paired gamma loss contributes at most
$e^{2\delta}$; the conjugate-exponent shift is charged separately
by the positive dual mass, as in the previous canopy proof.
The AB and C remainders are multiplied by the actual mollifier
bound $M e^{\delta\log J}$.
For completeness, the dual triangle step uses
$|\mu_J\overline{S_y}|=|\mu_J|\,|S_y|=|\mu_J S_y|$;
it does not assert a cancellation between incorrectly conjugated
polynomials.

The complete cover has $153$ logarithmic columns of width $1/20$
and adaptive exact rational height seams in $[37/200,33/40]$.
All $2449$ rectangles pass at both precisions.
Nine use the reciprocal-head budget, all at heights at most $1/5$.
The other $2440$ use \cref{lem:thermal-mollifier-overlap}, with
catalogues through $5,11,19,29$ in respective counts $2414,17,8,1$.
Every accepted cell has mollified margin $>1/10000$ and
true normalized amplitude $>1/30000$ after all error debits.
The smallest recorded margins exceed $0.000297$ and $0.000161$,
respectively.  The geometry and method choices agree at both
precisions.  The changing integer cutoff is covered by $T_+$,
rather than by sampling selected cutoff values.

At $u\ge10^{16}$, a fresh terminal-band amplitude calculation at
$(t_\dagger,y_\dagger)$, with the exponent shift absorbed as in
\cref{thm:canopy-frontier}, gives $|H/B|>0.0698$.
The last log edge overlaps this infinite tail, completing the
canopy.  The same published rectangle applies since
$t_\dagger<0.186$, $y_\dagger>0.17$, and
$\sqrt{1-y_\dagger^2}<1$.  The finite-RH height and horizontal
overlap are unchanged.  Polymath's global criterion now gives
\[
 t_\dagger+\frac{y_\dagger^2}{2}
   =\frac{12867+1369}{80000}
   =\frac{3559}{20000}.
\]
The normalization yields \eqref{eq:inverse-global-ceiling}.
The initial-RH and time-barrier runs remain external published
inputs; neither is independently replayed here.
\end{proof}

The new ceiling is below the audit-relative collective-contraction
endpoint $0.177953254642861\ldots$ of
\cite[Corollary~4.3]{PlanatNewmanCompanion2026} by approximately
$3.25464\times10^{-6}$.
The earlier endpoint consumes the stated imported Gomila frontier
premises; the present transfer instead uses the newly certified
finite polynomial norm and the published Polymath and
Platt--Trudgian inputs.  This compares different proof architectures,
not independently accepted numerical records.
The gain uses global coefficient information and does not continue
the local genus-two core to these heat times.

\begin{remark}[What the QRH input changes here]
\label{rem:QRH-inverse-canopy}
Under $H_{7/8}$, the functional equation places the initial $H_0$
zeros in $|\Imm z|\le3/4$.  De Bruijn contraction then gives
\[
 |\Imm z(H_t)|\le
       \sqrt{\max\{9/16-2t,0\}}.
\]
At $t=t_\dagger$, this lowers the cap top from
$\sqrt{27133}/200\approx0.823605$ to
$\sqrt{9633}/200\approx0.490739$.
It reduces the region requiring an additional canopy check, but
does not remove any of the nine reciprocal-head rectangles, which
all lie at $y\le0.2$.  In particular it does not improve
\eqref{eq:inverse-global-ceiling} merely by truncating this cover.
The classical contraction consequence of $H_{7/8}$ alone is
$\Lambda\le9/32$, already weaker than the present ceiling.
QRH may help a different quantitative contraction or barrier argument;
no such additional saving is certified here.  The directed geometry
check treats $H_{7/8}$ as an external hypothesis and does not audit
the QRH release.  It also supplies no new irrationality statement
for the local period ratios.
\end{remark}

\subsection{Localized reciprocal heads: the ceiling \texorpdfstring{$7/40$}{7/40}}
\label{sec:localized-inverse}

The preceding proof transports one reference norm across several columns.
This can lose reserve even when the norm at a nearby reference is smaller.
The following refinement uses the same transfer lemma with several reference
pairs; no optimization of the inverse coefficients is needed.

\begin{theorem}[Localized reciprocal-head canopy]\label{thm:localized-canopy}
Put
\[
 t_\ell=\frac{12631}{80000}=0.1578875,\qquad
 y_\ell=\frac{37}{200},\qquad
 \kappa_\ell=\frac{12631}{320000}.
\]
For every $u\ge5\times10^{12}$ and
$y_\ell\le y\le\sqrt{1-2t_\ell}$,
\[
 \left|\frac{H_{t_\ell}(u+iy)}{B_{t_\ell}(u+iy)}\right|>
 \frac1{30000}.
\]
With the same published initial-RH and rectangular time-barrier inputs as
\cref{thm:inverse-canopy},
\begin{equation}\label{eq:localized-global-ceiling}
 \Xi_{\kappa_\ell}(z)\ne0\quad
 \text{for }|\Imm z|\ge\frac{37}{400},
 \qquad \boxed{\Lambda\le\frac7{40}=0.175.}
\end{equation}
No QRH hypothesis is used.
\end{theorem}
\begin{proof}
Use the exact inverse head \eqref{eq:inverse-head} with $J=512$ and
$c=t_\ell/4$.  At each reference choose
\[
 L_r=\log\frac{5\times10^{12}-1}{4\pi}+\ell_r,
 \qquad N_r=\lfloor e^{L_r/2}-1\rfloor,
 \qquad \sigma_r=(1+y_r)/2+cL_r.
\]
The eleven pairs $(\ell_r,y_r)$ are
\[
\begin{split}
 &(0,37/200),\ (0,31/160),\ (0,303/1600),\\
 &(1/10,37/200),\ (1/10,1/5),\ (1/10,313/1600),\\
 &(1/5,37/200),\ (3/10,37/200),\ (2/5,37/200),\\
 &(1/2,37/200),\ (3/5,37/200).
\end{split}
\]
These are reference locations, not numerical samples used to infer a
uniform sign.  Each accepted rectangle satisfies $L\ge L_r$, $y\ge y_r$;
\cref{lem:inverse-norm-transfer} bounds the whole rectangle, including
all intervening integer cutoffs.  Different rectangles may use different
mollifiers, since each proves nonvanishing of the same function.

The complete reference convolution and quantization procedure is that of
\cref{thm:inverse-canopy}, with the same integer scales and overflow guards.
For the first pair, the two exact integer masses are
\[
 I_P=\frac{832080961139007258}{2^{59}},\qquad
 I_Q=\frac{286385155079489980}{2^{59}}.
\]
The total quantization debit is less than $0.000297$, and
$I_P+I_Q+E<1.940527$.
All eleven reference receipts retain both integer masses, array hashes,
coefficient bounds and directed quantization errors.

The $153$ logarithmic columns of width $1/20$ are covered up to
$y=17/20$, which exceeds $\sqrt{1-2t_\ell}$.
The exact-rational adaptive cover consists of $2483$ rectangles:
$51$ reciprocal-head rectangles and $2390,26,11,5$ Euler rectangles
with catalogues ending at $5,11,19,29$, respectively.
The directed real-part, gamma, exponent-shift, AB and C debits are
charged exactly as in \cref{thm:inverse-canopy}.
Every mollified margin exceeds $0.0002248$ and every normalized amplitude
lower bound exceeds $0.0001194>1/30000$.
At $u\ge10^{16}$ the fresh terminal-band exterior bound gives
$|H/B|>0.02185$; all its monotonicity guards are strict.
The finite cover overlaps that exterior.

Both $160$- and $256$-bit runs pass.  They give identical integer
reference arrays and masses, the same rectangle assignments and seams,
and overlapping directed cell intervals.  The program and separate
receipt auditor are in \path{certificates/canopy_localized/}.
The auditor checks every cell, reference assignment and rational seam;
the two precision runs are not independent implementations.

Finally $t_\ell<0.186$, $y_\ell>0.17$ and the required barrier width
$\sqrt{1-y_\ell^2}<1$ fit the same published rectangular run.
The finite-RH height and horizontal overlap are unchanged.
The global criterion therefore gives
\[
 t_\ell+\frac{y_\ell^2}{2}
 =\frac{12631+1369}{80000}=\frac7{40}.
\]
This proves \eqref{eq:localized-global-ceiling}.  The two external
numerical verifications are adopted from their published sources,
not independently replayed here.
\end{proof}

\begin{remark}[Transfer losses and the scope of a failed norm]\label{rem:localized-limits}
For $t=t_\ell$, on the rectangle with logarithmic offsets
$[1/4,3/10]$ and $y\in[0.185,0.19]$, the original reference at offset
zero gives margin less than $-0.0853$, whereas the reference at offset
$1/4$ gives margin greater than $0.0867$.
The cutoff debit decreases from more than $0.1243$ to less than $0.01835$.
The total improvement also captures the favorable change of the reference
norm with height; it is not entirely a cutoff effect.

Conversely, at the trial ceiling $0.173$, with $y_r=0.185$, $J=512$
and the initial reference cutoff $N_r=630782$, the absolute reference
coefficient norm is greater than $2.01715$.
The exact quantized norm minus its Arb transfer debit proves this lower
bound.  Narrower rectangles alone cannot make this particular reference
pass the sufficient norm test.  This is neither an optimality theorem
nor an obstruction to other heads, other heights, or phase-sensitive
bounds.  At $0.174$, narrow local rectangles still pass; \cref{cor:fine-canopy}
completes a stronger global certificate.
At the improved time $t_\ell$, the same conditional QRH comparison
lowers the classical cap top from $\sqrt{27369}/200$ to
$\sqrt{9869}/200$.  All $51$ reciprocal-head rectangles have
$y\le43/200<\sqrt{9869}/200$, so none is removed by this truncation.
\end{remark}

\begin{corollary}[Finer localization below $0.174$]\label{cor:fine-canopy}
Set
\[
 t_f=\frac{2507}{16000}=0.1566875,\qquad
 y_f=\frac{37}{200},\qquad
 \kappa_f=\frac{2507}{64000}.
\]
For $u\ge5\times10^{12}$ and $y_f\le y\le\sqrt{1-2t_f}$,
\[
 \left|\frac{H_{t_f}(u+iy)}{B_{t_f}(u+iy)}\right|>\frac1{30000}.
\]
With the same published external inputs as \cref{thm:localized-canopy},
\begin{equation}\label{eq:fine-global-ceiling}
 \Xi_{\kappa_f}(z)\ne0\quad\text{for }|\Imm z|\ge\frac{37}{400},
 \qquad \boxed{\Lambda\le\frac{869}{5000}=0.1738.}
\end{equation}
\end{corollary}
\begin{proof}
Keep $J=512$ and apply \cref{lem:inverse-norm-transfer} to $37$
local reference pairs, with every reference location and exact integer
mass retained in \path{certificates/canopy_fine/}.
At the initial reference $I_P+I_Q+E<1.986018$.
The finer logarithmic grid has width $1/100$ up to offset $1/5$,
then width $1/20$ through offset $199/20$.
Its $215$ columns and adaptive rational height seams cover $3571$
rectangles up to $y=17/20$.
Of these, $199$ use reciprocal heads; the Euler catalogues ending at
$5,11,19,29$ cover $3313,25,27,7$ rectangles, respectively.
The real-part, gamma, cutoff, conjugate-exponent and AB/C remainder
allowances are those of \cref{thm:inverse-canopy}.
Every mollified margin exceeds $0.0001320$, and every amplitude
lower bound exceeds $0.0000797>1/30000$.

A fresh exterior calculation at $t_f$, now starting at $u=10^{17}$,
gives $|H/B|>0.1929$ with strict monotonicity guards.
The finite logarithmic endpoint overlaps this exterior.
The older phase-tail statement for $\kappa\ge1/25$ is not extended
below its domain or used at the new time.
The $160$- and $256$-bit reference computations give identical
integer masses and array hashes.  Replaying the same finite plan at
both precisions checks every cell and exact rational seam.
These are consistency replays with the same arithmetic library.

The same published time-barrier rectangle contains $t_f<0.186$,
$y_f>0.17$, and the required width $\sqrt{1-y_f^2}<1$.
The initial finite-RH height and horizontal overlap are unchanged.
Consequently the global criterion gives
\[
 t_f+\frac{y_f^2}{2}
 =\frac{12535+1369}{80000}=\frac{869}{5000}.
\]
The external verifications are not independently rerun, and no QRH
hypothesis is used.  This improves the certificate without changing
the head length or using LP coefficient optimization.
\end{proof}

The earlier estimate near $0.174$ was therefore a conservative working
frontier, not a saturation theorem.  The failed $0.173$ reference in
\cref{rem:localized-limits} still concerns only its specified head,
height and sufficient norm test; no exact saturation constant follows.

\Needspace{14\baselineskip}
\begin{corollary}[Extended reciprocal head]\label{cor:extended-head-canopy}
Set
\[
 y_e=\frac{18257}{100000}=0.18257,\qquad
 t_e=\frac{3130681951}{20000000000}=0.15653409755,\qquad
 \kappa_e=\frac{t_e}{4}.
\]
For $u\ge5\times10^{12}$ and $y_e\le y\le\sqrt{1-2t_e}$,
\[
 \left|\frac{H_{t_e}(u+iy)}{B_{t_e}(u+iy)}\right|>\frac1{30000}.
\]
With the same published external inputs as \cref{thm:localized-canopy},
\begin{equation}\label{eq:extended-head-global-ceiling}
 \Xi_{\kappa_e}(z)\ne0\quad\text{for }|\Imm z|\ge\frac{18257}{200000},
 \qquad \boxed{\Lambda\le\frac{433}{2500}=0.1732.}
\end{equation}
\end{corollary}
\begin{proof}
Apply \cref{lem:inverse-norm-transfer} with head length $J=1024$
and $59$ local reference pairs.  The programs, exact integer masses,
array hashes and directed receipts are retained in the ancillary
\path{certificates/canopy_1024/} package.
The logarithmic grid has $215$ columns: width $1/100$ up to offset
$1/5$, then width $1/20$ through offset $199/20$.
Exact rational height seams partition the canopy from $y_e$ to $17/20$
into $3741$ rectangles.  Of these, $341$ use the reciprocal head;
the Euler catalogues ending at $5,11,19,29$ cover
$3315,45,27,13$ rectangles, respectively.
The quantization-transfer debit and every real-part, gamma, cutoff,
conjugate-exponent and AB/C remainder allowance are included.
The integer convolutions use explicit overflow guards.
Every finite margin exceeds $0.00011010472$, and every normalized
amplitude lower bound exceeds $0.00005539885>1/30000$.

A fresh exterior estimate at $t_e$, starting at $u=10^{17}$,
gives $|H/B|>0.18596685$, with all monotonicity guards strict.
The finite logarithmic endpoint overlaps this exterior, and
$17/20>\sqrt{1-2t_e}$.
The retained fresh computations at $192$ and $256$ bits agree on
integer masses, array hashes and rectangle geometry.  For this release,
the receipt checker and both interval audits were rerun; these audits
recompute every rectangle inequality, the exact partition and the
exterior from the saved reference norms.  They do not regenerate the
full reference convolutions and are not independent implementations.

The published rectangular time barrier contains $t_e<0.186$,
$y_e>0.17$ and width $\sqrt{1-y_e^2}<1$.
The same finite-RH height and horizontal overlap therefore close
Polymath's global criterion, exactly as in \cref{thm:localized-canopy}.
Finally,
\[
 t_e+\frac{y_e^2}{2}
 =\frac{3130681951+333318049}{20000000000}
 =\frac{433}{2500}.
\]
The external finite-RH and time-barrier computations are adopted from
the published sources, not rerun.  No QRH hypothesis or LP coefficient
optimization is used.
\end{proof}

\section{Certification architecture}\label{sec:certification}

We summarize the validated-numerics ingredients because they are part of the theorem statements, not merely exploratory computation.

\subsection{Theta tails and directed moments}
The theta sum \eqref{eq:Phi} is truncated after finitely many terms and the remainder is bounded analytically using the doubly exponential majorant
\begin{equation}\label{eq:Phi-majorant}
 \Phi(u)<4\pi^2\exp\left(\frac92u-\pi e^{2u}\right),
 \qquad u\ge0.
\end{equation}
The moments $M_{2n}$ are integrated on a finite interval with outward-rounded symmetric Taylor quadrature, while the outer $u$-tail is enclosed by \eqref{eq:Phi-majorant}.  The current high-precision run uses a common directed table through $M_{80}$.  Since the same moment intervals feed every period, all branch and period certificates remain independent of any Riemann-zero table.

\subsection{Regularization at \texorpdfstring{$T=0$}{T=0}}
The formula \eqref{eq:Delta-general} contains a removable quotient by $T$.  Direct interval propagation across an interval containing zero would be badly conditioned.  Proposition~\ref{prop:regular-K} replaces it by $K/Y^6$.  On the small symmetric central segment, an even Taylor jet of $\sqrt{\Delta_0(T)}$ is integrated analytically; odd contributions cancel.  In the computation underlying \cref{thm:farey}, the resulting central integral enclosure has width below $10^{-24}$, so $T=0$ is no longer a significant error source.

\subsection{Complex branch points}
At a proposed center $c_0$, write
\[
 K(c_0+w)=k_0+k_1w+\sum_{m\ge2}k_mw^m.
\]
A directed Rouch\'e test proves
\[
 |k_0|+\sum_{m\ge2}|k_m|\rho^m<|k_1|\rho
\]
on $|w|=\rho$.  This certifies exactly one simple zero in the disk.  A separate lower enclosure for $|Y|$ excludes a transverse pole.  For both certified nonreal pairs, the final branch-disk radius is $\rho=10^{-9}$.

\subsection{Heat jets, branch contraction, and the projective tangent certificate}
The heat program uses the exact gamma--zeta representation of $\xi$ in
Arb/Acb Taylor arithmetic through \texttt{python-flint} 0.9.0.  The identity
$M_{2n}=\xi^{(2n)}(1/2)/4$ relates its central jets to the theta moments;
the program does not import the earlier numerical moment table.  Four
Rouch\'e checks independently refine the upper and real selected roots,
with conjugation supplying the remaining two.  All source centers are
exact decimal rationals; their role is to propose disks, and acceptance
requires the displayed strict Rouch\'e inequalities.

The shifted jets give $\partial_\kappa K$, while automatic polynomial
differentiation along \eqref{eq:D6-flow} gives the scalar Lax tangent.
The input for the latter is the complete static parameter box of
\cref{thm:D6-certified-realization}; its existence remains the content of
that earlier realization theorem.  The new program does not replace its
Krawczyk proof.  Each derivative denominator is excluded from zero, and
the rational rectangle for $\mathcal D$ is checked by strict interval
comparisons.  The supplement records successful 192-bit and 256-bit
replays.  Agreement between these precisions is a consistency check, not
independence of transcendental arithmetic implementations.
The separate contraction program reuses the refined-root and heat-jet
identities, evaluates all fifteen logarithmic rates in the natural $T$
coordinate, and checks $c_{jk}>1/9$ by strict directed comparisons.  It
also checks both representative rational intervals in
\eqref{eq:heat-branch-contraction-examples}; its 192-bit and 256-bit logs
and JSON receipts are included.

\subsection{Contour periods}
Away from branch endpoints, $\Delta_0$ is expanded on certified zero-free disks.  Interval Taylor coefficients for the selected branch of $\sqrt{\Delta_0}$ are obtained recursively from
\[
 \left(\sum_{n\ge0}b_nw^n\right)^2=\sum_{n\ge0}d_nw^n.
\]
Even terms are integrated exactly on symmetric panels.  A Cauchy bound on a larger certified disk controls the Taylor remainder.  Adjacent square-root image disks overlap on the same sheet and are disjoint from the opposite-sheet disks; this propagates the branch without repeated pointwise reevaluation.  Simple-root endpoint tails are treated separately.

\subsection{Directed heat derivatives of periods}\label{sec:heat-period-certification}
The period engine constructs entire Taylor models for
$Y,E,\dot E$ from central gamma--zeta jets, with rigorous global Cauchy
tails on $|T|=4096$.  This handles the removable division by $T$ in
$E$ and $\dot E$ before contour evaluation, including panels that
cross zero.  The exact shifted heat jets produce $\dot K$ and
$\dot\Delta$.  Interior panels integrate the continued functions
$\sqrt{\Delta_0}$ and $\dot\Delta/(2\sqrt{\Delta_0})$ together.

For a simple endpoint $T_j$, write $K(T)=(T-T_j)k(T)$.  On a disk
containing the root and the omitted segment, the certificate proves
\[
 m\le|k(T)|\le M,\quad |Y(T)|\ge y>0,\quad
 |\dot K(T)|\le B,\quad |\dot Y(T)|\le H,
\]
with $m>0$.  The bounds on $k$ follow from a nonzero enclosure for
$K'$ at the rational root center and a Taylor--Cauchy variation bound.
For an omitted segment of length at most $\ell$, including root-location
uncertainty, the two endpoint errors are bounded by
\begin{align}
 \left|\int_{\rm tail}\varpi_0\right|
 &\le\frac{2\sqrt M}{3y^3}\ell^{3/2},\label{eq:heat-period-base-tail}\\
 \left|\int_{\rm tail}\dot\varpi\right|
 &\le\frac{B}{y^3\sqrt m}\ell^{1/2}
       +\frac{2\sqrt M H}{y^4}\ell^{3/2}.
       \label{eq:heat-period-derivative-tail}
\end{align}
These follow by integrating the pointwise square-root and
inverse-square-root majorants.  In particular, the derivative tail is
bounded analytically rather than discarded because its endpoint moves.
Both precision replays reproduce period boxes strictly inside the
original calibration boxes and verify the six rational derivative
intervals in \cref{tab:heat-action-motion}.  The countable-dense-set
conclusion of \cref{cor:heat-generic-nonresonance} is then an analytic
consequence of monotonicity and the kernel criterion, requiring no
sampling of rational heat times.
The simple-resonance proposition and generic Diophantine theorem are
further analytic deductions from the same inputs.  Their measure estimate
uses a summable rational-approximation covering, rather than a new
numerical certificate.  The negative-heat corollary combines this local
arithmetic with the independently established Rodgers--Tao lower bound.

\subsection{Uniform heat models and explicit Farey persistence}
\label{sec:uniform-heat-certification}
The uniform heat engine constructs heat Taylor coefficients through
order thirteen from exact central jets.  If
$Y_0(T)=\sum_{n\ge0}c_nT^n$, then
\[
 [\kappa^j]Y_\kappa(T)
 =\sum_{n\ge0}\frac{(2n+2j)!}{(2n)!j!}\,c_{n+j}T^n.
\]
Products and the normalization $A_*(\kappa)$ are formed in truncated
heat series; the exact removable constant term of $E$ is shifted away
before evaluation.  Spatial degree 85 for $Y$ and degree 80 for $E$
are retained.  The spatial jets of each heat coefficient are evaluated
first, and $K$ is then formed in heat-series arithmetic.  Only afterwards
is $\kappa$ replaced by the real ball $[-10^{-9},10^{-9}]$.
This order retains cancellations lost by independent interval
substitution into the primitive factors of $K$.

Joint Cauchy estimates use the complex heat disk
$|\kappa|\le\rho=10^{-3}$, the global spatial circle $|T|=4096$,
and the local circle $|T|=512$.  Positivity of the theta kernel bounds
$Y$ and its derivatives on these circles.  Split its integral at
$u=8$: on the finite part, $|e^{\kappa u^2}|\le e^{64\rho}$;
on the outer part, \eqref{eq:Phi-majorant} bounds integrals with
additional factors $u^m e^{\rho u^2+vu}$.
The logarithmic derivative of this majorant is decreasing beyond eight
and negative there, giving an explicit first-point divided-by-slope
bound for the integral.  Moreover,
\[
 |Y_\kappa(0)|\ge(2-e^{64\rho})Y_0(0)
 -\int_8^\infty e^{\rho u^2}\Phi(u)\,\dd u>0.
\]
Thus $A_*(\kappa)$ is holomorphic on this complex heat disk, and
triangle inequalities also bound the entire functions $E$ and $K$.

Write $q=10^{-9}/\rho$.  If $M_L$ bounds an entire spatial function
$F_\kappa$ on $|T|=512$, uniformly on the complex heat disk, its
heat remainder after order twelve and the corresponding derivative
remainder are bounded by
\begin{align*}
 \left|F_\kappa-\sum_{j=0}^{12}F_j\kappa^j\right|
 &\le\frac{M_Lq^{13}}{1-q},\\
 \left|\partial_\kappa F_\kappa
 -\sum_{j=0}^{12}(j+1)F_{j+1}\kappa^j\right|
 &\le\frac{M_L}{\rho}
       \frac{q^{13}(14-13q)}{(1-q)^2},
 \qquad F_j=[\kappa^j]F_\kappa.
\end{align*}
Cauchy bounds on the remaining spatial radius control derivatives
of these remainders; the global circle controls omitted spatial
coefficients of each heat Taylor coefficient.  Both error terms are
included in every jet enclosure.

For the moving-root certificate, the center value is bounded with
its common heat dependence:
\[
 |K_\kappa(c)|\le |K_0(c)|
 +10^{-9}|\partial_\kappa K_0(c)|
 +\tfrac12\,10^{-18}
       \sup_{|s|\le10^{-9}}|\partial_s^2K_s(c)|.
\]
A uniform nonzero linear spatial jet and a second-derivative bound
then give the Rouch\'e test on radius $10^{-7}$.
The same uniform jets and endpoint majorants control the continued
periods and their heat derivatives.  Each Farey form is bounded by
integrating its combined derivative, retaining the common heat time
instead of assigning an independent time to each period coordinate.
Both precision replays check the explicit interval, all four strict
form boxes, and both positive velocity boxes in
\cref{thm:explicit-Farey-heat}.

\subsection{Directed exterior phase constants}
The \path{certificates/phase_tail/} directory contains two standalone
Arb programs with \texttt{python-flint==0.9.0}, a pinned requirement,
and successful $160$- and $256$-bit logs and JSON receipts for each.
The baseline evaluates the two positive heads ending at $n=100$, the four convex
endpoint tails, the two approximation-error majorants, and the
prefactor phase reserve from \cref{thm:uniform-phase-tail}.
Every published rational bound is a strict program assertion.
It also asserts each domain, loss, endpoint, and monotonicity guard
used in the proof.  The uniform conclusion over unbounded horizontal
coordinates and the heat interval follows from the analytic
monotonicity and coefficient comparisons, rather than from sampling
points or times.  Cauchy is applied to the exact $H_t/B_t-1$ on
the expanded disks, so no derivative is taken across a change of
the Riemann--Siegel cutoff.  Agreement at two precisions is a
consistency check within the same arithmetic implementation.
The centred refinement uses heads ending at $n=1000$, exact
$\operatorname{erfi}$ mass integrals, and their logarithmic and centred
moments.  Its $74$ strict assertions also check the terminal-band
monotonicity guards, the one-pair cutoff-jump bound, and every
published constant in \cref{thm:centered-phase-tail}.
It differentiates a holomorphic model with the centre's cutoff fixed;
Cauchy bounds only its remainder after the possible jump is included.
In particular, the expanded model's amplitude bound need not be
less than one.  The exact centre amplitude bound proves
nonvanishing before the phase estimate is applied.
The earlier baseline program and receipts are retained unchanged.
Neither exterior program checks the finite cap directly.
The following canopy certificate supplies the complementary input
for \cref{cor:cap-closure-018}.

\subsection{Mollified canopy and published-input transfer}
The \path{certificates/finite_cap/} directory contains the directed
canopy engine and covering driver, a pinned requirement, $160$- and
$256$-bit logs and full cell receipts.  All $1993$ rectangles pass
the true-function nonvanishing margins; exact rational seams cover
the entire finite logarithmic range and the height interval.
Thermal coefficient positivity and the tail-overlap formula have
an analytic proof in \cref{lem:thermal-mollifier-overlap}.
A separate finite enumeration checks twelve small convolutions at
each precision and rejects an incorrect overlap range.
An exact-rational checker records the domain inclusions for the
global transfer.  It checks neither the original time-barrier run
nor Platt--Trudgian's large finite-RH computation.
Those two published results are explicitly cited in
\cref{cor:cap-closure-018}; only the canopy is newly computed here.
The separate \path{certificates/canopy_frontier/} directory generalizes
the same finite coefficient budget to \cref{thm:canopy-frontier}.
Its $3591$ rectangles, fresh infinite-tail bound and exact transfer
inclusions pass at both precisions.  Separate point diagnostics up
to prime $61$ document failed budgets without interpreting them
as zeros or optimality statements.
The additional \path{certificates/canopy_inverse/} directory supplies
the finite inverse head and exact integer convolution, all $2449$
covering rectangles at both precisions, a fresh infinite tail and a
separate receipt/index auditor. Its QRH geometry checker assumes
$H_{7/8}$ only to compare cap heights; the new ceiling does not use it.

The localized extension in \path{certificates/canopy_localized/}
retains all eleven reference-pair receipts, the $2483$-rectangle covers
at $160$ and $256$ bits, a fresh exterior bound, a separate exact-seam
and interval auditor, and point/transfer diagnostics.
The enlarged rank test in \path{certificates/lax_enlarged/} recertifies
the branch disks and proves the rational determinant and rate bounds
at $192$ and $256$ bits.  Its inverse-function consequence concerns
local branch moduli; no Stokes-data certificate is supplied.

The fine extension in \path{certificates/canopy_fine/} retains
the $37$ reference pairs, full $3571$-cell receipts at both precisions,
a fresh $10^{17}$ exterior, a portable frozen-plan replay and an auditor.

\subsection{Arithmetic aggregation}
The period intervals themselves are not the final object.  Integer forms such as
\[
 419336u_1-13749u_2
\]
are aggregated with outward rounding, and the sign is accepted only when the resulting interval is strictly separated from zero.  Farey-neighbor determinants are exact integer identities.  The supplementary certificate prints the four decisive linear-form intervals and both determinant checks.

\section{Discussion and open arithmetic problems}\label{sec:discussion}

The construction links three levels of description: a real genus-two core extracted from the Xi discriminant, its calibrated action arithmetic, and the global phase of the heat-deformed entire function.  Table~\ref{tab:certificate} records the certified connections.  The principal open questions concern the distinguished arithmetic point, rather than further finite numerical refinement.

The first is \cref{conj:nonresonance}: are $u_1/u_2$ and $r_1/r_2$ both irrational at $\kappa=0$?  The Farey barriers exclude rational relations up to substantial denominator height, but finite-width enclosures cannot establish irrationality.  Positive heat velocities give complementary generic statements, including almost-everywhere Diophantine bounds, whose constants depend on heat time and deteriorate near the dense resonant set.  They do not single out zero.  A direct arithmetic interpretation of the theta-moment action coordinates, or a functional identity giving fixed-point Diophantine bounds, would address this gap.  Standard number-field period theorems are not directly applicable because no algebraic number-field model of this Xi-derived core is known.

The second is \cref{conj:v9-minimal-Picard}.  The finite Humbert exclusions and the Mordell--Weil height barrier constrain possible Picard enhancements, while \cref{thm:v8-global-Humbert} reduces their complete exclusion to irrationality of $\det\Imm\Omega$ and $\QQ$-linear independence of the three entries of $\Imm\Omega$.  Humbert divisors are dense on the real slice, so extending the finite exclusion ladder cannot settle this question.  The numerical period matrix is consistent with minimal rank, but a suitable exact arithmetic argument remains necessary.

The $D_6$ correspondence clarifies the deformation problem without resolving either arithmetic conjecture.  The heat tangent differs from the matched fixed-exponent isomonodromic tangent.  Allowing the two exponent directions gives a local lift of the branch moduli, but does not identify the full scalar potentials, preserve Stokes data, or establish global continuation.  A global positive canonical system with Weyl--Titchmarsh function
\[
 m(z)=-\frac{\Xi'(z)}{\Xi(z)}
\]
would require an additional Herglotz/positivity argument beyond this local curve correspondence.  A Herglotz function is a holomorphic function mapping the upper half-plane into itself; a Weyl--Titchmarsh function records the boundary response of a canonical differential system, a first-order system governed by a positive semidefinite Hamiltonian.

The global canopy certificate answers a separate analytic question: combined with the published finite-RH and rectangular time-barrier inputs, it gives $\Lambda\le0.1732$.  The bound uses no QRH assumption and makes no optimality claim.  Its proof requires neither conjecture above.  Conversely, the coexistence of nonresonant cores and nonreal heat zeros at negative times shows that generic action separation alone is not a zero-reality criterion.  Any proposed RH interpretation at $\kappa=0$ must therefore supply additional fixed-point arithmetic structure.

\section*{Supplementary computational files}
The ancillary directory \path{anc/} contains twelve certificate packages, listed in Table~\ref{tab:supplement}.  Directory names in that table are relative to \path{anc/certificates/}; other displayed \path{certificates/} paths are relative to \path{anc/}.  The file \path{anc/COMPUTATIONAL_INDEX.md} gives the exact Python inventory, entry points, dependencies and replay scope.  Each package retains its own instructions and numerical records.

\Needspace{38\baselineskip}
\begingroup\small
\setlength{\tabcolsep}{4pt}
\renewcommand{\arraystretch}{1.12}
\begin{longtable}{@{}>{\raggedright\arraybackslash}p{0.22\textwidth}>{\raggedright\arraybackslash}p{0.42\textwidth}>{\raggedright\arraybackslash}p{0.31\textwidth}@{}}
\caption{Computational packages supplied with the manuscript.}\label{tab:supplement}\\
\toprule Directory & Purpose & Retained evidence\\\midrule
\endfirsthead
\multicolumn{3}{l}{\textit{Table~\thetable\ continued}}\\
\toprule Directory & Purpose & Retained evidence\\\midrule
\endhead
\bottomrule\endfoot
\path{humbert/} & Nine low-discriminant Humbert exclusions. & Eleven principal programs, two support modules, exact $H_{16}$ data and replay logs.\\
\path{heat/} & Branch contraction and heat--$D_6$ tangent separation. & Directed root/velocity bounds; $192/256$-bit logs and receipts.\\
\path{lax_enlarged/} & Enlarged real-moduli tangent rank and matching rates. & $192/256$-bit determinant/rate certificates; uses the heat package.\\
\path{period_heat/} & Calibrated periods and their heat derivatives. & Taylor engine; $512/768$-bit period-containment and velocity checks.\\
\path{finite_heat/} & Uniform core and Farey persistence. & Uniform Taylor engine; $512/768$-bit interval certificates.\\
\path{phase_tail/} & Baseline and centred-moment exterior phase estimates. & $160/256$-bit constants, analytic-tail and monotonicity checks.\\
\path{finite_cap/} & Mollified canopy and global transfer geometry. & $1993$ rectangles; structural and parameter-inclusion checks.\\
\path{canopy_frontier/} & Joint strip-height/time refinement. & $3591$ rectangles, exterior bound, receipt auditor and point diagnostics.\\
\path{canopy_inverse/} & Finite reciprocal-head convolution and transfer. & $2449$ rectangles, integer guards, exterior and index audits.\\
\path{canopy_localized/} & Local placement of reference norms. & Eleven references; $2483$ rectangles, exterior and transfer diagnostics.\\
\path{canopy_fine/} & Finer localization with $J=512$. & $37$ references; $3571$ rectangles, frozen-plan replay and exterior.\\
\path{canopy_1024/} & Extended head giving $\Lambda\le0.1732$. & $59$ references; $3741$ rectangles, $192/256$-bit records and saved-norm audits.\\
\end{longtable}
\endgroup

The eleven programs in the Humbert proof are supplied under \path{certificates/humbert/}, together with the generated files required by its $H_{16}$ chain.  The extended-head package is nested under \path{canopy_1024/head_extension_J1024/}; its README distinguishes full convolution regeneration from audits using saved norms.  Floating point diagnostics are identified separately and are not theorem inputs.

The archive provides the computational coverage listed above; it does not contain every standalone calculation used during development.  In particular, separate static theta-moment, initial Lax-matching, Kummer-model and homology-check scripts are not included as independent replay packages.  Their displayed derivations and bounds remain in the manuscript; the supplied heat and period packages also recertify selected branch and period data, within their documented scope.  The computational index records these limits explicitly.  The published finite-RH and time-barrier computations, and the external QRH release, are cited inputs rather than computations replayed by this supplement.

\section*{Disclosure of AI-assisted preparation}
OpenAI ChatGPT 6.1 and Anthropic Claude Opus 5.5 assisted with mathematical exploration, code generation and review, manuscript organization, and language editing.  The author is responsible for the mathematical claims, references, computational certificates, and final text.  AI assistance and successful computational replay do not constitute independent mathematical refereeing.

\end{document}